\documentclass[10pt]{article}
\usepackage[utf8]{inputenc}

\usepackage{amsmath,color}
\usepackage{amssymb,amsthm}

\usepackage[permil]{overpic}
\usepackage[multiple]{footmisc}
\usepackage{xr}
\usepackage[left=1.6cm,right=1.6cm,top=2.50cm,bottom=2.50cm]{geometry}
\usepackage{graphicx}
\usepackage[font=small,labelfont=bf,
   justification=justified,
   format=plain]{caption}
\usepackage{subfig}
\usepackage{float}
\usepackage[english]{babel}
\usepackage[font=small,labelfont=bf,justification=centering]{caption}
\usepackage[T1]{fontenc}
\usepackage{lastpage}
\usepackage{enumerate}
\usepackage{enumitem}
\usepackage{lmodern}
\usepackage{array}
\usepackage{bm}
\usepackage{multirow}
\usepackage{dsfont}
\usepackage{tensor}
\usepackage{fancyhdr}
\usepackage{listings}
\usepackage{dsfont}
\usepackage{siunitx}
\usepackage{titling}
\usepackage{lipsum}
\usepackage{tabularx}
\usepackage{verbatim}
\usepackage{authblk}
\usepackage{csquotes}
\usepackage[
backend=biber,
style=numeric,
giveninits=true,
sorting=anyt
]{biblatex}

\graphicspath{{Simulations/}}

\def\txtd{{\textnormal{d}}}

\def\txti{{\textnormal{i}}}

\def\txtD{{\textnormal{D}}}

\makeatletter
\newsavebox{\@brx}
\newcommand{\llangle}[1][]{\savebox{\@brx}{\(\m@th{#1\langle}\)}%
  \mathopen{\copy\@brx\kern-0.5\wd\@brx\usebox{\@brx}}}
\newcommand{\rrangle}[1][]{\savebox{\@brx}{\(\m@th{#1\rangle}\)}%
  \mathclose{\copy\@brx\kern-0.5\wd\@brx\usebox{\@brx}}}
\makeatother

\newtheorem{deff}{Definition}[section]
\newtheorem{prop}[deff]{Proposition}
\newtheorem{thm}[deff]{Theorem}
\newtheorem{lm}[deff]{Lemma}
\newtheorem{cor}[deff]{Corollary}

\title{Bifurcations and Early-Warning Signs for SPDEs \\with Spatial Heterogeneity}

\author[*]{P. Bernuzzi}
\author[*,$\dag$]{C. Kuehn}
\affil[*]{\footnotesize{Technical University of Munich, School of Computation Information and Technology, Department of
Mathematics, Boltzmannstraße 3, 85748 Garching, Germany}}
\affil[$\dag$]{\footnotesize{Munich Data Science Institute, Technical University of Munich}}

\date{\today}

\begin{document}

\maketitle{}

\begin{abstract} 
Bistability is a key property of many systems arising in the nonlinear sciences. For example, it appears in many partial differential equations (PDEs). For scalar bistable reaction-diffusions PDEs, the bistable case even has taken on different names within communities such as Allee, Allen-Cahn, Chafee-Infante, Nagumo, Ginzburg-Landau, $\Phi^4$, Schl\"ogl, Stommel, just to name a few structurally similar bistable model names. One key mechanism, how bistability arises under parameter variation is a pitchfork bifurcation. In particular, taking the pitchfork bifurcation normal form for reaction-diffusion PDEs is yet another variant within the family of PDEs mentioned above. More generally, the study of this PDE class considering steady states and stability, related to bifurcations due to a parameter is well-understood for the deterministic case. For the stochastic PDE (SPDE) case, the situation is less well-understood and has been studied recently. In this paper we generalize and unify several recent results for SPDE bifurcations. Our generalisation is motivated directly by applications as we introduce in the equation a spatially heterogeneous term and relax the assumptions on the covariance operator that defines the noise. For this spatially heterogeneous SPDE, we prove a finite-time Lyapunov exponent bifurcation result. Furthermore, we extend the theory of early warning signs in our context and we explain, the role of universal exponents between covariance operator warning signs and the lack of finite-time Lyapunov uniformity. Our results are accompanied and cross-validated by numerical simulations.
\end{abstract}

\pagestyle{fancy}
\fancyhead{}
\renewcommand{\headrulewidth}{0pt}
\fancyhead[C]{\textit{Bifurcations and Early-Warning Signs for SPDEs with Spatial Heterogeneity}}

\section{Introduction}
\label{sec:intro}

Scalar PDEs with bistability have been deeply studied within many communities~\cite{allee1949principles,fitzhugh1955mathematical,stommel1961thermohaline,allen1972ground,CI,Henry81}. An algebraically particularly simple form of this class of PDEs is given by
\begin{equation}
\label{eq:CIintro}
\partial_t u(x,t)=\Delta u(x,t)+\alpha u(x,t)-u(x,t)^3
\end{equation}
where $x\in[0,L]$ for some fixed interval length $L>0$, $t\geq0$ and with a parameter $\alpha\in\mathbb{R}$. One natural option is to view \eqref{eq:CIintro} as an initial-boundary value problem endowed with zero Dirichlet conditions on $[0,L]$ and $u(x,0)=u_0(x)$, given $u_0\in H^1_0$, where $H^1_0=H^1_0([0,L],\mathbb{R})$ denotes the usual Sobolev space with one weak derivative in $L^2([0,L],\mathbb{R})$. The main result presented by Chafee-Infante in~\cite{CI} was to study the number of steady states for~\eqref{eq:CIintro} and their respective stability, depending on the value of $\alpha$. More precisely, let us denote by $\{-\lambda_k'\}_{k\in\mathbb{N}\setminus\{0\}}$ the eigenvalues of $\Delta$, the Laplacian operator with Dirichlet conditions, that take the form $-\lambda_k':=-\bigg( \dfrac{\pi k}{L} \bigg)^2$ for $k\in\mathbb{N}\setminus\{0\}$. It has been shown that for $\alpha<\lambda_1'$ the origin in $H^1_0$, i.e. the state being identically equal to zero, is the only steady state of the system and it is asymptotically stable. Conversely, for $\lambda_k'<\alpha\leq\lambda_{k+1}'$ and $k\in\mathbb{N}\setminus\{0\}$, the origin loses its stability and the number of different stationary solutions becomes $2k+1$. Of those, only two are asymptotically stable while the rest are unstable. Specifically, the only stable solutions are the ones that bifurcate from the origin when $\alpha$ crosses $\lambda_1'$. The study of the steady states of more general systems has been presented in detail in~\cite{Henry81} using a geometric approach. For example, the equation
\begin{equation}
\label{eq:hetPDE}
\partial_t u(x,t)=\Delta u(x,t)-g(x)u(x,t)+\alpha u(x,t)-u(x,t)^3
\end{equation}
for a bounded and continuous almost everywhere positive function $g$ and assumptions taken as in the previous case, has been studied. We will refer to~\eqref{eq:hetPDE} as a heterogeneous case due to the additional spatial heterogeneity induced by $g$. We are going to use certain important properties of the classical spatially homogeneous PDE that are inherited in the heterogeneous case, such as the loss of stability by the origin when $\alpha$ crosses $\lambda_1$, for $\{-\lambda_k\}_{k\in\mathbb{N}\setminus\{0\}}$ the eigenvalues of the Schrödinger operator $\Delta-g$. Moreover, the dimension of the unstable manifold of the origin is $k$ when $\lambda_k<\alpha<\lambda_{k+1}$. The existence of an attractor is satisfied for almost all $\alpha$ values by the dissipativity of system~(\cite{Henry81}, Chapter 5).\medskip

Beyond the deterministic PDE dynamics, our second main ingredient are stochastic dynamics techniques as we want to study an SPDE variant of~\eqref{eq:hetPDE}. We start with an introduction to the background in the stochastic case. It is well known that the definition of a stochastic bifurcation is much more complex than in the deterministic case~\cite{arnold1995random,berglund2006noise,kuehn2011mathematical}. For example, in~\cite{CF} it is shown, how an ODE supercritical pitchfork normal form, when perturbed by noise, yields a unique attractor in the form of a singleton despite the ODE being bistable above the deterministic pitchfork bifurcation point. In particular, it was proven how the Lyapunov exponent along such an attractor is strictly negative. Thus, the solutions under the same noise are expected to synchronize asymptotically following the attractor of the random dynamical system described by the problem. This event is called synchronization by noise. Later on, a similar approach was carried over~\cite{Caraballo} for an SPDE variant of the pitchfork normal form, i.e., for a stochastic variant of~\eqref{eq:CIintro}. For the SODE pitchfork case, it was shown that on finite time scales, the bifurcation effect still persists, first in the fast-slow setting \cite{berglund2002pathwise,berglund2006noise}. Afterwards, synchronization by noise was shown to be not instantaneous for the pitchfork SODE case~\cite{callaway}, in the sense that \textit{finite-time Lyapunov exponents} have a positive probability to be positive above the bifurcation threshold. The recent work~\cite{https://doi.org/10.48550/arxiv.2108.11073} has proven, how this result carries over to SPDEs without spatially heterogeneous terms. One of the objectives of this article is to prove that several results can be extended to a generalized, heterogeneous in space, family of SPDEs.\\

The SPDE we study is a reaction-diffusion equation with additive noise
\begin{equation}
\label{main_syst}
\begin{cases}
\txtd u(x,t)=(\Delta u(x,t) + f(x) u(x,t) -u(x,t)^3)~\txtd t +\sigma~\txtd W_t\\
u(\cdot,0)=u_0\in \mathcal{H}\\
\left.u(\cdot,t)\right|_{\partial\mathcal{O}}=0\;\;\;\;,\;\;\;\;\forall t\in\mathbb{R}
\end{cases}
\end{equation}
with $x\in\mathcal{O}:=[0,L]$, $\mathcal{H}:=\{v\in L^2(\mathcal{O}):v(0)=v(L)=0\}$, $\sigma>0$, $\mathcal{V}:=H^1_0(\mathcal{O})$ and $W_t$ a stochastic forcing to be specified below. We will assume $f(x)=\alpha-g(x)$ to be bounded, with $g$ having H\"older regularity $\gamma>\dfrac{1}{2}$ and $g>0$ almost everywhere in $\mathcal{O}$. The stochastic process $W_t$  will be taken as a two-sided $Q$-Wiener process $W_t$, whose covariance operator $Q$ is trace class and self-adjoint. We assume that the eigenvalues of $Q$ satisfy $q_j>0$, for $j\in\mathbb{N}\setminus\{0\}$, and they are associated with eigenfunctions $b_j$ such that $b_j=e_j'$ for $j>D>0$ and some integer $D>0$, with $e_j'$ normalized eigenfunctions of the Laplacian with zero Dirichlet boundary conditions. This assumption leads to the fact that the $b_j$ are finite combinations of $e_k'$ for $j,k\in\{1,...,D\}$. It can be shown (following the steps in \cite[Chapter 7]{DaPrato}) that the SPDE has a $\mathbb{P}-a.s.$ mild solution 
\begin{equation*}
u\in L^2(\Omega\times(0,T);\mathcal{V})\cap L^2(\Omega;\mathcal{C}((0,T);\mathcal{H}))\;\;.
\end{equation*}
The assumptions on $Q$ imply the existence of a compact attractor for the random dynamical system $(\varphi,\theta)$ defined by \eqref{main_syst} as shown in~\cite{debussche}. Furthermore, the attractor $a$ is a singleton for every $\alpha\in\mathbb{R}$; see~\cite{Caraballo}. Such a singleton attractor result is proven using the order-preservation of the SPDE \cite{chueshov,twardowska}.\medskip

This paper proves the behaviour of finite-time Lyapunov exponents (FTLE) on the attractor of \eqref{main_syst} depending on $\alpha$. For $\alpha<\lambda_1$ the FTLE are $\mathbb{P}-a.s.$ negative and for $\alpha>\lambda_1$ there is positive probability for them to be positive. Such a result can also be expanded, following \cite{https://doi.org/10.48550/arxiv.2108.11073}, for the cases where $\alpha$ crosses other eigenvalues of $\Delta-g$. In addition to FTLE, there is another natural way to study small stochastic perturbations near deterministic bifurcations, which is based upon moments. For stochastic ODEs~\cite{berglund2006noise,kuehn2011mathematical,berglund2012hunting,kuehn2013mathematical,berglund2015random}, it is well understood that critical slowing down near a bifurcation induces growth of the covariance. This growth can then be used in applications as an early-warning sign. For our SPDE, we study early warning signs associated to the covariance operator for the solution of the linearized system following the results of \cite{early,kuehn2019scaling}. In particular, we are going to prove a hyperbolic-function divergence of the covariance operator, when approaching a bifurcation, and of the variance in time of the solution $u$ on every point in $\mathring{\mathcal{O}}$. The precise rate of the phenomenon can be described depending on the studied spatial point. We then discuss the estimate of the error due to the linearization of the system and we cross-validate our theoretical results by numerical simulations.

\subsection{Hypothesis and assumptions}

In this paper, the object of study is the following parabolic SPDE with zero Dirichlet and given initial conditions:
\begin{equation}
\label{mainsyst}
\begin{cases}
\txtd u(x,t)=(\Delta u(x,t) + f(x) u(x,t) -u(x,t)^3)~\txtd t +\sigma~\txtd W_t\\
u(\cdot,0)=u_0\in \mathcal{H}\\
\left.u(\cdot,t)\right|_{\partial\mathcal{O}}=0\;\;\;\;,\;\;\;\;\forall t\in\mathbb{R}
\end{cases}\end{equation}
with $x\in\mathcal{O}:=[0,L]$, $\mathcal{H}:=\{v\in L^2(\mathcal{O}):v(0)=v(L)=0\}$, and $\mathcal{V}:=H^1_0(\mathcal{O})$. The stochastic process $W_t$ is a two-sided $Q$-Wiener process in $\mathcal{H}$ with covariance operator $Q$ that is assumed trace class, self-adjoint and positive. The function $f$ will be taken as $f=-g+\alpha$ with $g$ almost everywhere positive and uniformly H\"older regular of order $\dfrac{1}{2}<\gamma\leq1$, in particular it could be uniformly Lipschitz in $[0,L]$, and $\alpha\in\mathbb{R}$. The first equation in \eqref{mainsyst} can written as
\begin{equation}
\label{alternative}
\txtd u(x,t)=(A u(x,t)+\alpha u(x,t) -u(x,t)^3)~\txtd t +\sigma~\txtd W_t\;\;,
\end{equation}
with $A:=\Delta-g$. Therefore, we observe that A is a Schr\"odinger operator whose main properties are collected in Appendix A. In particular, the eigenvalues of $A:=\Delta-g$, $\{-\lambda_k\}_{k\in\mathbb{N}\setminus\{0\}}$, are strictly negative. As described in the introduction above, the eigenvalues and eigenfunctions of $\Delta$ are $-\lambda_k':=-\Big(\dfrac{\pi k}{L}\Big)^2$ and $e_k':=\sqrt{\dfrac{2}{L}}\text{sin}\Big(\dfrac{\pi k}{L} x\Big)$ with $k\in\mathbb{N}\setminus\{0\}$ and $x\in\mathcal{O}$. In Appendix A, it is illustrated how $\{-\lambda_k\}_{k\in\mathbb{N}\setminus\{0\}}$ and the corresponding normalized in $\mathcal{H}$ eigenfunctions of $A$, $\{e_k\}_{k\in\mathbb{N}\setminus\{0\}}$, behave asymptotically in $k$ in comparison to the eigenvalues $\{-\lambda_k'\}_{k\in\mathbb{N}\setminus\{0\}}$ and eigenfunctions $\{e_k'\}_{k\in\mathbb{N}\setminus\{0\}}$ respectively. Both $\Delta$ and $A$ are assumed to be closed operators on $\mathcal{H}$.\\ \medskip
The norm on a Banach space $X$ is denoted as $\lvert\lvert\cdot\rvert\rvert_X$ with the exception of the $L^p$ spaces for which it is noted as $\lvert\lvert\cdot\rvert\rvert_p$ for any $1\leq p\leq \infty$ or for further cases later described in detail. We use $\langle\cdot,\cdot\rangle$ to indicate the standard scalar product in $L^2(\mathcal{O})$. We use the Landau notation on sequences $\{\rho^1_k\}_{k\in\mathbb{N}\setminus\{0\}}\subset\mathbb{R},\;\{\rho^2_k\}_{k\in\mathbb{N}\setminus\{0\}}\subset\mathbb{R}_{>0}$ as $\rho^1_k=O(\rho^2_k)$ if  $0\leq\underset{k\rightarrow\infty}{\lim}\dfrac{\lvert\rho^1_k\rvert}{\rho^2_k}<\infty$.  We define the symbol $\sim$ to indicate $\rho_1\sim \rho_2$ if $0<\underset{\rho_2\rightarrow\infty}{\lim}\dfrac{\rho_1}{\rho_2}<\infty$ and we apply it also on sequences $\{\rho^1_k\}_{k\in\mathbb{N}\setminus\{0\}}\subset\mathbb{R},\;\{\rho^2_k\}_{k\in\mathbb{N}\setminus\{0\}}\subset\mathbb{R}_{>0}$ meaning that $0<\underset{k\rightarrow\infty}{\lim}\dfrac{\rho^1_k}{\rho^2_k}<\infty$.\medskip

Next, we proceed to specify the assumptions on the noise. The probability space that we will use is $(\Omega,\mathcal{F},\mathbb{P})$ with $\Omega:=\mathcal{C}_0(\mathbb{R};\mathcal{H})$ composed of the functions $\omega:\mathbb{R}\rightarrow \mathcal{H}$ such that $\omega(0)=0$ and endowed with the compact-open topology. With $\mathcal{F}$ we denote the Borel sigma-algebra on $\Omega$ and with $\mathbb{P}$ a Wiener measure generated by a two-sided Wiener process. As in \cite{https://doi.org/10.48550/arxiv.2108.11073}, we consider $\omega_t=W_t(\omega)$ and the two-sided filtration $(\mathcal{F}_{t_1}^{t_2})_{t_1<t_2}$ defined as $\mathcal{F}_{t_1}^{t_2}:=\mathfrak{B}(\omega_{t_2}-\omega_{t_1})$. We can then obtain
\begin{equation*} 
\begin{split}
\mathcal{F}_{-\infty}^t&=\mathfrak{B}(\mathcal{F}_{t_1}^t: t_1<t)\\
\mathcal{F}_t^\infty&=\mathfrak{B}(\mathcal{F}_t^{t_1}: t<t_1) \;.
\end{split} 
\end{equation*}
We also introduce the Wiener shift for all $t>0$
\begin{equation*} (\theta^t \omega)_{t_1}:=\omega_{t+t_1}-\omega_{t_1},\;\;\text{for any}\;\; t_1\in\mathbb{R},\omega\in\Omega\;.\end{equation*}
Thus $(\theta^t)_{t\in\mathbb{R}}$ is a family of $\mathbb{P}$-preserving transformations that satisfies the flow property and $(\Omega,\mathcal{F},\mathbb{P},(\theta^t)_{t\in\mathbb{R}})$ is an ergodic metric dynamical system~\cite{chueshov2001inertial}. The covariance operator $Q$ is assumed to have eigenvalues $\{q_k\}_{k\in\mathbb{N}\setminus\{0\}}$ and eigenfunctions $\{b_k\}_{k\in\mathbb{N}\setminus\{0\}}$ with the following properties:
\begin{enumerate}
\item \label{firstproperty} there exists a $D\in\mathbb{N}\setminus\{0\}$ such that $b_j=e_j'$ for all $j>D$,
\item \label{secondproperty} $q_j>0$ for all $j\in \mathbb{N}\setminus\{0\}$,
\item \label{thirdproperty} $\sum_{j=1}^\infty q_j \lambda_j'^\gamma <+\infty$ for a $\gamma>0$.
\end{enumerate}
The first property implies that $b_k$ is a finite combination of $\{e_j'\}_{j\in\{1,...,D\}}$ for $k\leq D$. The choice of $D>2$ can be used to have more freedom on the choice of $Q$ both within analytical and numerical results. As stated in Appendix B, the first and third assumption imply the continuity of the solutions of \eqref{mainsyst} in $\mathcal{V}$ if $u_0\in \mathcal{V}$ and that, for $u_0\in \mathcal{H}$, there exists $\mathbb{P}-a.s.$ a unique mild solution
\begin{equation*}u\in L^2(\Omega\times(0,T);\mathcal{V})\cap L^2(\Omega;\mathcal{C}((0,T);\mathcal{H}))\;\;.\end{equation*}

\subsection{Properties of the determistic case}

The heterogeneous system~\eqref{mainsyst} with $\sigma=0$ inherits some properties from the classical spatially homogeneous PDE:

\begin{prop}
\label{stab}
Set $u_*\in \mathcal{H}$ and let $\mathcal{D}(A)$ be the domain of $A$. Suppose $A:\mathcal{D}(A)\subset \mathcal{H}\longrightarrow \mathcal{H}$ is a sectorial linear operator and $F:U\rightarrow \mathcal{H}$ is a locally Lipschitz operator so that $U$ is a neighbourhood of $u_*$ and we have
\begin{equation*}
F(u_*+v)=F(u_*)+B v + G(v)  
\end{equation*}
with $u_*+v\in U$, $B\in\mathcal{L}_b(\mathcal{H})$ with $\mathcal{L}_b(\mathcal{H})$ denoting bounded linear operators on $\mathcal{H}$, $G:U\rightarrow \mathcal{H}$ and $\| G(u)\|\leq c \| u\|_\mathcal{V}^{\gamma} $ for $\gamma>1$, any $u\in \mathcal{V}$ and a constant $c>0$.\footnote{We use $c$ and $C$ to denote constants that are independent from other variables and parameters, unless clearly stated otherwise.} Assume $0=Au_*+F(u_*)$ and that the spectrum of $A+B$ is in $\{\lambda\in\mathbb{C}:\text{Re}(\lambda)<-\epsilon\}$ for a $\epsilon>0$, then $u_*$ is a locally asymptotically stable steady state of the evolution equation $\dot{u}=Au+F(u)$, where $u(t)=u(\cdot,t)$ and dot denotes the time derivative.
\end{prop}

A proof of this standard proposition can be found in \cite[Theorem 5.1.1]{Henry81}. By Proposition \ref{stab} it is clear that the zero function is an asymptotic stable steady state of~\eqref{mainsyst} with $\sigma=0$ when $\alpha<\lambda_1$. In fact, for this parameter range it is also the only steady state by the strict negativity of $A$ and the strict dissipativity given by $F$. When $\alpha>\lambda_1$ the stability setting changes. The next proposition describes such a case; see \cite[Corollary 5.1.6]{Henry81}.

\begin{prop}
\label{unstab}
Set $u_*$, $A$ and $F$ as in Proposition \ref{stab}. Assume again $0=Au_*+F(u_*)$ and that the spectrum of $A+B$ has non-empty intersection with $\{\lambda\in\mathbb{C}:\text{Re}(\lambda)>\epsilon\}$ for an $\epsilon>0$, then $u_*$ is an unstable steady state of $\dot{u}=Au+F(u)$.
\end{prop}

The last result can be applied to show that the zero function is unstable for $\alpha>\lambda_1$. Nonetheless, the global dissipativity still implies the existence of a compact attractor~\cite[Section 5.3]{Henry81}. The next theorem, based on \cite{crandall1971bifurcation}, establishes the bifurcation of near steady states from the trivial branch of zero solutions when $\alpha$ crosses a bifurcation point.

\begin{thm}[Crandall-Rabinowitz]
Set $X$ and $Y$ Banach spaces, $u_*\in X$ and $U$ neighbourhood of $(u_*,0)$ in $X\times\mathbb{R}$. Let $F:U\subset X\times\mathbb{R}\rightarrow Y$ be a $\mathcal{C}^3$ function in $(u_*,0)$ with $F(u_*,p)=0$ for all $p\in(-\delta,\delta)$ for $\delta>0$. Suppose the linearization $G:=\txtD_X F(0,0)$ is a Fredholm operator with index zero. Assume that $\dim(\textnormal{Ker}(G))=1$, $\textnormal{Ker}(G)=\textnormal{Span}\{\phi_0\}$ and $\txtD_p \txtD_X F(u_*,0)\phi_0\not\in\textnormal{Range}(G)$. Then $(u_*,0)$ is a bifurcation point and there exists a $\mathcal{C}^1$ curve $s\mapsto (\phi(s),p(s))$, for $s$ in a small interval, that passes through $(u_*,0)$ so that
\begin{equation}\label{cr}
    F(\phi(s),p(s))=0\;.
\end{equation}
In a sufficiently small neighbourhood $U$ the only solutions to $\eqref{cr}$ are the map $s\mapsto(\phi(s),p(s))$ and the trivial curve $\{(u_*,\alpha):(u_*,\alpha)\in U\}$.
\end{thm}

A proof of the last result can be found in \cite[Lemma 6.3.1, Theorem 6.3.2]{Henry81} and \cite[Theorem 5.1]{kuehn2019pde}. For $Y=\mathcal{H}$, the space $X$ generated by $\{e_k\}_{k\in\mathbb{N}\setminus\{0\}}$, the normalized eigenfunctions of $A$, and given the nonlinear operator 
\begin{equation*}F(u,\lambda_k-\alpha)=Au+\alpha u -u^3\;,\end{equation*} 
then the Crandall-Rabinowitz theorem implies the existence of a bifurcation, which can be shown to be a pitchfork, at $\alpha=\lambda_k$ and in $u_*$ taken as the origin function in $\mathcal{V}$. It can also be proven that the new steady states that arise at $k=1$ are locally asymptotically stable.

\subsection{Random dynamical system and attractor}

The next properties are described in \cite{caraballo2000stability} and we summarize them here as we will need them later on.

\begin{prop}
\label{g1}
There exists a $\theta^t$-invariant subset $\Omega'\subset\Omega$ of full probability measure such that for all $\omega\in\Omega'$ and $t\geq0$ there is a Fréchet differentiable $\mathcal{C}^1$-semiflow on $\mathcal{H}$
\begin{equation*} 
u_0\longmapsto\varphi(t,\omega,u_0)=:\varphi_\omega^t(u_0)
\end{equation*}
that safisfies the following for any $\omega\in\Omega'$:
\begin{enumerate}
\item[(P1)] For all $T>0$ and $u_0\in \mathcal{H}$, the mapping $(\omega,t)\mapsto \varphi_\omega^t(u_0)$ in $\Omega\times[0,T]\mapsto \mathcal{H}$ is the unique pathwise mild solution of $\eqref{mainsyst}$.
\item[(P2)] It satisfies the cocycle property, i.e., for any $t_1,t_2>0$
\begin{equation*}
\varphi_\omega^{t_1+t_2}=\varphi_{\theta^{t_1} \omega}^{t_2} \circ \varphi_\omega^{t_1}\;.
\end{equation*}
\item[(P3)] Fixing $t_2>t_1>0$, for any $u\in \mathcal{H}$ the $\mathcal{H}$-valued random variable $\varphi_{\theta^{t_1} \omega} ^{t_2-t_1}(u)$ is $\mathcal{F}_{t_1}^{t_2}$-measurable.
\end{enumerate}
\end{prop}

The pair of mappings $(\theta,\varphi)$ is the random dynamical system (RDS) generated by $\eqref{mainsyst}$. Also, property (P3) implies that the process $u_t$, defined by $u_t:=\varphi_\omega^t(u_0)$ for generic $\omega\in\Omega'$, is $\mathcal{F}_0^t$-adapted. For any $\mathcal{C}^1$ Fréchet differentiable mapping $\varphi:\mathcal{H}\rightarrow \mathcal{H}$, we denote by $\txtD_u\varphi\in \mathcal{L}(\mathcal{H})$ the Fréchet derivative of $\varphi$ in $u\in \mathcal{H}$, for $\mathcal{L}(\mathcal{H})$ that indicates the space of linear operators on $\mathcal{H}$. The next proposition, regarding the Fréchet derivative, is obtained by \cite[Lemma 4.4]{debussche}. 

\begin{prop}
\label{g2}
For all $T>0$ and $u_0,v_0\in \mathcal{H}$ and given the Fréchet differentiable $\mathcal{C}^1$ semiflow described in Proposition \ref{g1}, we have that the mapping $\Omega\times[0,T]\mapsto \mathcal{H}$ given by $(\omega,t)\mapsto \txtD_{u_0} \varphi_\omega^t(v_0)$ is the unique solution of the first variation evolution equation along $u$ with initial conditions set on $u_0$ and $v_0$:
\begin{equation}\label{vareq} 
\begin{cases}\txtd v=(\Delta v + f v - 3 u^2 v)~\txtd t \\
u(\cdot,0)=u_0\\
v(\cdot,0)=v_0\;. \end{cases}
\end{equation}
\end{prop}

Next, we present the relevant result for the existence and uniqueness of an invariant measure. The next lemma and proposition \cite[Chapter 8]{cerrai} are important for following results and are implied by Lemma \ref{contlm} and the strict positivity of the eigenvalues of $Q$.

\begin{lm}
\label{implm}
The transition semigroup of the system $\eqref{mainsyst}$ satisfies the strong Feller property and irreducibility.
\end{lm}

By Krylov-Bogoliubov Theorem, for the existence, Doob's Theorem and by Khas'minskii's Theorem~\cite{khas1960ergodic,da1996ergodicity} we obtain the next proposition.

\begin{prop}
\label{impprop}
There exists a unique stationary measure $\mu$ for the process $u_t$ which is equivalent to any transition probability measure $P_t(u_0,\cdot)$ for all $t>0$ and $u_0\in \mathcal{H}$, in the sense that they are mutually absolutely continuous. It is defined as
\begin{equation*}
P_t(u_0,\mathcal{B}):=\mathbb{P}(\varphi_\omega^t(u_0)\in \mathcal{B}) \;\;\; 
\text{for}\; u_0\in \mathcal{H}\;,\; t>0\;,
\end{equation*}
for any $\mathcal{B}\in \mathfrak{B}(\mathcal{H})$, the sigma-algebra of Borelian sets in $\mathcal{H}$, and it is concentrated on $\mathcal{V}$.
\end{prop}

Finally, we can consider the random attractor itself.

\begin{deff}
A \textbf{global random attractor} of an RDS $(\varphi,\theta)$ is a compact random set $\mathcal{A}\subset \mathcal{H}$, i.e., it depends on $\omega\in\Omega'$ and satisfies:
\begin{itemize}
\item it is invariant under the RDS, which means $\varphi_\omega^t(\mathcal{A}(\omega))=\mathcal{A}(\theta^t \omega)$;
\item it is attracting so that for every bounded set $\mathcal{B}\subset \mathcal{H}$ 
\begin{equation*}\lim_{t\rightarrow\infty} \| \varphi^t_{\theta^{-t}\omega}(\mathcal{B})-\mathcal{A}(\omega) \|_\mathcal{H} =0
\end{equation*}
for $\mathbb{P}$-a.s.~in $\omega$.
\end{itemize}
\end{deff}

We can use Proposition \ref{impprop} to prove a strong property of the random attractor.

\begin{prop}
For any $\alpha\in\mathbb{R}$ the random dynamical system generated by $\eqref{mainsyst}$ has a global random attractor $\mathcal{A}$ that is a singleton, i.e., there exists a random variable $a:\Omega\rightarrow \mathcal{H}$ such that $\mathcal{A}(\omega)=a(\omega)$ $\mathbb{P}$-a.s., and it is $\mathcal{F}_{-\infty}^0$-measurable. The law of the attractor $a$ is the unique invariant measure of the system.
\end{prop}

The proof of this proposition follows the steps of \cite[Theorem 6.1]{Caraballo}. In particular, the existence of the global random attractor $\mathcal{A}$ is the result of \cite[Theorem 4.5]{debussche} and its $\mathcal{F}_{-\infty}^0$-measurability can be derived from \cite[Theorem 3.3]{crauel2000white}. Another important property is the order-preservation of the system~\eqref{mainsyst}, i.e., for two initial conditions $u_1(x,0)\leq u_2(x,0)$ for almost all $x\in\mathcal{O}$ the solutions satisfy $u_1(x,t)\leq u_2(x,t)$ for almost all $x\in\mathcal{O}$ and $t>0$. This is a consequence of \cite[Theorem 5.1]{twardowska} and \cite[Theorem 5.8]{chueshov}.\medskip

The main result of the paper \cite{https://doi.org/10.48550/arxiv.2108.11073} is the description of the possible influences of the attractor $a$ on close solutions, for different values of $\alpha$. This can be expressed by the sign of the leading finite-time Lyapunov exponent (FTLE).

\begin{deff}
The (leading) \textbf{finite-time Lyapunov exponents}, at a time $t>0$ and sample $\omega\in\Omega$, of an RDS $(\varphi,\theta)$ with Fréchet differentiable semiflow $\varphi$ is defined by the following equation:
\begin{equation*} 
\mathfrak{L}_1(t;\omega,u):=\dfrac{1}{t}\log \| \txtD_{u} \varphi^t_\omega \|_\mathcal{H}
\end{equation*}
for any $u\in \mathcal{H}$.
\end{deff}

The FTLE describes the influence of the linear operator $\txtD_{u} \varphi^t_\omega$, for given $\omega\in\Omega$, $u\in \mathcal{H}$ and $t>0$, on elements of $\mathcal{H}$ close to $u$. A positive FTLE indicates that close functions in $\mathcal{H}$, that are near $u$, tend to separate in time $t$. Conversely, a negative FTLE indicates the distance between elements in a neighbourhood of $u$ tends to be smaller after time $t$. Of great interest is the FTLE on the attractor, defined as 
\begin{equation*} 
\mathfrak{L}_1(t;\omega):=\mathfrak{L}_1(t;\omega,a(\omega)).
\end{equation*}
By the subadditive ergodic theorem there exists a limit $t\rightarrow \infty$ of $\mathfrak{L}_1(t;\omega)$ for $\omega-\mathbb{P}$-a.s., referred to as the Lyapunov exponent of $a(\omega)$.\footnote{From the synchronization of the system it would be expected that the Lyapunov exponent is non-positive with probability $1$. This is still an open problem.}\medskip

Next, we comment on extension to further bifurcations beyond the first pitchfork bifurcation point. As previously described, any $\alpha\in\{\lambda_j\}_{j\in\mathbb{N}\setminus\{0\}}$ is a deterministic bifurcation threshold of the system for which, when crossed, two new steady states appear and the origin in $\mathcal{H}$ increases dimension of its unstable manifold by one. To study this case we make use of wedge products. For a separable Hilbert space $\mathcal{H}$ and for $v_1,...,v_k\in \mathcal{H}$, the wedge product ($k$-blade) is denoted by
\begin{equation*} 
v_1\wedge...\wedge v_k.
\end{equation*}
We define the following scalar product for all $v_1,...,v_k,w_1,...,w_k\in \mathcal{H}$,
\begin{equation} 
\label{other scal}
(v_1\wedge...\wedge v_k, w_1\wedge...\wedge w_k):=\det[(v_{j_1},w_{j_2})_{j_1,j_2}]
\end{equation}
with $(v_{j_1},w_{j_2})_{j_1,j_2}$ denoting the $k\times k$ matrix with $j_1,j_2$-th element $(v_{j_1},w_{j_2})$. The set $\wedge^k \mathcal{H}$ is the closure of finite linear combinations of $k$-blades under the norm defined by such inner product. Given $\{e_j\}_{j\in\mathbb{N}\setminus\{0\}}$ as a basis of $\mathcal{H}$, we can define a basis for $\wedge^k \mathcal{H}$ whose elements are the $k$-blades
\begin{equation*} 
\textbf{e}_{\textbf{i}}:=e_{i_1}\wedge...\wedge e_{i_k} 
\end{equation*}
with $\textbf{i}=(i_1,...,i_k)$ for $0<i_1<...<i_k$. Given $B\in \mathcal{L}(\mathcal{H})$ we can obtain an operator in $\mathcal{L}(\wedge^k \mathcal{H})$ by the following operation:
\begin{equation*} 
\wedge^k B(v_1\wedge...\wedge v_k):=Bv_1\wedge...\wedge Bv_k,
\end{equation*}
for all $v_1,...,v_k\in \mathcal{H}$. Such operators satisfy the property\footnote{The operator $\left. B\right |_E:E\rightarrow B(E)$ is assumed to be a linear operator on a finite-dimensional inner product space and we set $\det(\left. B\right |_E)=0$ for $E$ such that $\text{dim}(E)<k$.}
\begin{equation*} 
\| \wedge^k B \|_{\wedge^k \mathcal{H}}=\max\{\lvert \det(\left. B\right |_E)\lvert\;:\; E\subset \mathcal{H},\; \text{dim}(E)=k\}\;. 
\end{equation*}
The last definitions permit us to introduce
\begin{equation*}
\begin{split} &\mathfrak{L}_k(t;\omega,u):=\dfrac{1}{t}\log \| \wedge^k \txtD_u \varphi_\omega^t \|_{\wedge^k \mathcal{H}} \\
& \mathfrak{L}_k(t;\omega):=\mathfrak{L}_k(t;\omega,a(\omega)). 
\end{split}\end{equation*}
Such functions describe the behaviour of volumes defined by the position of elements of $\mathcal{H}$ in a neighbourhood of $u\in \mathcal{H}$ or of the attractor. 



\section{Bounds for FTLEs}

The proofs of the theorems considered in this section build upon the approach described in \cite{https://doi.org/10.48550/arxiv.2108.11073}. The main results are Theorems \ref{tm1} and \ref{tm2}. Theorem \ref{tm1} gives an upper bound to $\mathfrak{L}_k$ for any value of $\alpha$. It follows from this result that for $k=1$ and $\alpha<\lambda_1$ the FTLE is almost surely negative. Theorem \ref{tm2} provides a lower bound to the highest admissible value that $\mathfrak{L}_k$ can assume. Specifically, it shows that there is a positive probability of $\mathfrak{L}_k$ to be positive for $\alpha$ beyond the $k$-th bifurcation. 


\subsection{Upper bound}
\begin{thm}\label{tm1}
For any $k\geq 1$,
\begin{equation*}\mathfrak{L}_k(t;\omega)\leq\sum_{j=1}^k\big(\alpha-\lambda_j \big)\end{equation*}
with probability $1$ for all  $t>0$.
\end{thm}
\begin{proof}
Given $\textbf{v}_t=v^1_t\wedge...\wedge v^k_t=\wedge^k \txtD_{a(\omega)} \varphi_\omega^t(\textbf{v}_0)$, with $\textbf{v}_0\in \wedge^k \mathcal{H}$, we obtain by \cite[Lemma 3.4]{https://doi.org/10.48550/arxiv.2108.11073} and the min-max principle that
\begin{equation*}\begin{split}\dfrac{1}{2}\dfrac{\textnormal{d}}{\textnormal{d}t}\lvert\lvert\textbf{v}_t\rvert\rvert_{\wedge^k \mathcal{H}}^2&=\sum_{j=1}^k (\textbf{v}_t,v^1_t\wedge...\wedge v^{j-1}_t\wedge (f+\Delta+B_\omega^t) v^j_t \wedge v^{j+1}_t \wedge ... \wedge v^k_t)\\
&\leq\sum_{j=1}^k (\textbf{v}_t,v^1_t\wedge...\wedge v^{j-1}_t\wedge (f+\Delta) v^j_t \wedge v^{j+1}_t \wedge ... \wedge v^k_t)\\
&=\dfrac{\txtd}{\txtd r}\bigg|_{r=0} (\textbf{v}_t,\wedge^k e^{(f+\Delta)r} \textbf{v}_t)
\leq\dfrac{\txtd}{\txtd r}\bigg|_{r=0} e^{\sum_{j=1}^k(\alpha-\lambda_j) r}\lvert\lvert\textbf{v}_t\rvert\rvert_{\wedge^k \mathcal{H}}^2
=\sum_{j=1}^k\big(\alpha-\lambda_j\big)\lvert\lvert \textbf{v}_t\rvert\rvert_{\wedge^k \mathcal{H}}^2
\end{split}\end{equation*}
for $B^t_\omega:=-3 a (\theta^t \omega)^2$.
\end{proof}

\subsection{Lower bound}
\begin{lm}\label{Lemma1}
With probability $1$, $a(\theta^t\omega)\in \mathcal{V}$ for all $\alpha,t\in \mathbb{R}$. Also, for any $T,\epsilon>0$, there exists an $\mathcal{F}^T_{-\infty}$-measurable set $\Omega_0\subset\Omega$ such that $\mathbb{P}(\Omega_0)>0$ and 
\[\lvert\lvert a(\theta^t\omega)\rvert\rvert_\mathcal{V}\in(0,\epsilon)\;\;\; \text{for all}\;\alpha\in\mathbb{R},\;t\in[0,T]\;\text{and}\;\omega\in\Omega_0.\]
\end{lm}

\begin{proof}
We begin the proof assuming for the initial condition $u_0\in \mathcal{V}$. From Proposition \ref{impprop} and the continuity of the solution in $\mathcal{V}$ we get for the singleton attractor that $a(\omega)\in \mathcal{V}$ for all $\omega\in\Omega$. Additionally, from Lemma \ref{implm} and Proposition \ref{impprop} we know that the invariant measure that describes the law of the attractor is locally positive on $\mathcal{V}$. As a result it follows that $\lvert\lvert a(\omega)\rvert\rvert_\mathcal{V}\in(0,\eta)$, for all $\omega\in\Omega_1$ with $\Omega_1$ that is $\mathcal{F}_{-\infty}^0$-measurable (\cite[Proposition 3.1]{chueshov2001inertial}) and $\mathbb{P}(\Omega_1)>0$ and for a constant $\eta>0$ dependent on $\Omega_1$. We now define the family of operators $S(t):=\textnormal{e}^{(A+\alpha)t}$ in order to study the solutions of
\begin{equation*} 
\txtd u=(\Delta u - g u+\alpha u - u^3) ~\txtd t+\sigma~\txtd W_t \;.\end{equation*}
From this SPDE we subtract the Orstein-Uhlenbeck process which is solution of
\begin{equation*} \begin{cases}
\txtd z=(\Delta z - g z+\alpha z) ~\txtd t+ \sigma ~\txtd W_t \;,\\
z(x,0)=0\;\;\;\;\;,\;\;\;\;\forall x\in\mathcal{O},
\end{cases}\end{equation*}
with zero Dirichlet boundary conditions and takes the form
\begin{equation*} z(t)=\sigma\int_0^t S(t-t_1)\;\txtd W_{t_1} \;,\end{equation*}
to introduce $\tilde{u}:=u-z$.
Hence, we obtain the random PDE
\begin{equation*}\txtd\tilde{u}=(\Delta \tilde{u}- g \tilde{u}+\alpha \tilde{u}) ~\txtd t - (\tilde{u}+z)^3~ \txtd t\;, \end{equation*}
whose mild solution is
\begin{equation}\label{mild1}\tilde{u}(t)=S(t)\tilde{u}_0+\int_0^t S(t-t_1)\;\iota(\tilde{u}(t_1)+z(t_1)) \txtd t_1\;\end{equation}
with $\tilde{u}_0:=\tilde{u}(0)=u_0$ and $\iota(x)=-x^3$ for all $x\in\mathcal{O}$. From \cite{agmon1962eigenfunctions,da2004functional} we obtain that $S(t)$ is an analytic semigroup. The norms $\lvert\lvert\cdot\rvert\rvert_\mathcal{V}=\lvert\lvert(-\Delta)^{\frac{1}{2}}\cdot\rvert\rvert_\mathcal{H}$ and $\lvert\lvert\cdot\rvert\rvert_A=\lvert\lvert (-A)^{\frac{1}{2}}\cdot\rvert\rvert_\mathcal{H}$ are equivalent in $\mathcal{V}$ as proven in Appendix A. Therefore we obtain the following inequalities from \cite{Henry81}:
\begin{equation}\begin{split}
\lvert\lvert S(t)\rvert\rvert_{\mathcal{L}(\mathcal{V})}&\leq \textnormal{e}^{(-\lambda_1+\alpha)t}, \;\text{for all}\;t>0\label{est1}\;,\\
\lvert\lvert S(t)\rvert\rvert_{\mathcal{L}(\mathcal{H},\mathcal{V})}&\leq c t^{-\frac{1}{2}} \textnormal{e}^{(-\lambda_1+\alpha)t}, \;\text{for all}\;t>0
\end{split}\end{equation}
for a certain constant $c>0$. Another important estimate is obtained by the fact that the nonlinear term $\iota:\mathcal{V}\longrightarrow \mathcal{H}$ is locally Lipschitz and thus for any $u_1,u_2\in U\subset \mathcal{V}$ there exists a constant $\ell>0$ such that
\begin{equation*}
\lvert\lvert \iota(u_1)-\iota(u_2)\rvert\rvert_\mathcal{H}\leq \ell \lvert\lvert u_1-u_2\rvert\rvert_\mathcal{V} 
\end{equation*}
with $U$ being a bounded subset of $\mathcal{V}$. Using a cut-off technique, we can truncate $\iota$ outside of a ball in $\mathcal{V}$ of radius $R>0$ and center in the null function and obtain the globally Lipschitz function
\begin{equation*} \tilde\iota(u):=-\Theta\bigg(\dfrac{\lvert\lvert u \rvert\rvert_\mathcal{V}}{R}\bigg) u^3\;,\end{equation*}
with $\Theta:\mathbb{R}^+\longrightarrow [0,1]$, a $\mathcal{C}^1$ cut-off function. As for $\iota$, the Lipschitz inequality takes the form, for a certain $\tilde\ell>0$,
\begin{equation*}
\lvert\lvert \tilde\iota(u_1)-\tilde\iota(u_2)\rvert\rvert_\mathcal{H}\leq \tilde\ell \lvert\lvert u_1-u_2\rvert\rvert_\mathcal{V}\;\text{for all}\;u_1,u_2\in \mathcal{V} \;. 
\end{equation*}
From \eqref{mild1}, the estimates \eqref{est1} and the fact that $\tilde\iota(u)=\iota(u)$ on the ball with center in the null function and radius $R$ in $\mathcal{V}$, we can obtain the following inequality
\begin{equation}\begin{split} \label{stepa}
\lvert\lvert \tilde{u}(t) \rvert\rvert_\mathcal{V} &\leq \lvert\lvert S(t)\tilde{u}_0\rvert\rvert_\mathcal{V}+\int_0^t \lvert\lvert S(t-t_1) \rvert\rvert_{\mathcal{L}(\mathcal{H},\mathcal{V})}\lvert\lvert \tilde\iota(\tilde{u}(t_1)+z(t_1))\rvert\rvert_\mathcal{H} \txtd t_1\\
&\leq \textnormal{e}^{(-\lambda_1+\alpha)t}\lvert\lvert \tilde{u}_0\rvert\rvert_\mathcal{V} + c \int_0^t \textnormal{e}^{(-\lambda_1+\alpha)(t-t_1)}(t-t_1)^{-\frac{1}{2}}(\tilde\ell \lvert\lvert \tilde{u}(t_1)+z(t_1)\rvert\rvert_\mathcal{V}) \txtd t_1\;,
\end{split}\end{equation}
for which we have assumed $R$ large enough to have $\underset{t\in[0,T]}{\sup} \lvert\lvert \tilde{u}(t)+z(t) \rvert\rvert_\mathcal{V}<R$.
By \cite[Proposition 3.1]{chueshov2001inertial} we can consider the set
\begin{equation*}\Omega_2:=\Big\{\omega\in\Omega : \underset{t\in[0,T]}{\sup}\lvert\lvert z(t) \rvert\rvert_\mathcal{V}\leq\eta\Big\}\in\mathcal{F}_0^T\;,\end{equation*}
which has positive probability. Since $\Omega_1\in\mathcal{F}_{-\infty}^0$ and $\Omega_2\in\mathcal{F}_0^T$, they are independent and $\Omega_0:=\Omega_1\cap\Omega_2\in\mathcal{F}_{-\infty}^T$ has also positive probability. We therefore fix $\omega\in\Omega_0$ and derive from \eqref{stepa}
\begin{equation*}
\lvert\lvert \tilde{u}(t) \rvert\rvert_\mathcal{V} \leq \textnormal{e}^{(-\lambda_1+\alpha)t}\lvert\lvert \tilde{u}_0\rvert\rvert_\mathcal{V} +\tilde\ell \eta c \int_0^t \textnormal{e}^{(-\lambda_1+\alpha)(t-t_1)}(t-t_1)^{-\frac{1}{2}} \txtd t_1 + \tilde\ell c \int_0^t \textnormal{e}^{(-\lambda_1+\alpha)(t-t_1)}(t-t_1)^{-\frac{1}{2}} \lvert\lvert \tilde{u}(t_1)\rvert\rvert_\mathcal{V} \txtd t_1\;,\end{equation*}
and then
\begin{equation} \label{stepb}
\textnormal{e}^{-(-\lambda_1+\alpha)t}\lvert\lvert \tilde{u}(t) \rvert\rvert_\mathcal{V} \leq \lvert\lvert \tilde{u}_0\rvert\rvert_\mathcal{V} +\tilde\ell \eta c \int_0^t \textnormal{e}^{-(-\lambda_1+\alpha)t_1}(t-t_1)^{-\frac{1}{2}} \txtd t_1 + \tilde\ell c \int_0^t \textnormal{e}^{-(-\lambda_1+\alpha)t_1}(t-t_1)^{-\frac{1}{2}} \lvert\lvert \tilde{u}(t_1)\rvert\rvert_\mathcal{V} \txtd t_1\;.\end{equation}
The rest of the proof is equivalent to the steps in \cite[Proposition 2.6]{https://doi.org/10.48550/arxiv.2108.11073}, where it is proven that the right-hand side in $\eqref{stepb}$ is bounded by $c_1 \eta$, for $c_1>0$ that does not depend on $t$ but is dependent on $T$, by assuming $\tilde{u}_0=a(\omega)$.
\end{proof}

We define now, for $\textbf{v}=v_1\wedge...\wedge v_k$ and $\textbf{w}=w_1\wedge...\wedge w_k \in \wedge^k \mathcal{H}$,
\begin{equation*} Q^{(k)}_\delta(\textbf{v},\textbf{w}):=\delta\big(\Pi_{\textbf{i}_0}\textbf{v},\textbf{w}\big)-\big(\Pi_{\textbf{i}_0}^\perp\textbf{v},\textbf{w}\big)\end{equation*}
with $\Pi_{\textbf{i}}$ denoting the projection on $\textbf{e}_{\textbf{i}}:=e_{i_1}\wedge...\wedge e_{i_k}$, $\textbf{i}=\{i_1,...,i_k\}$ and $\textbf{i}_0=\{1,...,k\}$ aside from permutations. The scalar product present in such construction is defined in $\eqref{other scal}$. We denote by $\lvert\lvert\cdot\rvert\rvert$ the norm defined by it on $\wedge^k\mathcal{H}$. In the next proof we make use of $\Lambda_{\textbf{i}_0}:=\sum_{j=1}^k (\alpha-\lambda_{i_j})$.
\begin{lm}\label{Lemma2}
Let $T>0$ and $0<\epsilon\ll \dfrac{1}{k} (\lambda_{k+1}-\lambda_k)$ be fixed, and let $\omega\in\Omega$ be an event with the property that the nonlinear term $B^t_\omega:=-3 a (\theta^t \omega)^2$ satisfies
\[\lvert\lvert B_\omega^t\rvert\rvert_\mathcal{V}\leq\epsilon\]
for all $t\in[0,T]$. Finally, assume $\textbf{v}_0=v_0^1\wedge...\wedge v_0^k\in\wedge^k L^2$ satisfies $Q_\delta^{(k)}(\textbf{v}_0)>0$ for $\delta>0$ for which
\[\epsilon (1+\delta)k\leq \Lambda_{\textbf{i}_0}-\Lambda_{\textbf{i}}-\frac{\epsilon (1+\delta) k}{\delta} \;\;\;,\]
for all $\textbf{i}\neq\textbf{i}_0$. Under these conditions, the $k$-blade $\textbf{v}_t:=\wedge^k \txtD_{a(\omega)}\varphi_\omega^t(\textbf{v}_0)$, corresponding to the solutions $v_t^j=\txtD_{a(\omega)}\varphi_\omega^t(v_0^j)$ of the first variation equation with initial condition set at $v_0^j$ for $j\in\{1,...,k\}$ and assumed at time $t>0$, satisfies the inequality
\begin{equation}\label{disq}\frac{1}{2}\frac{\txtd}{\txtd t}Q_\delta^{(k)}(\textbf{v}_t)\geq \bigg(\Lambda_{\textbf{i}_0}-\frac{(1+\delta)k\epsilon}{\delta}\bigg)Q_\delta^{(k)}(\textbf{v}_t).
\end{equation}
\end{lm}

\begin{proof}
We know that

\begin{equation*} \lvert Q_\delta^{(k)}(\textbf{v},\textbf{w})\rvert \leq (1+\delta) \lvert\lvert \textbf{v}\rvert\rvert \; \lvert\lvert \textbf{w}\rvert\rvert \;\;.\end{equation*}
We then denote $v_j=v^j_t=\txtD_{a(\omega)}\varphi_\omega^t(v^j_0)$. For the multiplication operator $B=B_\omega^t$ we use the hypothesis $\lvert\lvert B\rvert\rvert_\mathcal{V}\leq \epsilon$ which implies that $\lvert\lvert B v\rvert\rvert_\mathcal{H} \leq \epsilon \lvert\lvert v \rvert\rvert_\mathcal{H}$ for any $v \in \mathcal{H}$. Hence, we obtain
\begin{equation*}\begin{split} 
\dfrac{1}{2}\dfrac{\txtd}{\txtd t} Q_\delta^{(k)}(v_1\wedge...\wedge v_k, \; v_1\wedge...\wedge v_k)&=\sum_{j=1}^k  Q_\delta^{(k)}(v_1\wedge...\wedge v_k, \; v_1\wedge...\wedge v_{j-1}\wedge \dot{v}_j \wedge v_{j+1}\wedge...\wedge v_k)\\
&=\sum_{j=1}^k  Q_\delta^{(k)}(v_1\wedge...\wedge v_k, \; v_1\wedge...\wedge v_{j-1}\wedge (\Delta-g+\alpha+B)v_j \wedge v_{j+1}\wedge...\wedge v_k)\\
&\geq\sum_{j=1}^k  Q_\delta^{(k)}(v_1\wedge...\wedge v_k, \; v_1\wedge...\wedge v_{j-1}\wedge (\Delta-g+\alpha)v_j \wedge v_{j+1}\wedge...\wedge v_k)\\
&\quad -\epsilon(1+\delta)k\lvert\lvert v_1\wedge...\wedge v_k\rvert\rvert^2\;\;.
\end{split}\end{equation*}
We know the existence of some coefficients $\rho_{\textbf{i}}$ so that

\begin{equation*}v_1\wedge...\wedge v_k=\sum_{\textbf{i}} \rho_{\textbf{i}} \textbf{e}_{\textbf{i}}\;\;.\end{equation*}
This leads to

\begin{equation*} 
Q_\delta^{(k)}(v_1\wedge...\wedge v_k, \; v_1\wedge...\wedge v_{j-1}\wedge (\Delta-g+\alpha)v_j \wedge v_{j+1}\wedge...\wedge v_k)=\sum_{\textbf{i},\textbf{i}'}\rho_{\textbf{i}} \rho_{\textbf{i}'}(\alpha-\lambda_{i'_j})Q_\delta^{(k)}(\textbf{e}_{\textbf{i}},\textbf{e}_{\textbf{i}'})\;\;,
\end{equation*}
for which

\begin{equation*}
Q_\delta^{(k)}(\textbf{e}_{\textbf{i}},\textbf{e}_{\textbf{i}'})=\begin{cases} 
0 &,\text{if}\;\; \textbf{i}\neq\textbf{i}'\\ 
\delta &, \text{if} \;\; \textbf{i}=\textbf{i}'=\textbf{i}_0\\
-1 &, \text{if} \;\; \textbf{i}=\textbf{i}'\neq\textbf{i}_0
\end{cases} \;\;.\end{equation*}
We can obtain therefore

\begin{equation*} 
Q_\delta^{(k)}(v_1\wedge...\wedge v_k, \; v_1\wedge...\wedge v_{j-1}\wedge (\Delta-g+\alpha)v_j \wedge v_{j+1}\wedge...\wedge v_k)=\delta (\alpha-\lambda_j) \rho_{\textbf{i}_0}^2-\sum_{\textbf{i}\neq\textbf{i}_0}\rho_{\textbf{i}}^2(\alpha-\lambda_{i_j})
\end{equation*}
and

\begin{equation*} 
\sum_{j=1}^k Q_\delta^{(k)}(v_1\wedge...\wedge v_k, \; v_1\wedge...\wedge v_{j-1}\wedge (\Delta-g+\alpha)v_j \wedge v_{j+1}\wedge...\wedge v_k)=\delta \Lambda_{\textbf{i}_0} \rho_{\textbf{i}_0}^2-\sum_{\textbf{i}\neq\textbf{i}_0} \Lambda_{\textbf{i}} \rho_{\textbf{i}}^2\;\;.
\end{equation*}
This gives us the following result

\begin{equation*}\begin{split} 
\dfrac{1}{2}\dfrac{\txtd}{\txtd t} Q_\delta^{(k)}(v_1\wedge...\wedge v_k, \; v_1\wedge...\wedge v_k)&\geq \delta \Lambda_{\textbf{i}_0} \rho_{\textbf{i}_0}^2-\sum_{\textbf{i}\neq\textbf{i}_0} \Lambda_{\textbf{i}} \rho_{\textbf{i}}^2-\epsilon(1+\delta)k\sum_{\textbf{i}} \rho_{\textbf{i}}^2\\
&=\delta\bigg(\Lambda_{\textbf{i}_0}-\frac{\epsilon(1+\delta)k}{\delta}\bigg) \rho_{\textbf{i}_0}^2-\sum_{\textbf{i}\neq\textbf{i}_0}(\Lambda_{\textbf{i}}+\epsilon (1+\delta)k)\rho_{\textbf{i}}^2\;\;.
\end{split}\end{equation*}
Hence, for parameters that satisfy

\begin{equation*}
\epsilon (1+\delta)k\leq \Lambda_{\textbf{i}_0}-\Lambda_{\textbf{i}}-\frac{\epsilon (1+\delta) k}{\delta}
\end{equation*}
for all $\textbf{i}\neq\textbf{i}_0$, or equivalently,
\begin{equation*}
\epsilon k \frac{(1+\delta)^2}{\delta}
\leq \lambda_{k+1}-\lambda_k\;,
\end{equation*}
we can deduce that

\begin{equation*}
\dfrac{1}{2}\dfrac{\txtd}{\txtd t} Q_\delta^{(k)}(v_1\wedge...\wedge v_k, \; v_1\wedge...\wedge v_k)\geq \delta\bigg(\Lambda_{\textbf{i}_0}-\frac{\epsilon(1+\delta)k}{\delta}\bigg) Q_\delta^{(k)}(v_1\wedge...\wedge v_k, \; v_1\wedge...\wedge v_k) \;\;.
\end{equation*}

\end{proof}

\begin{thm}\label{tm2}
For any $k\geq 1$, $0<\eta\ll \lambda_{k+1}-\lambda_k$ and $T>0$, there exists $\Omega_0\subset \Omega$, a positive probability event, such that
\begin{equation*}\mathfrak{L}_k(t;\omega)\geq\sum_{j=1}^k\big(\alpha-\lambda_j \big)-\eta\end{equation*}
for all $\omega\in\Omega_0, \; t\in[0,T]$.
\end{thm}

\begin{proof}
Lemma \ref{Lemma1} proves for any $\epsilon>0$ the existence of a set $\Omega_0\subset\Omega$ so that for all $\omega\in\Omega_0$ we have the bound $\lvert\lvert a(\theta^t \omega)\rvert\rvert_\mathcal{V}<\epsilon$ for $t\in[0,T]$. Such result, along with the fact that $\mathcal{V}$ is a Banach algebra and a subset of $L^2(\mathcal{O})$, satisfies the first hypothesis of Lemma \ref{Lemma2}. Precisely, for any $\epsilon>0$ we obtain $\mathbb{P}\Big(\{\lvert\lvert B_\omega^t\rvert\rvert_\mathcal{V}< \epsilon, \text{for all }t\in[0,T]\}\Big)>0$.\\
From  Lemma \ref{Lemma2}, we know that $Q_\delta^{(k)}(\textbf{v}_0)>0$ implies $Q_\delta^{(k)}(\textbf{v}_t)>0$ for all $t\in[0,T]$ and, assuming $\delta=\sqrt{\epsilon}\ll1$ in $\eqref{disq}$,
\begin{equation*}\frac{1}{2}\frac{\txtd}{\txtd t}Q_\delta^{(k)}(\textbf{v}_t)\geq(\Lambda_{\textbf{i}_0}-2k\delta)Q_\delta^{(k)}(\textbf{v}_t)\;.\end{equation*}
This result gives
\begin{equation}\label{chip}Q_\delta^{(k)}(\textbf{v}_t)\geq\exp\{(\Lambda_{\textbf{i}_0}-2k\delta)t\}Q_\delta^{(k)}(\textbf{v}_0)\; .\end{equation}
We take $\textbf{v}_0\in\wedge^k \mathcal{H}$ and $M>1$ such that $Q_{\frac{\delta}{M}}^{(k)}(\textbf{v}_0)\geq0$.\footnote{Note that  for any $\dfrac{\delta}{M}>0$ this can be satisfied in a neighborhood of $\textbf{v}_0=\textbf{e}_{\textbf{i}_0}$.} It follows that
\begin{equation*}Q_\delta^{(k)}(\textbf{v}_0)=\delta\lvert\lvert\Pi_{\textbf{i}_0}\textbf{v}_0\rvert\rvert^2-\lvert\lvert\Pi_{\textbf{i}_0}^\perp\textbf{v}_0\rvert\rvert^2=\delta\Big(1-\dfrac{1}{M}\Big)\lvert\lvert\Pi_{\textbf{i}_0}\textbf{v}_0\rvert\rvert^2+Q_{\frac{\delta}{M}}^{(k)}(\textbf{v}_0)\geq \frac{\delta(M-1)}{M}\lvert\lvert\Pi_{\textbf{i}_0}\textbf{v}_0\rvert\rvert^2\;.
\end{equation*}
Since $Q_{\frac{\delta}{M}}^{(k)}(\textbf{v}_0)\geq0$ we know that $\lvert\lvert\textbf{v}_0\rvert\rvert^2\leq(1+\frac{\delta}{M})\lvert\lvert\Pi_{\textbf{i}_0}\textbf{v}_0\rvert\rvert^2$ and, therefore,
\begin{equation}\label{chop}Q_\delta^{(k)}(\textbf{v}_0)\geq \frac{\delta(M-1)}{M\big(1+\frac{\delta}{M}\big)}\lvert\lvert\textbf{v}_0\rvert\rvert^2=\frac{\delta(M-1)}{M+\delta}\lvert\lvert\textbf{v}_0\rvert\rvert^2\;.\end{equation}
Since $Q_\delta^{(k)}(\textbf{v}_t)\leq\delta\lvert\lvert\textbf{v}_t\rvert\rvert^2$, and using \eqref{chip} and \eqref{chop}, we obtain
\begin{equation*}\lvert\lvert\textbf{v}_t\rvert\rvert^2\geq \frac{M-1}{M+\delta} \exp\{(\Lambda_{\textbf{i}_0}-2k\delta)t\} \lvert\lvert\textbf{v}_0\rvert\rvert^2\;.\end{equation*}
The proof is complete considering Lemma \ref{Lemma1}, the fact that the prefactor can be close to $1$ for $M\longrightarrow +\infty$ and taking $\eta=2 k \delta\ll \lambda_{k+1}-\lambda_k$ with $\delta=\sqrt{\epsilon}$.
\end{proof}

\paragraph{Remark.} Theorems $\ref{tm1}$ and $\ref{tm2}$ prove that the right extreme of the interval of the possible values assumed by $\mathfrak{L}_k$ is $\Lambda_{\textbf{i}_0}$. In particular for $k=1$ it is $\alpha-\lambda_1$.\medskip

For $k>1$ the value of $\alpha$ that satisfies $\Lambda_{\textbf{i}_0}=0$ is not $\lambda_k$, therefore the change of sign of the highest possible value assumed by $\mathfrak{L}_k$ is not associated with a bifurcation event. The fact that $\dfrac{1}{k}\sum_{j=1}^k \lambda_j<\lambda_k$ implies that $\Lambda_{\textbf{i}_0}=0$ is satisfied for $\alpha<\lambda_k$ i.e. before the $k$-th bifurcation threshold.


\paragraph{Remark.} The introduction of the heterogeneous term $g$ in \eqref{mainsyst} induces a change of value in the bifurcation thresholds $\{\lambda_k\}_{k\in\mathbb{N}\setminus\{0\}}$ but maintains their existence. For any $k\in\mathbb{N}\setminus\{0\}$ there is a dependence on $\Lambda_{\textbf{i}_0}$ of $g$ for all $\alpha\in\mathbb{R}$ and therefore the choice of $g$ shifts the values of $\alpha$ at which the highest possible value of $\mathfrak{L}_k$ changes sign. 

\section{Early warning signs}

The Chafee-Infante equation has been used to model dynamics in many application areas, e.g., in climate systems (for instance in \cite{debussche2013dynamics}). It is well-understood that, in many applications, it is crucial to take into account stochastic dynamics. In this class of systems, stochasticity can reveal early warning signs for abrupt changes. In the case of $\eqref{mainsyst}$, the tipping phenomenon happens at bifurcation points upon varying $\alpha$. The most important change in the system happens when $\alpha$ crosses $\lambda_1$. Then a pitchfork bifurcation occurs and the null function loses its stability. Therefore, we we will consider $\alpha<\lambda_1$ and the operator $A+\alpha$ that generates a $\mathcal{C}_0$-contraction semigroup (see Appendix B).

\subsection{Early warning signs for the linearized problem}

In \cite{early,kuehn2019scaling} early warning signs for SPDEs which are based on the covariance operator of the linearized problem are presented. The reliability of such approximation relies upon estimates for the higher-order terms, which we are going to consider in the next section. First, we focus on the linearized problem, which is given by 
\begin{equation}\label{linsyst}\begin{cases}
\txtd w=A_\alpha w ~\txtd t+ \sigma ~\txtd W_t\\
w(\cdot,0)=w_0\in \mathcal{H},
\end{cases}\end{equation}
for $t>0$, $x\in\mathcal{O}$, with $A_\alpha=A+\alpha$ and Dirichlet boundary conditions. The existence and uniqueness in $L^2(\Omega,\mathcal{F},\mathbb{P};\mathcal{H})$ of the mild solution of $\eqref{linsyst}$, defined as 
\begin{equation*} w_{A_\alpha}(t)=\textnormal{e}^{A_\alpha t}w_0+\sigma\int_0^t \textnormal{e}^{A_\alpha(t-t_1)} \txtd W_{t_1}\;,\end{equation*}
is satisfied when its covariance operator,
\begin{equation*} V(t):=\sigma^2\int_0^t \textnormal{e}^{A_\alpha t_1} Q \textnormal{e}^{A_\alpha^* t_1} \txtd t_1,\end{equation*}
is trace class for all $t>0$ (\cite[Theorem 5.2]{DaPrato}).\footnote{The symbol * indicates the adjoint in respect to the scalar product on $\mathcal{H}$.} Since $A_\alpha$ generates a $\mathcal{C}_0$-contraction semigroup such a requirement can be satisfied by bounded $Q$.\\
In order to prove the existence and uniqueness in $L^2(\Omega,\mathcal{F},\mathbb{P};\mathcal{H})$ of the mild solution of \eqref{mainsyst} one only has to require continuity in $\mathcal{H}$ of the solution of $\eqref{linsyst}$. Hence the conditions (see Appendix B) we take for $Q$ are: 

\begin{enumerate}[label=\arabic*$'$., ref=\arabic*$'$]
\item \label{property1} there exists an $D>0$ such that $b_j=e_j'$ for all $j>D$,
\item \label{property2} $q_j>0$ for all $j\in \mathbb{N}\setminus\{0\}$,\footnote{The strict positivity is not required but avoids further assumptions. Description on the case in which non-negativity of the spectrum is assumed are discussed in the relevant parts separately.}
\item \label{property3} $\sum_{j=1}^\infty q_j \lambda_j'^\gamma <+\infty$ for a $\gamma>-1$,\footnote{From the order of divergence of $\{\lambda_j'\}_j$, the assumption is satisfied by bounded $Q$ and $-1<\gamma<-\frac{1}{2}$.}
\item \label{property4} $Q$ is bounded.\footnote{For $Q$ bounded but not trace class, the $Q$-Wiener process $W_t$ does not have continuous paths in $\mathcal{H}$, but there is another Hilbert space $\mathcal{H}_1$ in which it does and its defining series converges in $L^2(\Omega, \mathcal{F},\mathbb{P};\mathcal{H}_1)$ (\cite[Chapter 4]{DaPrato}).}
\end{enumerate}

We can then conclude (\cite[Theorem 2.34]{da2004kolmogorov}) that the transition semigroup of the system $\eqref{linsyst}$ has a unique invariant measure which is Gaussian and has mean zero and covariance operator $V_\infty:=\underset{t\rightarrow\infty}{\lim} V(t)$. Such an operator is linear, self-adjoint, non-negative, and continuous in $\mathcal{H}$. Using the stated properties it is possible to prove the following proposition and obtain a first early warning sign building upon \cite{kuehn2019scaling}.

\begin{prop}\label{EWS1}
Consider any pair of linear operators $A$ and $Q$, to whose respective eigenfunctions and eigenvalues we refer as $\{e_k,\lambda_k\}_{k\in\mathbb{N}\setminus\{0\}}$ and $\{b_k,q_k\}_{k\in\mathbb{N}\setminus\{0\}}$, that satisfy the properties \textnormal{\ref{property1}}, \textnormal{\ref{property2}}, \textnormal{\ref{property3}}, \textnormal{\ref{property4}}, and assume $A_\alpha=A+\alpha$ self-adjoint operator that generates a $\mathcal{C}_0$-contraction semigroup for a fixed $\alpha\in\mathbb{R}$. Then the covariance operator $V_\infty$ given by the first equation in $\eqref{linsyst}$ satisfies
\begin{equation}\label{ews1}
    \langle V_\infty e_{j_1},e_{j_2}\rangle=\dfrac{\sigma^2}{\lambda_{j_1}+\lambda_{j_2}-2\alpha}\sum_{n=1}^\infty q_n \langle e_{j_1},b_n\rangle \langle e_{j_2}, b_n\rangle
\end{equation}
for all $j_1,j_2\in\mathbb{N}\setminus\{0\}$.
\end{prop}
\begin{proof}
First, we consider the Lyapunov equation derived in~\cite[Lemma 2.45]{da2004kolmogorov} and obtain
\begin{equation}\label{lyap} \langle A_\alpha V_\infty f_1,f_2\rangle +\langle A_\alpha^* f_1, V_\infty f_2\rangle =-\sigma^2 \langle f_1,Q f_2\rangle \end{equation}
for all $f_1,f_2\in \mathcal{H}$. 
From the self-adjointness of $A_\alpha$ and $V_\infty$, for any $j_1,j_2\in \mathbb{N}\setminus\{0\}$ we deduce that
\begin{equation*} \langle A_\alpha V_\infty e_{j_1},e_{j_2}\rangle +\langle A_\alpha e_{j_1}, V_\infty e_{j_2}\rangle =(-\lambda_{j_1}-\lambda_{j_2}+2\alpha)\langle V_\infty e_{j_1},e_{j_2}\rangle . \end{equation*}
We can then obtain the relation (see Appendix C)
\begin{equation*}\begin{split}&\langle V_\infty e_{j_1},e_{j_2}\rangle =\dfrac{\sigma^2}{\lambda_{j_1}+\lambda_{j_2}-2\alpha}\langle e_{j_1},Q e_{j_2}\rangle =\dfrac{\sigma^2}{\lambda_{j_1}+\lambda_{j_2}-2\alpha}\sum_{n=1}^\infty\langle e_{j_1},b_n\rangle  \langle b_n,Q e_{j_2}\rangle\\
&=\dfrac{\sigma^2}{\lambda_{j_1}+\lambda_{j_2}-2\alpha}\sum_{n=1}^\infty\langle e_{j_1},b_n\rangle \sum_{m=1}^\infty\langle b_n, b_m\rangle \langle b_m,Q e_{j_2}\rangle =\dfrac{\sigma^2}{\lambda_{j_1}+\lambda_{j_2}-2\alpha}\sum_{n=1}^\infty\langle e_{j_1},b_n\rangle \langle b_n,Q e_{j_2}\rangle \\
&=\dfrac{\sigma^2}{\lambda_{j_1}+\lambda_{j_2}-2\alpha}\sum_{n=1}^\infty q_n \langle e_{j_1},b_n\rangle \langle e_{j_2}, b_n\rangle ,\end{split}\end{equation*}
which finishes the proof.
\end{proof}

The early warning sign is qualitatively visible because for $j_1=j_2=1$ the equation $\eqref{ews1}$ diverges when $\alpha=\lambda_1$. This allows one to apply usual techniques (extracting a scaling law from a log-log plot) to estimate the bifurcation point, if sufficient data from the system is available. We also note that the proposition can easily be extended to operators $A$ that are not self-adjoint by studying in that case the real parts of its eigenvalues. Additionally, the strict positivity of $Q$ can be substituted by non-negativity of the operator, although in such case the divergence by the $j_1=j_2=1$ component is not implied. The result $\eqref{ews1}$ leads to a second type of early warning sign. 

\begin{thm}\label{EWS2}
For any pair of linear operators $A$ and $Q$ that satisfy the properties stated in Proposition \ref{EWS1} and for $A$ whose eigenfunctions $\{e_k\}_{k\in\mathbb{N}\setminus\{0\}}$ form a basis of $\mathcal{H}$, we have
\begin{equation}\label{ews2}
    \langle V_\infty f_1,f_2 \rangle =\sigma^2\sum_{j_1=1}^\infty \sum_{j_2=1}^\infty \dfrac{ \langle f_1, e_{j_2}\rangle \langle f_2, e_{j_1}\rangle }{\lambda_{j_1}+\lambda_{j_2}-2\alpha}\bigg(\sum_{n=1}^\infty q_n \langle e_{j_1},b_n\rangle \langle e_{j_2}, b_n\rangle \bigg),
\end{equation}
for any $f_1,f_2\in \mathcal{H}$.
\end{thm}

\begin{proof}
Applying the method described in Appendix C and the previous proposition,
\begin{equation*}\begin{split}
&\langle V_\infty f_1,f_2\rangle =\sum_{j_1=1}^\infty \langle V_\infty f_1, e_{j_1}\rangle  \langle e_{j_1}, f_2\rangle =\sum_{j_1=1}^\infty \langle f_1, V_\infty e_{j_1}\rangle \langle e_{j_1}, f_2\rangle \\
&=\sum_{j_1=1}^\infty \sum_{j_2=1}^\infty \langle f_1, e_{j_2}\rangle \langle e_{j_2},V_\infty e_{j_1}\rangle  \langle e_{j_1}, f_2\rangle \\
&=\sigma^2\sum_{j_1=1}^\infty \sum_{j_2=1}^\infty \dfrac{ \langle f_1, e_{j_2}\rangle \langle f_2, e_{j_1}\rangle }{\lambda_{j_1}+\lambda_{j_2}-2\alpha}\bigg(\sum_{n=1}^\infty q_n \langle e_{j_1},b_n\rangle \langle e_{j_2}, b_n\rangle \bigg)\;.
\end{split}\end{equation*}
\end{proof}

The last theorem demonstrates that, given $f_1$ and $f_2$ that have a component in $e_1$, the resulting scalar product $\langle V_\infty f_1,f_2\rangle$ diverges in $\alpha=\lambda_1$. Theorem \ref{EWS1} can describe local behaviour and it can also be used to study the effect of stochastic perturbations on a single point $x_0\in\mathring{\mathcal{O}}$. This can be achieved by taking $f_1, f_2\in \mathcal{H}$ in $\eqref{ews2}$ equal to a function 
\begin{equation}\label{constr_f}
f_\epsilon\in \mathcal{H}\qquad \text{ with}\quad f_\epsilon=\left.\dfrac{1}{\epsilon}\chi_{[x_0-\frac{\epsilon}{2},x_0+\frac{\epsilon}{2}]}\right|_\mathcal{O}\text{ for $0<\epsilon\ll 1$,}    
\end{equation}
for $\chi_{[\cdot,\cdot]}$ meant as the indicator function on an interval. Clearly, $f_\epsilon$ is in $L^1(\mathcal{O})$ and in $L^2(\mathcal{O})$. The $L^1$ norm is equal to $1$ for small $\epsilon$ and the sequence $\{f_\epsilon\}_\epsilon$ converges weakly in $L^1$ to the Dirac delta in $x_0$ for $\epsilon\rightarrow0$. For these reasons and recalling that $\{e_j\}_{j\in\mathbb{N}\setminus\{0\}}\subset H^1_0(\mathcal{O}) \subset L^\infty(\mathcal{O})$ by the Sobolev inequality, we can obtain $\underset{{\epsilon\rightarrow0}}{\lim}\langle V_\infty f_\epsilon,f_\epsilon\rangle$.

\begin{cor}\label{EWS3}
Consider any pair of linear operators $A$ and $Q$ with same properties as required in Theorem \ref{EWS2} and $\{f_\epsilon\}$ defined in \eqref{constr_f}, then 
\begin{equation}\label{ews3}
    \underset{{\epsilon\rightarrow0}}{\lim}\langle V_\infty f_\epsilon,f_\epsilon\rangle =\sigma^2\sum_{j_1=1}^\infty \sum_{j_2=1}^\infty \dfrac{ e_{j_1}(x_0) e_{j_2}(x_0)}{\lambda_{j_1}+\lambda_{j_2}-2\alpha}\bigg(\sum_{n=1}^\infty q_n \langle e_{j_1},b_n\rangle \langle e_{j_2}, b_n\rangle \bigg),
\end{equation}
for all $x_0\in\mathring{\mathcal{O}}$.
\end{cor}

\begin{proof}
Fix $x_0\in\mathring{\mathcal{O}}$. The goal of this proof is to show that the series in $\eqref{ews2}$ satisfies the dominated convergence hypothesis for $f_1=f_2=f_\epsilon$ for all $0<\epsilon$. In order to achieve that, we split the series
\begin{equation}\label{part0}
    \langle V_\infty f_\epsilon,f_\epsilon\rangle =
    \sum_{j_1=1}^\infty \sum_{j_2=1}^\infty \bigg\lvert\dfrac{ \langle f_\epsilon,e_{j_1}\rangle \langle f_\epsilon,e_{j_2}\rangle}{\lambda_{j_1}+\lambda_{j_2}-2\alpha}\langle e_{j_1},Q e_{j_2}\rangle \bigg\rvert
\end{equation}
in two parts.\\
\begin{itemize}
\item[Part 1.] Following $\eqref{sync}$ in Appendix A we get the existence of $c,M>0$ so that $\lvert\lvert e_j-e_j'\rvert\rvert_\infty\leq c\dfrac{1}{j}$, $\lvert\lvert e_j\rvert\rvert_\infty\leq M$ and $\lambda_j\sim j^2$ for all indexes $j\in\mathbb{N}\setminus\{0\}$. For $j_1=j_2$, we can use H\"older's inequality to study the series
\begin{equation}\label{part1}
    \dfrac{1}{2} \sum_{j=1}^\infty \bigg\lvert\dfrac{\langle f_\epsilon,e_j\rangle^2}{\lambda_j-\alpha}\langle e_j,Q e_j\rangle \bigg\rvert\leq \dfrac{M^2}{2} q_* \sum_{j=1}^\infty \dfrac{1}{\lambda_j-\alpha}<+\infty\;,
\end{equation}
for $q_*=\underset{j\in\mathbb{N}\setminus\{0\}}{\sup}\{q_j\}$.
\item[Part 2.] To prove the convergence of $\eqref{part0}$ we consider
\begin{equation*}
\begin{split}
    &\sum_{j_1=1}^\infty \sum_{j_2\neq j_1} \bigg\lvert\dfrac{ \langle f_\epsilon,e_{j_1}\rangle \langle f_\epsilon,e_{j_2}\rangle}{\lambda_{j_1}+\lambda_{j_2}-2\alpha}\langle e_{j_1},Q e_{j_2}\rangle \bigg\rvert
    \leq M^2 \sum_{j_1=1}^\infty \sum_{j_2\neq j_1} \bigg\lvert\dfrac{\langle e_{j_1},Q e_{j_2}\rangle }{\lambda_{j_1}+\lambda_{j_2}-2\alpha}\bigg\rvert\\
    &\leq M^2 \sum_{j_1=1}^\infty \sum_{j_2\neq j_1} \bigg\lvert\dfrac{\langle e_{j_1}',Q e_{j_2}'\rangle +\langle e_{j_1}',Q (e_{j_2}-e_{j_2}')\rangle +\langle e_{j_1}-e_{j_1}',Q e_{j_2}'\rangle +\langle e_{j_1}-e_{j_1}',Q (e_{j_2}-e_{j_2}')\rangle }{\lambda_{j_1}+\lambda_{j_2}-2\alpha}\bigg\rvert\;.
\end{split}
\end{equation*}
We recall that $\langle e_{j_1}',Q e_{j_2}'\rangle =q_{j_1} \delta_{j_2}^{j_1}$ for $j_1,j_2>D$, for $\delta$ meant as the Kronecker delta. Also, by the inclusion in $\mathcal{H}$ of the space of functions $L^\infty(\mathcal{O})$ with Dirichlet conditions, we can state there exists $c_1>0$ so that for all $w\in \mathcal{H}$, $\lvert\lvert w\rvert\rvert_\mathcal{H}\leq c_1 \lvert\lvert w\rvert\rvert_\infty$. Therefore, by H\"older's inequality and the min-max principle, it follows that 
\begin{equation}\label{part2}
\begin{split}
&\sum_{j_1=1}^\infty \sum_{j_2\neq j_1} \bigg\lvert\dfrac{ \langle f_\epsilon,e_{j_1}\rangle  \langle f_\epsilon,e_{j_2}\rangle }{\lambda_{j_1}+\lambda_{j_2}-2\alpha}\langle e_{j_1},Q e_{j_2}\rangle \bigg\rvert\leq M^2 q_* D^2+
M^2 q_* \sum_{j_1=1}^\infty \sum_{j_2\neq j_1} \bigg\lvert\dfrac{c c_1 \frac{1}{j_2}+c c_1 \frac{1}{j_1}+c^2 c_1^2 \dfrac{1}{j_1 j_2}}{\lambda_{j_1}+\lambda_{j_2}-2\alpha}\bigg\rvert\\
&\leq M^2 q_* D^2+ M^2 q_* c c_1 \Bigg(
\sum_{j_1=1}^\infty \sum_{j_2\neq j_1} \dfrac{1}{(\lambda_{j_1}+\lambda_{j_2}-2\alpha)} \Bigg( \bigg\lvert\dfrac{1}{j_2}\bigg\rvert
+ \bigg\lvert\dfrac{1}{j_1}\bigg\rvert
+ \bigg\lvert\dfrac{c c_1}{j_1 j_2}\bigg\rvert\Bigg)\Bigg)<\infty\;.
\end{split}
\end{equation}
\end{itemize}
From $\eqref{part1}$ and $\eqref{part2}$ we have proven the convergence of $\eqref{part0}$. By the weak convergence in $\mathcal{H}$ of $f_\epsilon$ to $\delta_{x_0}$, the Dirac delta on $x_0$, and the dominated convergence theorem we can then calculate
\begin{equation*}\begin{split}
&\underset{{\epsilon\rightarrow0}}{\lim}\langle V_\infty f_\epsilon,f_\epsilon\rangle =\underset{{\epsilon\rightarrow0}}{\lim}\sigma^2\sum_{j_1=1}^\infty \sum_{j_2=1}^\infty \dfrac{ \langle f_\epsilon,e_{j_1}\rangle  \langle f_\epsilon,e_{j_2}\rangle }{\lambda_{j_1}+\lambda_{j_2}-2\alpha}\bigg(\sum_{n=1}^\infty q_n \langle e_{j_1},b_n\rangle \langle e_{j_2}, b_n\rangle \bigg)\\
&=\sigma^2\sum_{j_1=1}^\infty \sum_{j_2=1}^\infty \underset{{\epsilon\rightarrow0}}{\lim} \dfrac{ \langle f_\epsilon,e_{j_1}\rangle \langle f_\epsilon,e_{j_2}\rangle }{\lambda_{j_1}+\lambda_{j_2}-2\alpha}\bigg(\sum_{n=1}^\infty q_n \langle e_{j_1},b_n\rangle \langle e_{j_2}, b_n\rangle \bigg)\\
&=\sigma^2\sum_{j_1=1}^\infty \sum_{j_2=1}^\infty \dfrac{ e_{j_1}(x_0) e_{j_2}(x_0)}{\lambda_{j_1}+\lambda_{j_2}-2\alpha}\bigg(\sum_{n=1}^\infty q_n \langle e_{j_1},b_n\rangle  \langle e_{j_2}, b_n\rangle \bigg),
\end{split}\end{equation*}
which yields the desired result.
\end{proof}

The efficiency of the early warning sign described in Corollary \ref{EWS3} for any $x_0\in\mathring{\mathcal{O}}$ is underlined by the following Proposition (\cite[Theorem 2.6]{poschel1987inverse}).

\begin{prop}\label{rootsprop}
Given the Schrödinger operator $A=\Delta-g:H^2(\mathcal{O})\cap H^1_0(\mathcal{O})\rightarrow L^2(\mathcal{O})$ and $g\in L^2(\mathcal{O})$, then its eigenfunctions $\{e_k\}_{k\in\mathbb{N}\setminus\{0\}}$ admit exactly $k-1$ roots in $\mathring{\mathcal{O}}$ respectively.
\end{prop}
From Proposition \ref{rootsprop} and the strict positivity of $Q$, the influence of the divergence of $\eqref{ews1}$ for $j_1=j_2=1$ affects all interior points of $\mathcal{O}$. Hence, on each point $x_0\in\mathring{\mathcal{O}}$ one gets a divergence of $\eqref{ews3}$ of hyperbolic-function type multiplied by the constant $e_1(x_0)^2\bigg(\sum_{n=1}^\infty q_n \langle e_1,b_n\rangle^2\bigg)$. Specifically, the point at which the divergence is visible the most is for $x_0\in\mathring{\mathcal{O}}$ such that $x_0=\underset{x\in\mathring{\mathcal{O}}}{\text{argmax}}\{\lvert e_1(x) \rvert\}$. This is highly relevant for practical applications as it determines a good measurement location for the spatially heterogeneous SPDE.

\paragraph{Remark.} The choice of $g$ influences the early warning signs described by \eqref{ews1}, \eqref{ews2} and \eqref{ews3} in two aspects. By affecting the values of $\{\lambda_j\}_j$ it is related to the value of $\alpha$ at which warning signs diverge and from its relation with the functions $\{e_j\}_j$ it influences the directions on $\mathcal{H}$ on which such divergence is stronger. The point $x_0=\underset{x\in\mathring{\mathcal{O}}}{\text{argmax}}\{\lvert e_1(x) \rvert\}$ at which \eqref{ews3} assumes the highest values close to the bifurcation, i.e. $x_0\in\mathring{\mathcal{O}}$ that satisfies
\begin{equation*}
\underset{{\epsilon\rightarrow0}}{\lim}\dfrac{\langle V_\infty f^{x_0}_\epsilon,f^{x_0}_\epsilon\rangle}{\langle V_\infty f^{x}_\epsilon,f^{x}_\epsilon\rangle}\geq 1 
\quad\text{for any}\quad x\in\mathring{\mathcal{O}} \quad\text{and}
\quad f^x_\epsilon=\left.\dfrac{1}{\epsilon}\chi_{[x-\frac{\epsilon}{2},x+\frac{\epsilon}{2}]}\right|_\mathcal{O},    
\end{equation*}
also depends on the choice of $g$ by construction.

\section{Error and moment estimates}

In the following section we will study the non-autonomous system
\begin{equation}\label{nonasyst}\begin{cases}
\txtd u(x,t)=(\Delta u(x,t) + \alpha(\epsilon t) u(x,t) - g(x) u(x,t) -u(x,t)^3) ~\txtd t +\sigma~ \txtd W_t\\
u(\cdot,0)=u_0\in \mathcal{H}\\
\left.u(\cdot,t)\right|_{\partial\mathcal{O}}=0\;\;\;\;,\;\;\;\;\forall t\geq0
\end{cases}\end{equation}
for $0<\epsilon\ll1$, $\alpha\in \mathcal{C}([0,\tau])$ for a given $\tau>0$ and the $\sigma,g,W_t,\mathcal{H},\mathcal{O}$ defined as in $\eqref{mainsyst}$. The covariance operator $Q$ is assumed to satisfy properties \ref{property1}, \ref{property2}, \ref{property3} and \ref{property4} from the previous section. The slow dependence on time is in accordance to the fact that in simulations of realistic applications, the previous parameter $\alpha$ in $\eqref{mainsyst}$ is not constant but slowly changing; see also~\cite{GnannKuehnPein}. For fixed $\epsilon$ we will assume $ \alpha(\epsilon t)<\lambda_1$ in $0\leq \epsilon t\leq\tau$. Therefore we will study the system before the first non-autonomous bifurcation. Following the methods used in \cite{berglund2006noise,berglund2022stochastic} we want to understand the probabilistic properties of the first time in which the solutions of $\eqref{nonasyst}$, with $u_0$ close to the null function, get driven apart from a chosen solution of the equivalent deterministic system, i.e. $\eqref{nonasyst}$ with $\sigma=0$, taking as metric the distance induced by a fractional Sobolev norm of small degree.\\ \medskip
In this section we use the Landau notation on scalars $\rho_1\in\mathbb{R},\;\rho_2>0$ as $\rho_1=O(\rho_2)$ if there exists a constant $c>0$ that satisfies $\lvert \rho_1 \rvert\leq c \rho_2$. Such constant is independent from $\epsilon$ and $\sigma$, however it depends on $g$ and $\alpha$.\\ \medskip
The first equation in $\eqref{nonasyst}$ can be studied for the slow time $\epsilon t$ setting, named again for convenience $t$, giving the form
\begin{equation*}
    \label{slow}
    \txtd u(x,t)=\dfrac{1}{\epsilon}(\Delta u(x,t) + \alpha(t) u(x,t) - g(x) u(x,t) -u(x,t)^3) ~\txtd t +\dfrac{\sigma}{\sqrt{\epsilon}} ~\txtd W_t
\end{equation*}
through rescaling of time.\\
\medskip
We will denote the $A^s$-norm of power $s\geq 0$ of any function $\phi=\sum_{k=1}^\infty \rho_k e_k \in \mathcal{H}$ with $\{\rho_k\}_{k\in\mathbb{N}\setminus\{0\}}\subset \mathbb{R}$ as
\begin{equation*}
    \lvert\lvert \phi \rvert\rvert_{A^s}^2:=\sum_{k=1}^\infty \lambda_k^s \rho_k^2\;\;\;.
\end{equation*}
The functions in $\mathcal{H}$ that present finite $A^s$-norm define the space $\mathcal{V}^s$. It can be proven that the $A^s-$norm is equivalent to the fractional Sobolev norm on $W^{s,2}_0(\mathcal{O})$, the Sobolev space of degree $s$, $p=2$ and Dirichlet conditions, labeled as $\mathcal{V}^s-$norm,
\begin{equation*}
\lvert\lvert \phi \rvert\rvert_{\mathcal{V}^s}^2:=\sum_{k=1}^\infty \lambda_k'^s \langle \phi,e_k'\rangle^2\;.
\end{equation*}
Such result is given by the fact that, as shown in Appendix A, $\mathcal{D}((-\Delta)^\frac{1}{2})=\mathcal{D}((-A)^\frac{1}{2})$ and from the characterization of interpolation spaces described by \cite[Theorem 4.36]{lunardi2018interpolation}. In particular, it implies $\mathcal{V}^s:=\mathcal{D}((-A)^\frac{s}{2})=W^{s,2}_0(\mathcal{O})$.\\
For simplicity, we will always assume $g(x)\geq 1$ for all $x\in\mathcal{O}$, so that for any couple of parameters $0\leq s<s_1$ and $\phi\in \mathcal{V}^{s_1}$, $\lvert\lvert \phi \rvert\rvert_{A^s}\leq \lvert\lvert \phi \rvert\rvert_{A^{s_1}}$. The following lemma shows an important property, a Young-type inequality, of the $A^s-$norms which we exploit.
\begin{lm}\label{young-lm}
Set $0<r,s$ and $0<q<\dfrac{1}{2}$ so that $q+\dfrac{1}{2}<r+s$. Then there exists $c(r,s,q)>0$ for which
\begin{equation}\label{young-ine}
\lvert\lvert \phi_1 \phi_2 \rvert\rvert_{A^q}\leq c(r,s,q) \lvert\lvert \phi_1 \rvert\rvert_{A^r} \lvert\lvert \phi_2 \rvert\rvert_{A^s}\;,
\end{equation}
for any $\phi_1\in \mathcal{V}^r$ and $\phi_2\in \mathcal{V}^s$.
\end{lm}
\begin{proof}
The equivalence of the $A^{s_1}$-norm and the $\mathcal{V}^{s_1}$-norm for any $0\leq s_1\leq 1$ justifies the study of $\eqref{young-ine}$ for the $\mathcal{V}^{s_1}$-norms. The important change that such choice permits is the fact that for any $\phi\in \mathcal{H}$ we can rewrite the series 
\begin{equation*}
\phi(x)=\sum_{k=1}^\infty \langle \phi,e_k'\rangle e_k'(x)=\sum_{k=-\infty}^\infty \tilde\rho_k \dfrac{1}{\sqrt{L}}\textnormal{e}^{\txti\frac{\pi k}{L} x}
\end{equation*}
for $\tilde\rho_k:=-\dfrac{\text{sign}(k) \txti}{\sqrt{2}}\langle \phi,e_k'\rangle $, $\tilde\rho_0=0$ and for the symbol $\txti$ meant as the imaginary unit.\footnote{Note that $\Big\{\dfrac{1}{\sqrt{L}}\textnormal{e}^{\txti\frac{\pi k}{L} x}\Big\}_{k\in\mathbb{Z}}$ is not a basis of $\mathcal{H}$, therefore it is not a Parseval identity on such space.} Furthermore, it is easy to prove that for any $0\leq s_1\leq 1$
\begin{equation*}
\lvert\lvert \phi \rvert\rvert_{\mathcal{V}^{s_1}}^2=\sum_{k=-\infty}^\infty \lambda_k'^{s_1} \lvert \tilde\rho_k \rvert^2\;.
\end{equation*}
Since the product of any couple of elements in $\Big\{\dfrac{1}{\sqrt{L}}\textnormal{e}^{\txti\frac{\pi k}{L} x}\Big\}_{k\in\mathbb{Z}}$ is proportional to another member of the set, we can apply the Young-type inequality in \cite[Lemma 4.3]{berglund2013sharp} which implies the existence of a constant $c'(r,s,q)$ that satisfies
\begin{equation*}
\lvert\lvert \phi_1 \phi_2 \rvert\rvert_{\mathcal{V}^q}\leq c'(r,s,q) \lvert\lvert \phi_1 \rvert\rvert_{\mathcal{V}^r}\; \lvert\lvert \phi_2 \rvert\rvert_{\mathcal{V}^s} \;,
\end{equation*}
for any $\phi_1\in \mathcal{V}^r$ and $\phi_2\in \mathcal{V}^s$. From the equivalence of the $A^s$-norms and the $\mathcal{V}^s$-norms we can state that there exists $c(r,s,q)>0$ that satisfies the desired inequality.
\end{proof}

From the fact that $\alpha(t)<\lambda_1$ for $t<\tau$ we know that any solution $\bar{u}$ of the deterministic problem
\begin{equation}\label{nonasystdet}\begin{cases}
\txtd \bar{u}(x,t)=\dfrac{1}{\epsilon}(\Delta \bar{u}(x,t) + \alpha(t) \bar{u}(x,t) - g(x) \bar{u}(x,t) -\bar{u}(x,t)^3) ~\txtd t\\
\left.\bar{u}(\cdot,t)\right|_{\partial\mathcal{O}}=0\;\;\;\;,\;\;\;\;\forall t\geq0\;\;,
\end{cases}\end{equation}
with initial conditions in $\mathcal{H}$, approaches the null function exponentially in $\mathcal{H}$ in time $0\leq t\leq \tau$. Furthermore, following the proof of \cite[Proposition 2.3]{berglund2022stochastic} we prove in the following Proposition equivalent results for the $A^1-$norm.

\begin{prop}\label{propdet}
Given $\bar{u}$, a solution of $\eqref{nonasystdet}$ such that $\lvert\lvert \bar{u}(\cdot,0)\rvert\rvert_A:=\lvert\lvert \bar{u}(\cdot,0)\rvert\rvert_{A^1}\leq \delta$, then \begin{equation*}
    \lvert\lvert \bar{u}(\cdot,t)\rvert\rvert_A\leq \delta
\end{equation*}
for all $0\leq t\leq \tau$. Furthermore, we have that $\lvert\lvert \bar{u}(\cdot,t)\rvert\rvert_A$ approaches $0$ exponentially in $0\leq t\leq \tau$.
\end{prop}

\begin{proof}
We define the Lyapunov function
\begin{equation*}
    F_\text{L}(\phi):=\dfrac{1}{2}\Big(\lvert\lvert \nabla \phi \rvert\rvert_\mathcal{H}^2+
    \lvert\lvert g^\frac{1}{2}\phi \rvert\rvert_\mathcal{H}^2\Big)
    =\dfrac{1}{2} \langle -A \phi, \phi\rangle
\end{equation*}
for any $\phi\in \mathcal{V}$.\\
Therefore we obtain
\begin{equation*}\begin{split}
& \dfrac{\txtd}{\txtd t}F_\text{L}(\bar{u})=-\langle A \bar{u},\partial_t \bar{u}\rangle =\dfrac{1}{\epsilon}\Big(-\lvert\lvert A \bar{u} \rvert\rvert_\mathcal{H}^2-\alpha(t)\langle A \bar{u},\bar{u}\rangle +\langle A \bar{u}, \bar{u}^3\rangle \Big)\\
&=\dfrac{1}{\epsilon}\Big(-\langle (A+\lambda_1) \bar{u},A \bar{u}\rangle -(\alpha(t)-\lambda_1)\langle A \bar{u},\bar{u}\rangle +\langle A \bar{u}, \bar{u}^3\rangle \Big)\\
&\leq\dfrac{1}{\epsilon}\Big(-(\alpha(t)-\lambda_1)\langle A \bar{u},\bar{u}\rangle +\langle A \bar{u}, \bar{u}^3\rangle \Big)=\dfrac{2}{\epsilon}(\alpha(t)-\lambda_1)F_\text{L}(\bar{u})+\dfrac{1}{\epsilon}\langle A \bar{u}, \bar{u}^3\rangle \\
&= \dfrac{2}{\epsilon}(\alpha(t)-\lambda_1)F_\text{L}(\bar{u})-\dfrac{3}{\epsilon}\langle \nabla \bar{u}, \bar{u}^2 \nabla \bar{u}\rangle -\dfrac{1}{\epsilon}\lvert\lvert g^\frac{1}{2} \bar{u}^2\rvert\rvert_\mathcal{H}^2\leq \dfrac{2}{\epsilon}(\alpha(t)-\lambda_1)F_\text{L}(\bar{u})
\end{split}\end{equation*}
by using the min-max principle in
\begin{equation*}
-\langle (A+\lambda_1) \bar{u},A \bar{u}\rangle =\langle -A(A+\lambda_1) \bar{u}, \bar{u}\rangle \leq 0 \;\;.
\end{equation*}
We have hence shown that, given initial condition $\bar{u}(\cdot,0)=u_0\in \mathcal{H}$,
\begin{equation*}
F_\text{L}(\bar{u}(\cdot,t))\leq F_\text{L}(u_0) \textnormal{e}^{-\frac{2 c_1}{\epsilon} t}
\end{equation*}
for a constant $c_1>0$. The result now follows since the norm defined by $F_\text{L}$ is equivalent to $\lvert\lvert\cdot\rvert\rvert_A$.
\end{proof}

Given a function $\bar{u}$ that satisfies Proposition \ref{propdet} for a fixed $\delta>0$, we define the set in $\mathcal{V}$
\begin{equation*}
\mathcal{B}(h):=\{(t,\phi)\in [0,\tau]\times \mathcal{V}:\; \lvert\lvert \phi-\bar{u}\rvert\rvert_{A^s}<h\}
\end{equation*}
for $s,h>0$. The first-exit time from $\mathcal{B}(h)$ is the stopping time
\begin{equation*}
\tau_{\mathcal{B}(h)}:=\text{inf}\{t>0:(t,u(\cdot,t))\not\in \mathcal{B}(h)\}\;.
\end{equation*}
With these definitions we can obtain an estimate for the probability of jump outside of the defined neighbourhood over a finite time scale.

\begin{thm}\label{error estimate}
Set $q_*=\underset{j\in\mathbb{N}\setminus\{0\}}{\textnormal{sup}}\{q_j\}$.
For any $0<s<\dfrac{1}{2}$ and $\epsilon,\nu>0$ there exist $\delta_0,\kappa=\kappa(s),h_1, C_{\frac{h^2}{q_* \sigma^2}}(\kappa,T,\epsilon,s)>0$ for which, given $0<\sqrt{q_*}\sigma\ll h<h_1 \epsilon^\nu$ and a function $\bar{u}$ that solves $\eqref{nonasystdet}$ and $\lvert\lvert\bar{u}(\cdot,0)\rvert\rvert_A\leq \delta_0$, the solution of $\eqref{nonasyst}$ with $u_0=\bar{u}(\cdot,0)$ satisfies
\begin{equation*}
\mathbb{P}\Big(\tau_{\mathcal{B}(h)}<T\Big)\leq C_{\frac{h^2}{q_* \sigma^2}} (\kappa,T,\epsilon,s)\textnormal{exp}\Bigg\{-\kappa\dfrac{h^2}{q_*\sigma^2}\bigg(1-O\Big(\dfrac{h}{\epsilon^\nu}\Big)\bigg)\Bigg\}\;\;,
\end{equation*}
for any $0\leq T \leq \tau$.
\end{thm}

The inequality obtained in the previous theorem does not require $h^2\gg q_* \sigma^2$ but such assumption simplifies the dependence of $C_{\frac{h^2}{q_* \sigma^2}}(\kappa,T,\epsilon,s)$ on $\dfrac{h^2}{q_* \sigma^2}$, otherwise nontrivial, and enables further study on the moments of the exit-time.\\
The proof of such error estimate is based on \cite[Proposition 2.4]{berglund2022stochastic} which is divided in the study of the linear problem  obtained from $\eqref{nonasyst}$ and the extension of the results to the nonlinear cases.
\paragraph{Preliminaries}\mbox{}\\
We define $\psi:=u-\bar{u}$ for $u$ solution of $\eqref{nonasyst}$ and $\bar{u}$ solution of $\eqref{nonasystdet}$ with initial conditions $\bar{u}(\cdot,0)=u(\cdot,0)=u_0$. Then $\psi$ is solution of
\begin{equation}\label{nonasystpsi}\begin{cases}
\txtd \psi(x,t)=\dfrac{1}{\epsilon}\big(\Delta \psi(x,t) + \alpha(t) \psi(x,t) - g(x) \psi(x,t) -(u(x,t)^3-\bar{u}(x,t)^3)\big) ~\txtd t +\dfrac{\sigma}{\sqrt\epsilon} ~\txtd W_t\\
\psi(x,0)=0\;\;\;\;,\;\;\;\;\forall x\in\mathcal{O}\\
\left.\psi(\cdot,t)\right|_{\partial\mathcal{O}}=0\;\;\;\;,\;\;\;\;\forall t\geq0\;\;.
\end{cases}\end{equation}
The first equation of the system $\eqref{nonasystpsi}$ is equivalent, by Taylor's formula, to
\begin{equation*}
\txtd \psi(x,t)=\Big(\Delta \psi(x,t) + \alpha(\epsilon t) \psi(x,t) - g(x) \psi(x,t) -3\bar{u}(x,t)^2\psi(x,t) +G(\psi(x,t))\Big)~\txtd t +\sigma ~ \txtd W_t,
\end{equation*}
with
\begin{equation}\label{defb}
G(\psi):=\left.\dfrac{1}{2}\partial_\phi^2 \Big(-\phi^3\Big)\right|_{\bar{u}+\eta \psi}\psi^2=-3(\bar{u}+\eta \psi)\psi^2
\end{equation}
for a certain $0<\eta<1$.

\subsection{Linear case}
Suppose $\psi_0$ is a solution of
\begin{equation}\label{nonasystlin}\begin{cases}
\txtd \psi_0(x,t)=\dfrac{1}{\epsilon}(\Delta \psi_0(x,t) + \alpha(t) \psi_0(x,t) - g(x) \psi_0(x,t)) ~\txtd t +\dfrac{\sigma}{\sqrt\epsilon} ~\txtd W_t\\
\psi_0(\cdot,0)\equiv 0\in \mathcal{H}\\
\left.\psi_0(\cdot,t)\right|_{\partial\mathcal{O}}=0\;\;\;\;,\;\;\;\;\forall t\geq0\;\;.
\end{cases}\end{equation}
Then $\psi_k:=\langle \psi_0,e_k\rangle $ satisfies
\begin{equation}\label{prelm2}\begin{split}
& \txtd \psi_k(t)=\dfrac{1}{\epsilon}(-\lambda_k \psi_k(t) + \alpha(t) \psi_k(t))~\txtd t +\dfrac{\sigma}{\sqrt\epsilon} ~\txtd \langle W_t,e_k\rangle \\
&= \dfrac{1}{\epsilon}(-\lambda_k \psi_k(t) + \alpha(t) \psi_k(t))~\txtd t +\dfrac{\sigma}{\sqrt\epsilon}~ \txtd \Big(\sum_{j=1}^\infty \sqrt{q_j} \beta_j(t) \langle b_j,e_k\rangle \Big)\;\;,
\end{split}\end{equation}
for a family of independent Wiener processes $\{\beta_j\}_{j\in\mathbb{N}\setminus\{0\}}$. We can now prove the following lemmas.

\begin{lm}\label{estlm1}
There exists a constant $c_0>0$ so that 
\begin{equation*}
\textnormal{Var}(\psi_k(t))\leq c_0 q_* \dfrac{\sigma^2}{\lambda_k}
\end{equation*}
for $\textnormal{Var}$ that indicates the variance, $q_*:=\sup_j \{q_j\}$, all $0\leq t\leq \tau$ and $k\in\mathbb{N}\setminus\{0\}$.
\end{lm}
\begin{proof}
By definition,
\begin{equation*}\begin{split}
&\text{Var}(\psi_k(t))=\dfrac{\sigma^2}{\epsilon}\int_0^t \Big\langle e_k, \textnormal{e}^{\frac{1}{\epsilon}\big(A(t-t_1)+\int_{t_1}^t \alpha(t_2)\txtd t_2\big)} Q \textnormal{e}^{\frac{1}{\epsilon}\big(A(t-t_1)+\int_{t_1}^t \alpha(t_2)\txtd t_2\big)} e_k\Big\rangle  \txtd t_1 \\
&=\dfrac{\sigma^2}{\epsilon}\int_0^t \Big\langle \textnormal{e}^{\frac{1}{\epsilon}\big(-\lambda_k(t-t_1)+\int_{t_1}^t \alpha(t_2) \txtd t_2\big)}e_k, Q \textnormal{e}^{\frac{1}{\epsilon}\big(-\lambda_k(t-t_1)+\int_{t_1}^t \alpha(t_2) \txtd t_2\big)} e_k\Big\rangle  \txtd t_1\\
&=\dfrac{\sigma^2}{\epsilon}\int_0^t \textnormal{e}^{\frac{2}{\epsilon}\big(-\lambda_k(t-t_1)+\int_{t_1}^t \alpha(t_2) \txtd t_2\big)} \langle e_k,Q e_k\rangle  \txtd t_1
\leq \dfrac{\sigma^2}{\epsilon} q_*\int_0^t \textnormal{e}^{\frac{2}{\epsilon}\big(-\lambda_k+\underset{{0\leq t_2\leq\tau}}{\max}\{\alpha(t_2)\}\big)(t-t_1)} \txtd t_1\\
& = q_*\dfrac{\sigma^2}{2 \big(\lambda_k-\underset{{0\leq t_1\leq \tau}}{\max}\{\alpha(t_1)\}\big)}\Big(1-\textnormal{e}^{\frac{2}{\epsilon}\big(-\lambda_k+\underset{{0\leq t_1\leq \tau}}{\max}\{\alpha(t_1)\}\big)t}\Big)
\leq c_0 q_* \dfrac{\sigma^2}{\lambda_k}\;\;,
\end{split}\end{equation*}
for $0\leq t \leq \tau$ and $k\in\mathbb{N}\setminus\{0\}$.
\end{proof}

The subsequent lemma relies on methods presented in \cite{berglund2022stochastic} and part of the proof follows \cite[Theorem 2.4]{berglund2002pathwise}. The generality in noise requires additional steps, due in particular to the fact that the eigenfunctions of $A$ do not diagonalize $Q$.\footnote{Such steps do not require property \ref{property1}. Property \ref{property2} can be removed through the addition of considerations for particular cases in which $\langle e_k,Q e_k\rangle=0$ for a certain $k\in\mathbb{N}\setminus\{0\}$.}

\begin{lm}\label{estlm2}
Given $c^+>0$ and $0 \leq T \leq \tau$ for which $\lambda_k-\underset{0\leq t \leq T}{\textnormal{min}} \{\alpha(t)\}\leq c^+ \lambda_k$ for all $k\in\mathbb{N}\setminus\{0\}$, then there exists a constant $\gamma_0>0$ that satisfies
\begin{equation}\label{single dim}
\mathbb{P}\bigg(\sup_{0\leq t\leq T}\lvert \psi_k(t) \rvert\geq h \bigg)\leq \dfrac{2 c^+ \lambda_k T}{\gamma \epsilon} \textnormal{exp}\bigg\{-\dfrac{\textnormal{e}^{-2\gamma}}{2 c_0} \lambda_k \dfrac{h^2}{q_* \sigma^2}\bigg\}
\end{equation}
for any $0<\gamma\leq\gamma_0$, for $q_*:=\underset{j}{\sup} \{q_j\}$ and $c_0$ obtained in Lemma \ref{estlm1}.
\end{lm}

\begin{proof}
For fixed $k\in\mathbb{N}\setminus\{0\}$, the solution $\psi_k$ of $\eqref{prelm2}$ is a Gaussian process. In fact, we have 
\begin{equation}\label{wiener}
\sum_{j=1}^\infty \sqrt{q_j} \langle b_j,e_k\rangle  \beta_j(t)=\sqrt{\sum_{j=1}^\infty q_j \langle b_j,e_k\rangle^2} \; \beta(t)
\end{equation}
for a Wiener process $\beta$ and $0\leq t\leq \tau$. The equation $\eqref{wiener}$ can be proven with the following considerations:
\begin{itemize}
\item For any $0\leq t\leq \tau$ and $n\in\mathbb{N}\setminus\{0\}$ we define $\Xi_n(t)=\sum_{j=1}^n \sqrt{q_j} \langle b_j,e_k\rangle \beta_j(t)$. Then, given the integers $n>m>0$,
\begin{equation*}
\mathbb{E}\big(\lvert \Xi_n(t)-\Xi_m(t)\rvert^2\big)=t \sum_{j=m+1}^n q_j \langle b_j,e_k\rangle^2\leq\tau\sup_j\{q_j\}=\tau q_*\;.
\end{equation*}
Hence $\Xi_n(t)$ converges to a random variable $\Xi(t)$ in $L^2(\Omega)$ for all $0 \leq t \leq \tau$.
\item It is clear that 
\begin{equation*}
\Xi(0)=0
\end{equation*}
and that the series has independent increments.
\item The time increments of $\Xi$ from $t_1$ to $t_2$ with $t_1<t_2 \leq \tau$ are normally distributed with mean $0$ and with variance $(t_2-t_1) \sum_{j=1}^\infty q_j \langle b_j,e_k\rangle ^2$. This can be proven through the Lèvy continuity theorem and the pointwise convergence in time of the characteristic functions of $\Xi_n(t)$.
\item The almost sure continuity of $\Xi$ is implied by the Borel-Cantelli Lemma and by the almost sure continuity of all elements in $\{\Xi_n\}_{n\in\mathbb{N}\setminus\{0\}}$.
\end{itemize}
From $\eqref{prelm2}$ and $\eqref{wiener}$ we can state that $\psi_k$ is represented by Duhamel's formula,
\begin{equation*}
\psi_k(t)=\dfrac{\sigma}{\sqrt\epsilon}\sqrt{\sum_{j=1}^\infty q_j \langle b_j,e_k\rangle^2}\int_0^t \textnormal{e}^{\frac{1}{\epsilon}\big(-\lambda_k(t-t_1)+\int_{t_1}^t \alpha(t_2) \txtd t_2\big)} \txtd\beta(t_1)\;\;.
\end{equation*}
Due to the dependence on time $t$ of $-\lambda_k(t-t_1)+\int_{t_1}^t\alpha(t_2) \txtd t_2$, $\psi_k$ is not a martingale. Following \cite[Proposition 3.1.5]{berglund2006noise} we will study the martingale $\textnormal{e}^{\frac{1}{\epsilon}\big(\lambda_k t-\int_0^t\alpha(t_1)\txtd t_1 \big)}\psi_k$ and split the time interval in $0=s_0<s_1<...<s_N=T$. For convenience, we define for all $0\leq t\leq T$, $D(t):=-\lambda_k t+\int_0^t\alpha(t_1) \txtd t_1$. Hence
\begin{equation*}\begin{split}
&\mathbb{P}\bigg(\sup_{0\leq t\leq T}\lvert \psi_k(t) \rvert\geq h \bigg)\\
&=\mathbb{P}\bigg(\sup_{0\leq t\leq T}\Bigg\lvert \dfrac{\sigma}{\sqrt\epsilon}\sqrt{\sum_{j=1}^\infty q_j \langle b_j,e_k\rangle^2}\int_0^t \textnormal{e}^{\frac{1}{\epsilon}\big(-\lambda_k(t-t_1)+\int_{t_1}^t \alpha(t_2)\txtd t_2\big)} \txtd\beta(t_1) \Bigg\rvert\geq h \bigg)\\
&=\mathbb{P}\bigg(\exists j\in\{1,...,N\}: \sup_{s_{j-1}<t<s_j}\Bigg\lvert \dfrac{\sigma}{\sqrt\epsilon}\langle e_k,Q e_k\rangle^\frac{1}{2}\int_0^t \textnormal{e}^{\frac{1}{\epsilon}\big(D(t)-D(t_1)\big)} \txtd\beta(t_1) \Bigg\rvert\geq h \bigg)
\end{split}\end{equation*}
\begin{equation*}\begin{split}
&\leq\sum_{j=1}^N \mathbb{P}\bigg(
\dfrac{\sigma}{\sqrt\epsilon}\langle e_k,Q e_k\rangle^\frac{1}{2} \sup_{s_{j-1}<t<s_j}\Bigg\lvert \int_0^t \textnormal{e}^{\frac{1}{\epsilon}\big(D(t)-D(t_1)\big)} \txtd\beta(t_1) \Bigg\rvert\geq h \bigg)\\
&=2\sum_{j=1}^N \mathbb{P}\bigg(
\dfrac{\sigma}{\sqrt\epsilon}\langle e_k,Q e_k\rangle^\frac{1}{2} \sup_{s_{j-1}<t<s_j} \int_0^t \textnormal{e}^{\frac{1}{\epsilon}\big(D(t)-D(t_1)\big)} \txtd\beta(t_1) \geq h \bigg)\\
&\leq 2\sum_{j=1}^N \mathbb{P}\bigg(\langle e_k,Q e_k\rangle^{\frac{1}{2}}
 \sup_{s_{j-1}<t<s_j} \int_0^t \textnormal{e}^{-\frac{1}{\epsilon}D(t_1)} \txtd\beta(t_1) \geq \inf_{s_{j-1}<t<s_j} \textnormal{e}^{-\frac{1}{\epsilon}D(t)} \dfrac{\sqrt\epsilon}{\sigma} h \bigg)\\
 &=2\sum_{j=1}^N \mathbb{P}\bigg(\langle e_k,Q e_k\rangle^{\frac{1}{2}}
 \sup_{s_{j-1}<t<s_j} \int_0^t \textnormal{e}^{-\frac{1}{\epsilon}D(t_1)} \txtd\beta(t_1) \geq \textnormal{e}^{-\frac{1}{\epsilon}D(s_{j-1})} \dfrac{\sqrt\epsilon}{\sigma} h \bigg)\;\;.
\end{split}\end{equation*}
We can now apply a Bernstein-type inequality (\cite{berglund2006noise}). Therefore,

\begin{equation*}\begin{split}
& \mathbb{P}\bigg(\sup_{0\leq t\leq T}\lvert \psi_k(t) \rvert\geq h \bigg)\leq
2\sum_{j=1}^N \mathbb{P}\bigg(
 \sup_{0<t<s_j} \int_0^t \textnormal{e}^{-\frac{1}{\epsilon}D(t_1)} \txtd\beta(t_1) \geq \; \langle e_k,Q e_k\rangle^{-\frac{1}{2}} \textnormal{e}^{-\frac{1}{\epsilon}D(s_{j-1})} \dfrac{\sqrt\epsilon}{\sigma} h \bigg)\\
& \leq 2\sum_{j=1}^N \text{exp}\bigg\{- \dfrac{\epsilon h^2 \textnormal{e}^{-\frac{2}{\epsilon}D(s_{j-1})}}{2 \sigma^2 \langle e_k,Q e_k\rangle \int_0^{s_j} \textnormal{e}^{-\frac{2}{\epsilon}D(t_1)}\txtd t_1} \bigg\}=
2\sum_{j=1}^N \text{exp}\bigg\{- \dfrac{\epsilon h^2 \textnormal{e}^{\frac{2}{\epsilon}(D(s_j)-D(s_{j-1}))}}{2 \sigma^2 \langle e_k,Q e_k\rangle \int_0^{s_j} \textnormal{e}^{\frac{2}{\epsilon}(D(s_j)-D(t_1))}\txtd t_1} \bigg\}\\
& =2\sum_{j=1}^N \text{exp}\bigg\{- \dfrac{ h^2 \textnormal{e}^{\frac{2}{\epsilon}(D(s_j)-D(s_{j-1}))}}{2\; \text{Var}(\psi_k(s_j))} \bigg\}
\leq 2\sum_{j=1}^N \text{exp}\bigg\{- \dfrac{ h^2 \lambda_k \textnormal{e}^{\frac{2}{\epsilon}(D(s_j)-D(s_{j-1}))}}{2 c_0 q_* \sigma^2} \bigg\}\;\;,
\end{split}\end{equation*}
for which we used Lemma \ref{estlm1} in the last inequality. By assumption we know that $-\lambda_k+\alpha(t)<0$ for all $k\in\mathbb{N}\setminus\{0\}$ and $0\leq t\leq T$. We can set the sequence of times $\{s_j\}_{j=0}^N$ so that there exists $\gamma_0>0$ for which $\dfrac{\lambda_k T-\int_0^T \alpha(t_1)\txtd t_1}{\gamma_0\epsilon}=N\in\mathbb{N}\setminus\{0\}$ and
\begin{equation*}
-\lambda_k(s_{j+1}-s_j)+\int_{s_j}^{s_{j+1}}\alpha(t_1)\txtd t_1=-\gamma_0 \epsilon\;\;\;\;\;\;\forall j\in\{0,...,N-1\}\;.
\end{equation*}
In conclusion, for $\gamma_0>0$ small enough, such choice leads to

\begin{equation*}
\mathbb{P}\bigg(\sup_{0\leq t\leq T}\lvert \psi_k(t) \rvert\geq h \bigg)
\leq 2 \dfrac{\lambda_k T-\int_0^T \alpha(t_1)\txtd t_1}{\gamma_0\epsilon} \text{exp}\bigg\{- \dfrac{ h^2 \lambda_k \textnormal{e}^{-2\gamma_0}}{2 c_0 q_* \sigma^2} \bigg\}
\leq \dfrac{2 c^+ \lambda_k T}{\gamma \epsilon} \text{exp}\bigg\{- \dfrac{ h^2 \lambda_k \textnormal{e}^{-2\gamma}}{2 c_0 q_* \sigma^2} \bigg\}
\end{equation*}
for any $\gamma_0\geq \gamma$.

\end{proof}
The inequality $\eqref{single dim}$ can be treated as follows:
\begin{equation}\label{control}
\begin{split}
&\mathbb{P}\bigg(\sup_{0\leq t\leq T}\lvert \psi_k(t) \rvert\geq h \bigg)
\leq \dfrac{2 c^+ \lambda_k T}{\gamma \epsilon} \text{exp}\bigg\{- \dfrac{ h^2 \lambda_k \textnormal{e}^{-2\gamma}}{2 c_0 q_* \sigma^2}\bigg\}\\
&\leq \dfrac{2 c^+ \lambda_k T}{\gamma \epsilon} \text{exp}\bigg\{- \dfrac{ h^2 \lambda_k}{2 c_0 q_* \sigma^2}\bigg\}\text{exp}\bigg\{\dfrac{ h^2 \lambda_k \gamma}{c_0 q_* \sigma^2}\bigg\}\;,
\end{split}
\end{equation}
for assumptions equivalent to the previous lemma. Moreover, for $\dfrac{h^2}{q_* \sigma^2}$ big enough, $\eqref{control}$ can be optimized on $\gamma$ at $\gamma=\dfrac{c_0 q_* \sigma^2}{h^2 \lambda_k}$, resulting in
\begin{equation}\label{better}
\mathbb{P}\bigg(\sup_{0\leq t\leq T}\lvert \psi_k(t) \rvert\geq h \bigg)
\leq \dfrac{h^2}{q_* \sigma^2}\dfrac{2 c^+ \lambda_k^2 e T}{c_0 \epsilon} \text{exp}\bigg\{- \dfrac{h^2}{q_* \sigma^2}\dfrac{\lambda_k}{2 c_0}\bigg\}\;.
\end{equation}
The previous lemmas are sufficient to prove the subsequent theorem whose proof follows the steps of \cite[Theorem 2.4 in the linear case]{berglund2022stochastic}. 

\begin{thm}\label{nonathmlin}
For any $0<s<\dfrac{1}{2}$, $0<\epsilon$, $q_*\sigma^2\ll h^2$, $0 \leq T \leq \tau$ there exist the constants $0<\kappa(s), C_{\frac{h^2}{q_* \sigma^2}}(\kappa,T,\epsilon,s)$ such that the solution $\psi_0$ of $\eqref{nonasystlin}$ satisfies
\begin{equation*}
\mathbb{P}\bigg(\sup_{0\leq t\leq T}\lvert\lvert \psi_0(\cdot,t) \rvert\rvert_{A^s} \geq h \bigg)
\leq C_{\frac{h^2}{q_* \sigma^2}}(\kappa,T,\epsilon,s)\textnormal{e}^{-\kappa(s)\frac{h^2}{q_*\sigma^2}}\;\;.
\end{equation*}
\end{thm}
\begin{proof}
Set $\eta,\rho>0$ so that $\rho=\dfrac{1}{2}-s$. Then, for any sequence $\{h_k\}_{k\in\mathbb{N}\setminus\{0\}}\subset\mathbb{R}_{>0}$ that satisfies $h^2=\sum_{k=1}^\infty h_k^2$,
\begin{equation*}
\begin{split}
&\mathbb{P}\bigg\{\sup_{0\leq t\leq T}\lvert\lvert\psi_0(\cdot,t)\rvert\rvert_{A^s}\geq h\bigg\}
=\mathbb{P}\bigg\{\sup_{0\leq t\leq T}\lvert\lvert\psi_0(\cdot,t)\rvert\rvert_{A^s}^2\geq h^2\bigg\}
=\mathbb{P}\bigg\{\sup_{0\leq t\leq T}\sum_{k=1}^\infty \lambda_k^s \lvert\psi_k(t)\rvert^2\geq h^2\bigg\}\\
&\leq \sum_{k=1}^\infty \mathbb{P}\bigg\{\sup_{0\leq t\leq T} \lvert\psi_k(t)\rvert^2\geq \dfrac{h_k^2}{\lambda_k^s}\bigg\}
=\sum_{k=1}^\infty \mathbb{P}\bigg\{\sup_{0\leq t\leq T} \lvert\psi_k(t)\rvert \geq \dfrac{h_k}{\lambda_k^\frac{s}{2}}\bigg\}\\
&\leq \sum_{k=1}^\infty \dfrac{h_k^2}{q_* \sigma^2}\dfrac{2 c^+ \lambda_k^{2-s} e T}{c_0 \epsilon} \text{exp}\bigg\{- \dfrac{h_k^2}{q_* \sigma^2}\dfrac{\lambda_k^{1-s}}{2 c_0}\bigg\}\;,
\end{split}
\end{equation*}
for which we used $\eqref{better}$ in the last inequality. We can assume $h_k=C(\eta,s)h^2\lambda_k^{-1+s+\frac{\eta}{2}}$ with
\begin{equation*}
C(\eta,s):=\dfrac{1}{\sum_{k=1}^\infty \lambda_k^{-1+s+\frac{\eta}{2}}}\;.
\end{equation*}
Since $\lambda_k\sim k^2$ from $\eqref{sync}$, we can use Riemann Zeta function $\xi(\nu):=\sum_{k=1}^\infty k^{-\nu}$ to prove that $C(\eta,s)>0$ for any $\eta>0$. In fact, we have $\xi(2-2s-\eta)<\infty$ for any $0<\eta<2\rho$ and
\begin{equation*}
C(\eta,s)=\dfrac{1}{\sum_{k=1}^\infty \lambda_k^{-1+s+\frac{\eta}{2}}}
\propto \dfrac{1}{\sum_{k=1}^\infty k^{-2+2s+\eta}}
=\xi(2-2s-\eta)^{-1}\;,
\end{equation*}
for the symbol $\propto$ that indicates proportionality of the terms for a constant dependent on $g$. For $0<\eta<2\rho$ we can write 
\begin{equation*}
\begin{split}
&\mathbb{P}\bigg\{\sup_{0\leq t\leq T} \lvert\lvert\psi_0(\cdot,t)\rvert\rvert_{A^s}\geq h\bigg\}
\leq \sum_{k=1}^\infty \dfrac{h_k^2}{q_* \sigma^2}\dfrac{2 c^+ \lambda_k^{2-s} e T}{c_0 \epsilon} \text{exp}\bigg\{- \dfrac{h_k^2}{q_* \sigma^2}\dfrac{\lambda_k^{1-s}}{2 c_0}\bigg\}\\
&\leq \dfrac{h^2}{q_* \sigma^2} \dfrac{2 c^+ e T}{c_0 \epsilon} \sum_{k=1}^\infty \lambda_k^{2} \text{exp}\bigg\{- \dfrac{h^2}{q_* \sigma^2}\dfrac{C(\eta,s) \lambda_k^{\frac{\eta}{2}}}{2 c_0}\bigg\}
\leq \ell_T \sum_{k=1}^\infty (1+k^2)^2 \textnormal{e}^{-\ell (1+k^2)^{\frac{\eta}{2}}}\; 
\end{split}
\end{equation*}
for $\ell_T\propto \dfrac{h^2}{q_* \sigma^2} \dfrac{c^+ e T}{c_0 \epsilon}$ and $\ell\propto \dfrac{h^2}{q_* \sigma^2}\dfrac{C(\eta,s)}{c_0}$.
We define 
\begin{equation*}
\iota(x):=(1+x^2)^2\textnormal{e}^{-\ell (1+x^2)^\frac{\eta}{2}}
\end{equation*}
for which we know that
\begin{equation*}
\dfrac{\txtd}{\txtd x}\iota(x)=2x(1+x^2)\Big(2-\ell\dfrac{\eta}{2}(1+x^2)^\frac{\eta}{2}\Big)\textnormal{e}^{-\ell (1+x^2)^\frac{\eta}{2}}\;.
\end{equation*}
We can therefore bound
\begin{equation*}
\sum_{k=1}^\infty \iota(k)\leq \int_0^\infty \iota(x) \txtd x
\end{equation*}
by assuming $\iota$ decreasing in $[0,\infty)$. Such case is satisfied for $\dfrac{h^2}{q_* \sigma^2}$ larger than a constant of order $1$ dependent on the choice of $\eta$ and on $s$. Conversely, the theorem would be trivially proven.\\
\medskip
Setting $x'=\ell(1+x^2)^\frac{\eta}{2}$ and $y=-\ell+x'$ we can state
\begin{equation*}
\begin{split}
&\int_0^\infty \iota(x) \txtd x
=\int_0^\infty (1+x^2)^2\textnormal{e}^{-\ell (1+x^2)^\frac{\eta}{2}} \txtd x
=\dfrac{1}{\ell^\frac{6}{\eta}\eta}\int_\ell^\infty  \dfrac{x'^{\frac{6}{\eta}-1}}{\sqrt{\Big(\frac{x'}{\ell}\Big)^\frac{2}{\eta}-1}} \textnormal{e}^{-{x'}} \txtd x'\\
&=\dfrac{\textnormal{e}^{-\ell}}{\ell^\frac{6}{\eta}\eta}\int_0^\infty  \dfrac{(\ell+y)^{\frac{6}{\eta}-1}}{\sqrt{\Big(1+\frac{y}{\ell}\Big)^\frac{2}{\eta}-1}} \textnormal{e}^{-y} \txtd y
\leq \dfrac{\textnormal{e}^{-\ell}}{\ell^\frac{6}{\eta}\eta}\int_0^\infty  \dfrac{(\ell+y)^{\frac{6}{\eta}-1}}{\sqrt{\frac{2 y}{\eta \ell}}} \textnormal{e}^{-y} \txtd y\\
&=\dfrac{\textnormal{e}^{-\ell}}{\ell^\frac{6}{\eta}\eta}\sqrt{\frac{\eta \ell}{2}}\int_0^\ell  \dfrac{(\ell+y)^{\frac{6}{\eta}-1}}{\sqrt{y}} \textnormal{e}^{-y} \txtd y
+\dfrac{\textnormal{e}^{-\ell}}{\ell^\frac{6}{\eta}\eta}\sqrt{\frac{\eta \ell}{2}}\int_\ell^\infty  \dfrac{(\ell+y)^{\frac{6}{\eta}-1}}{\sqrt{y}} \textnormal{e}^{-y} \txtd y\\
&\leq \dfrac{\textnormal{e}^{-\ell}}{\ell^\frac{6}{\eta}\eta} (2\ell)^{\frac{6}{\eta}-1} \sqrt{\frac{\eta \ell}{2}} \int_0^\ell  \dfrac{1}{\sqrt{y}} \textnormal{e}^{-y} \txtd y
+\dfrac{\textnormal{e}^{-\ell}}{\ell^\frac{6}{\eta}\eta} \sqrt{\frac{\eta \ell}{2}} \int_\ell^\infty  \dfrac{(2y)^{\frac{6}{\eta}-1}}{\sqrt{y}} \textnormal{e}^{-y} \txtd y\\
&\leq c_1(\eta) \ell^{-\frac{1}{2}}\textnormal{e}^{-\ell}+c_2(\eta)\ell^{-\frac{6}{\eta}+\frac{1}{2}}\textnormal{e}^{-\ell}\;,
\end{split}
\end{equation*}
for which we used $\Big(1+\frac{y}{\ell}\Big)^\frac{2}{\eta}-1\geq \frac{2 y}{\eta \ell}$ and took the constants $c_1(\eta),c_2(\eta)>0$ that are uniformly bounded in $\ell$ since $\eta<12$. Such results lead to
\begin{equation*}
\begin{split}
&\mathbb{P}\bigg\{\sup_{0\leq t\leq T} \lvert\lvert\psi_0(\cdot,t)\rvert\rvert_{A^s}\geq h\bigg\}
\leq \ell_T \sum_{k=1}^\infty f(k)
\leq \ell_T \Big( c_1(\eta) \ell^{-\frac{1}{2}}\textnormal{e}^{-\ell}+c_2(\eta)\ell^{-\frac{6}{\eta}+\frac{1}{2}}\textnormal{e}^{-\ell} \Big)\\
&\propto \dfrac{h^2}{q_* \sigma^2} \dfrac{c^+ T}{c_0 \epsilon} \Big( c_1(\eta) \Big(\dfrac{h^2}{q_* \sigma^2}\dfrac{C(\eta,s)}{c_0}\Big)^{-\frac{1}{2}} +c_2(\eta)\Big(\dfrac{h^2}{q_* \sigma^2}\dfrac{C(\eta,s)}{c_0}\Big)^{-\frac{6}{\eta}+\frac{1}{2}} \Big)
\textnormal{e}^{-\ell}\;.
\end{split}
\end{equation*}
The proof is concluded assuming $\eta=\rho=\dfrac{1}{2}-s$, so that $C(\eta,s)\propto\xi\Big(\dfrac{3}{2}-s\Big)^{-1}$, $\kappa(s)\propto\dfrac{C(\eta,s)}{c_0}$ and 
\begin{equation*}
C_{\frac{h^2}{q_* \sigma^2}}(\kappa,T,\epsilon,s)
\propto \dfrac{h^2}{q_* \sigma^2} \dfrac{c^+ T}{c_0 \epsilon} \Big( c_1(\eta) \Big(\dfrac{h^2}{q_* \sigma^2}\dfrac{C(\eta,s)}{c_0}\Big)^{-\frac{1}{2}} +c_2(\eta)\Big(\dfrac{h^2}{q_* \sigma^2}\dfrac{C(\eta,s)}{c_0}\Big)^{-\frac{6}{\eta}+\frac{1}{2}} \Big)\;.
\end{equation*}

\end{proof}

\paragraph{Remark.} The dependence of $C_{\frac{h^2}{q_* \sigma^2}}(\kappa,\tau,\epsilon,s)$ on $\tau$ and $\epsilon$ is well known. In fact, $C_{\frac{h^2}{q_* \sigma^2}}(\kappa,\tau,\epsilon,s)\sim\dfrac{\tau}{\epsilon}$ is a property that we use to obtain moment estimates further in the paper. Such proportionality is inherited from $\eqref{better}$. Additionally, for $q_* \sigma^2\ll h^2 $, the relation $C_{\frac{h^2}{q_* \sigma^2}}(\kappa,\tau,\epsilon,s)\sim \Big(\dfrac{h^2}{q_* \sigma^2}\Big)^{\frac{1}{2}}$ holds. Furthermore, the constants $C_{\frac{h^2}{q_* \sigma^2}}(\kappa,\tau,\epsilon,s)$ and $\kappa$ depend on the choice of $g$ and $\alpha$ from construction and the definition of $c_0$ in Lemma \ref{estlm1}.
\medskip

We study now the error estimate given by $\psi_*$, solution of
\begin{equation}\label{nonasystlin1}\begin{cases}
\txtd \psi_*(x,t)=\dfrac{1}{\epsilon}(\Delta \psi_*(x,t) + \alpha(t) \psi_*(x,t) - g(x) \psi_*(x,t)-\bar{u}(x,t)^2\psi_*(x,t)) ~\txtd t +\dfrac{\sigma}{\sqrt\epsilon} ~\txtd W_t\\
\psi_*(x,0)=0\;\;\;\;,\;\;\;\;\forall x\in\mathcal{O}\\
\left.\psi_*(\cdot,t)\right|_{\partial\mathcal{O}}=0\;\;\;\;,\;\;\;\;\forall t\geq0\;\;,
\end{cases}\end{equation}
for the function $\bar{u}$ assumed in the construction of $\psi$ in \eqref{nonasystpsi}. In particular, $\bar{u}$ satisfies the hypothesis in Proposition \ref{propdet}. The proof of the next corollary is inspired by \cite[Proposition 3.7]{berglund2002pathwise} and makes use of the Young-type inequality proven in Lemma $\ref{young-lm}$.

\begin{cor}\label{nonacorlin}
For any $0<s<\dfrac{1}{2}$ and $0<\epsilon$, $q_*\sigma^2\ll h^2$, there exist $M,\delta_0>0$ such that for any $\delta\leq\delta_0$ and $\bar{u}$ solution of $\eqref{nonasystdet}$ for which $\lvert\lvert\bar{u}(\cdot,0)\rvert\rvert_A\leq\delta$, the constants $0<\kappa(s), C_{\frac{h^2}{q_* \sigma^2}} (\kappa,\tau,\epsilon,s)$ obtained in Theorem \ref{nonathmlin} satisfy
\begin{equation}\label{nonacorlin-ine}
\mathbb{P}\bigg(\sup_{0\leq t\leq T}\lvert\lvert \psi_*(\cdot,t) \rvert\rvert_{A^s} \geq M h \bigg)
\leq C_{\frac{h^2}{q_* \sigma^2}}(\kappa,\tau,\epsilon,s)\textnormal{e}^{-\kappa\frac{h^2}{q_* \sigma^2}}
\end{equation}
for the solution $\psi_*$ of $\eqref{nonasystlin1}$ and any $0 \leq T \leq \tau$.
\end{cor}
\begin{proof}
Define, for $M>0$,
\begin{equation*}
\tilde{\tau}=\inf \{0\leq t\;\;\;\text{s.t.}\;\;\lvert\lvert \psi_*(\cdot,t)\rvert\rvert_{A^s}\geq M h\}\in[0,\tau]\cup\infty
\end{equation*}
and the event
\begin{equation*}
\tilde\Omega:=\bigg\{ \sup_{0\leq t\leq T}\lvert\lvert \psi_0(\cdot,t) \rvert\rvert_{A^s}\leq h\bigg \}\cap \{\tilde{\tau}<\infty\}\;\;.
\end{equation*}
Denote then the norm $\lvert\lvert\cdot\rvert\rvert_{A^s}$ norm as
\begin{equation*}
\lvert\lvert \phi \rvert\rvert_{A^s}^2=\langle (-A)^s \phi,\phi\rangle\;,
\end{equation*}
for $\phi\in \mathcal{V}^s$, the fractional Sobolev space of order $2$ and degree $0<s<\dfrac{1}{2}$ with Dirichlet conditions.\\
Taking $\psi_0$ the solution of $\eqref{nonasystlin}$, we can state from Duhamel's formula that
\begin{equation*}\begin{split}
&\psi_*(x,t)=\int_0^t \textnormal{e}^{\frac{1}{\epsilon}\big(A(t-t_1)+\int_{t_1}^t\alpha(t_2)\txtd t_2\big)} \txtd W_{t_1}-
\int_0^t \textnormal{e}^{\frac{1}{\epsilon}\big(A(t-t_1)+\int_{t_1}^t\alpha(t_2)\txtd t_2\big)} \bar{u}(x,t_1)^2 \psi_*(x,t_1) \txtd t_1\\
&=\psi_0(x,t)-
\int_0^t \textnormal{e}^{\frac{1}{\epsilon}\big(A(t-t_1)+\int_{t_1}^t\alpha(t_2)\txtd t_2\big)} \bar{u}(x,t_1)^2 \psi_*(x,t_1) \txtd t_1\;\;,
\end{split}\end{equation*}
for all $0 \leq t \leq \tau$ . Therefore, for $t<\tilde{\tau}$ and for $\omega\in \tilde\Omega$,
\begin{equation} \label{steps}\begin{split}
&\lvert\lvert \psi_*(\cdot,t) \rvert\rvert_{A^s}^2
=\langle (-A)^s\psi_*(\cdot,t),\psi_*(\cdot,t)\rangle\\
&\leq\langle (-A)^s\psi_0(\cdot,t),\psi_0(\cdot,t)\rangle +2\bigg\lvert\Big\langle (-A)^s\int_0^t \textnormal{e}^{\frac{1}{\epsilon}\big(A(t-t_1)+\int_{t_1}^t\alpha(t_2)\txtd t_2\big)} \bar{u}(\cdot,t_1)^2 \psi_*(\cdot,t_1) \txtd t_1,\psi_0(\cdot,t)\Big\rangle \bigg\rvert\\
&+\bigg\lvert\Big\langle (-A)^s \int_0^t \textnormal{e}^{\frac{1}{\epsilon}\big(A(t-t_1)+\int_{t_1}^t\alpha(t_2)\txtd t_2\big)} \bar{u}(\cdot,t_1)^2 \psi_*(\cdot,t_1) \txtd t_1, \int_0^t \textnormal{e}^{\frac{1}{\epsilon}\big(A(t-t'_1)+\int_{t'_1}^t\alpha(t'_2)\txtd t'_2\big)} \bar{u}(\cdot,t'_1)^2 \psi_*(\cdot,t'_1) \txtd t'_1\Big\rangle \bigg\rvert\\
&\leq \lvert\lvert \psi_0(\cdot,t) \rvert\rvert_{A^s}^2+2 \int_0^t \lvert\lvert \textnormal{e}^{\frac{1}{\epsilon}\big(A(t-t_1)+\int_{t_1}^t\alpha(t_2)\txtd t_2\big)} \bar{u}(\cdot,t_1)^2 \psi_*(\cdot,t_1)\rvert\rvert_{A^s} \lvert\lvert \psi_0(\cdot,t) \rvert\rvert_{A^s} \txtd t_1\\
&+\int_0^t \int_0^t \lvert\lvert \textnormal{e}^{\frac{1}{\epsilon}\big(A(t-t_1)+\int_{t_1}^t\alpha(t_2)\txtd t_2\big)} \bar{u}(\cdot,t_1)^2 \psi_*(\cdot,t_1)\rvert\rvert_{A^s}\; \lvert\lvert \textnormal{e}^{\frac{1}{\epsilon}\big(A(t-t'_1)+\int_{t'_1}^t\alpha(t'_2)\txtd t'_2\big)} \bar{u}(\cdot,t'_1)^2 \psi_*(\cdot,t'_1)\rvert\rvert_{A^s} \txtd t_1\; \txtd t'_1\\
&\leq \lvert\lvert \psi_0(\cdot,t) \rvert\rvert_{A^s}^2 + 2 \int_0^t \textnormal{e}^{\frac{1}{\epsilon}\Big(-\lambda_1+\underset{0\leq t_2\leq t}{\max}\alpha(t_2)\Big)(t-t_1)} \lvert\lvert \bar{u}(\cdot,t_1)^2 \psi_*(\cdot,t_1) \rvert\rvert_{A^s} \lvert\lvert \psi_0(\cdot,t) \rvert\rvert_{A^s} \txtd t_1\\
&+\int_0^t \int_0^t \textnormal{e}^{\frac{1}{\epsilon}\Big(-\lambda_1+\underset{0\leq t_2\leq t}{\max}\alpha(t_2)\Big)(t-t_1)} \lvert\lvert \bar{u}(\cdot,t_1)^2 \psi_*(\cdot,t_1)\rvert\rvert_{A^s} \textnormal{e}^{\frac{1}{\epsilon}\Big(-\lambda_1+\underset{0\leq t_2'\leq t}{\max}\alpha(t_2')\Big)(t-t_1')} \lvert\lvert \bar{u}(\cdot,t'_1)^2 \psi_*(\cdot,t'_1)\rvert\rvert_{A^s} \txtd t_1\; \txtd t'_1
\end{split}\end{equation}
for which we used Cauchy-Schwarz inequality and the min-max principle on $\lvert\lvert\cdot\rvert\rvert_{A^s}$ .
From the Young-type inequality in Lemma $\ref{young-lm}$ and the control over the deterministic solution in Propostion $\ref{propdet}$ we obtain
\begin{equation} \label{young-in-lin}
\begin{split}
&\lvert\lvert \bar{u}(\cdot,t_1)^2 \psi_*(\cdot,t_1)\rvert\rvert_{A^s}\leq 
c(r,s)\lvert\lvert \bar{u}(\cdot,t_1)^2\rvert\rvert_{A} \lvert\lvert \psi_*(\cdot,t_1)\rvert\rvert_{A^s}\\
&\leq c(r,s)\lvert\lvert \bar{u}(\cdot,t_1)\rvert\rvert_{A}^2 \lvert\lvert \psi_*(\cdot,t_1)\rvert\rvert_{A^s}
\leq M c(r,s)\delta^2 h\;,
\end{split}
\end{equation}
for any $\lvert\lvert \bar{u}(\cdot,0) \rvert\rvert_A\leq \delta$.
Labeling $\eta=\lambda_1-\underset{0\leq t_1 \leq T}{\max} \alpha(t_1)>0$, we imply from $\eqref{steps}$ and $\eqref{young-in-lin}$,
\begin{equation}\label{last-step}\begin{split}
\lvert\lvert \psi_*(\cdot,t) \rvert\rvert_{A^s}^2 
&\leq h^2 
+ 2 M c(r,s)\delta^2 h^2 \int_0^t \textnormal{e}^{-\frac{\eta}{\epsilon}(t-t_1)} \txtd t_1
+M^2 c(r,s)^2 \delta^4 h^2 \int_0^t \int_0^t \textnormal{e}^{-\frac{\eta}{\epsilon}(t-t_1)} \textnormal{e}^{-\frac{\eta}{\epsilon}(t-t_1')}  \txtd t_1\; \txtd t'_1\\
& \leq h^2 
+ 2 M \dfrac{\epsilon c(r,s)}{\eta} \delta^2 h^2
+ M^2 \dfrac{\epsilon^2 c(r,s)^2}{\eta^2} \delta^4 h^2
=\Big(1+M\dfrac{\epsilon c(r,s)}{\eta}\delta^2\Big)^2 h^2 \;.
\end{split}\end{equation}
For any $0<s<\dfrac{1}{2}$, we can always find a pair of parameters $\delta_0,M>0$ such that $M^2>\Big(1+M\dfrac{\epsilon c(r,s)}{\eta}\delta^2\Big)^2$ for any $\delta\leq\delta_0$. From the definition of $\tilde{\tau}$ it holds $\lvert\lvert \psi_*(\cdot,\tilde{\tau}) \rvert\rvert_{A^s}^2=M^2 h^2$ which is in contradiction with $\eqref{last-step}$, hence we infer that $\mathbb{P}(\tilde\Omega)=0$ and
\begin{equation*}
\mathbb{P}\bigg(\sup_{0\leq t\leq T}\lvert\lvert \psi_*(\cdot,t) \rvert\rvert_{A^s} \geq M h \bigg) 
\leq \mathbb{P}\bigg(\sup_{0\leq t\leq T}\lvert\lvert \psi_0(\cdot,t) \rvert\rvert_{A^s}\geq h \bigg)\;.
\end{equation*}
The inequality \eqref{nonacorlin-ine} is proven from Theorem $\ref{nonathmlin}$.
\end{proof}
\paragraph{Notation.}
For notation, we omit $M$ from $\eqref{nonacorlin-ine}$ since it can be absorbed in the construction of $C_{\frac{h^2}{q_* \sigma^2}}(\kappa,\tau,\epsilon,s)$ and $\kappa$. We note that the constants $M$ and $\delta_0$ are dependent on $g$ and $\alpha$.
\paragraph{Remark.}
The exponential decay of $\lvert\lvert\bar{u}(\cdot,t)\rvert\rvert_A$ proven in Proposition $\ref{propdet}$ can be used in $\eqref{young-in-lin}$ to achieve more freedom on the choice of $\delta$.

\subsection{Nonlinear case}
In order to study the estimate error for $\eqref{nonasystpsi}$ the following lemmas are required.

\begin{lm}
Set $0 \leq t \leq \tau$. For $G$ defined as in $\eqref{defb}$ and a function $\psi(\cdot,t)\in V^s$ with $\dfrac{1}{3}<s<\dfrac{1}{2}$, the inequality
\begin{equation}\label{tough}
\lvert\lvert G(\psi(\cdot,t))\rvert\rvert_{A^r}\leq c(r,s) \max\{\lvert\lvert\psi(\cdot,t)\rvert\rvert_{A^s}^2,\lvert\lvert\psi(\cdot,t)\rvert\rvert_{A^s}^{3}\}
\end{equation}
holds when $0<r<\dfrac{1}{2}-3\Big(\dfrac{1}{2}-s\Big)$ and for a certain $c(r,s)<\infty$.
\end{lm}
\begin{proof}
The current proof relies on Lemma $\ref{young-lm}$. By definition of the function $G$, $\eqref{defb}$, and Theorem \ref{propdet},
\begin{equation*}\begin{split}
&\lvert\lvert G(\psi(\cdot,t))\rvert\rvert_{A^r}\leq 
3\lvert\lvert \bar{u}(\cdot,t) \psi(\cdot,t)^2\rvert\rvert_{A^r}+
3\lvert\lvert \psi(\cdot,t)^3\rvert\rvert_{A^r}\\
&\leq c(r,s) \Big(\lvert\lvert \bar{u}(\cdot,t) \rvert\rvert_{A^s}\;\lvert\lvert  \psi(\cdot,t)\rvert\rvert_{A^s}^2+ \lvert\lvert  \psi(\cdot,t)\rvert\rvert_{A^s}^3\Big)
\leq c(r,s) \Big(\lvert\lvert \bar{u}(\cdot,t) \rvert\rvert_{A}\;\lvert\lvert  \psi(\cdot,t)\rvert\rvert_{A^s}^2+ \lvert\lvert  \psi(\cdot,t)\rvert\rvert_{A^s}^3\Big)\\
&\leq c(r,s) \Big(\lvert\lvert \bar{u}(\cdot,0) \rvert\rvert_{A}\;\lvert\lvert  \psi(\cdot,t)\rvert\rvert_{A^s}^2+ \lvert\lvert  \psi(\cdot,t)\rvert\rvert_{A^s}^3\Big)
\leq c(r,s) \max\{\lvert\lvert\psi(\cdot,t)\rvert\rvert_{A^s}^2,\lvert\lvert\psi(\cdot,t)\rvert\rvert_{A^s}^{3}\}\;.
\end{split}\end{equation*}
In the last inequality, we have rewritten for simplicity $(1+\delta)c(r,s)$ as $c(r,s)$ for $\lvert\lvert \bar{u} (\cdot,0)\rvert\rvert_A\leq \delta$\;.
\end{proof}

The next lemma can be proven following the same steps as \cite[Lemma 3.5, Corollary 3.6]{berglund2022stochastic}.

\begin{lm}
Given $0<r<\dfrac{1}{2}$ so that $G(\psi(\cdot,t))\in V^r$ for all $0 \leq t \leq \tau$ and $\psi$ solution of $\eqref{nonasystpsi}$ and set $q<r+2$, then there exists $M(r,q)<\infty$ that satisfies, for all $0 \leq t \leq \tau$,
\begin{equation}\label{easier}
\begin{split}
&\lvert\lvert\int_0^t \textnormal{e}^{\frac{1}{\epsilon}\big((A+\alpha(t))(t-t_1)-\int_{t_1}^t\bar{u}(\cdot,t_2)^2\txtd t_2\big)}G(\psi(\cdot,t_1))\txtd t_1\rvert\rvert_{A^q}\\
&=\lvert\lvert \psi(\cdot,t)-\psi_*(\cdot,t) \rvert\rvert_{A^q}\leq M(r,q) \epsilon^{\frac{q-r}{2}-1}\sup_{0\leq t_1\leq t}\lvert\lvert G(\psi(\cdot,t))\rvert\rvert_{A^r}\;\;,
\end{split}
\end{equation}
with $\psi_*$ solution of $\eqref{nonasystlin1}$.
\end{lm}
We can then prove Theorem $\ref{error estimate}$ according to a similar method to the one used in the proof of \cite[Theorem 2.4]{berglund2022stochastic}.

\begin{proof}[Proof Theorem $\ref{error estimate}$.]
Assume $h_1,h_2>0$ such that $h=h_1+h_2$ and set $\dfrac{1}{3}<s<\dfrac{1}{2}$. Then
\begin{equation*} \begin{split}
& \mathbb{P}\Big(\tau_{\mathcal{B}(h)}<T\Big)= \mathbb{P}\Big(\sup_{0\leq t \leq \min\{T,  \tau_{\mathcal{B}(h)}\}} \lvert\lvert \psi(\cdot,t) \rvert\rvert_{A^s}\geq h\Big)\\
&\leq \mathbb{P}\Big(\sup_{0\leq t \leq T} \lvert\lvert \psi_*(\cdot,t) \rvert\rvert_{A^s}+\sup_{0\leq t \leq \min\{T,  \tau_{\mathcal{B}(h)}\}} \lvert\lvert \psi(\cdot,t)-\psi_*(\cdot,t) \rvert\rvert_{A^s}\geq h\Big)\\
& \leq \mathbb{P}\Big(\sup_{0\leq t \leq T} \lvert\lvert \psi_*(\cdot,t) \rvert\rvert_{A^s}\geq h_1\Big)+ 
\mathbb{P}\Big(
\sup_{0\leq t \leq \min\{T,  \tau_{\mathcal{B}(h)}\}} \lvert\lvert \psi(\cdot,t)-\psi_*(\cdot,t) \rvert\rvert_{A^s}>h_2\Big)\;\;.
\end{split} \end{equation*}
For $t<\tau_{\mathcal{B}(h)}$ and since $h=O(1)$, the inequalities $\eqref{tough}$ and $\eqref{easier}$ imply
\begin{equation*}\begin{split}
&\lvert\lvert \psi(\cdot,t)-\psi_*(\cdot,t) \rvert\rvert_{A^q}\leq M(r,q) \epsilon^{\frac{q-r}{2}-1}\sup_{0\leq t_1\leq t}\lvert\lvert G(\psi(\cdot,t))\rvert\rvert_{A^r}\\
&\leq M(r,q) \epsilon^{\frac{q-r}{2}-1} c(r,s) \lvert\lvert\psi(\cdot,t)\rvert\rvert_{A^s}^2
\leq M(r,q) \epsilon^{\frac{q-r}{2}-1} c(r,s) h^2\;,
\end{split}\end{equation*}
for $0<r<\dfrac{1}{2}-3\Big(\dfrac{1}{2}-s\Big)$ and $0<q<r+2$. Assuming $q\geq s$ and setting $h_2=M(r,q) \epsilon^{\frac{q-r}{2}-1} c(r,s) h^2$ we obtain
\begin{equation*}
\mathbb{P}\Big(\sup_{0\leq t \leq \min\{T,  \tau_{\mathcal{B}(h)}\}} \lvert\lvert \psi(\cdot,t)-\psi_*(\cdot,t) \rvert\rvert_{A^s}>h_2\Big)
\leq \mathbb{P}\Big(\sup_{0\leq t \leq \min\{T,  \tau_{\mathcal{B}(h)}\}} \lvert\lvert \psi(\cdot,t)-\psi_*(\cdot,t) \rvert\rvert_{A^q}>h_2\Big)=0\;.
\end{equation*}
We can then estimate $\mathbb{P}\Big(\tau_{\mathcal{B}(h)}<T\Big)\leq \mathbb{P}\Big(\underset{0\leq t\leq T}{\sup} \lvert\lvert \psi_*(\cdot,t) \rvert\rvert_{A^s}\geq h_1\Big)$ with Corollary $\ref{nonacorlin}$ for $h_1=h-h_2=h\Big(1-O\Big(\dfrac{h}{\epsilon^\nu}\Big)\Big)$ and $\nu=1-\dfrac{q-r}{2}$. The choice of $s\leq q <r+2$ makes so that the inequality \eqref{easier} holds for any $\nu>0$.\\
When $0<s<\dfrac{1}{3}$, the inclusion $V^{s_1}\subset V^s$ if $s<s_1$ implies that
\begin{equation}\label{last}
\mathbb{P}\Big(\tau_{\mathcal{B}(h)}<T\Big)= \mathbb{P}\Big(\sup_{0\leq t\leq \min\{T,  \tau_{\mathcal{B}(h)}\}} \lvert\lvert \psi(\cdot,t) \rvert\rvert_{A^s}\geq h\Big)\leq \mathbb{P}\Big(\sup_{0\leq t\leq \min\{T,  \tau_{\mathcal{B}(h)}\}} \lvert\lvert \psi(\cdot,t) \rvert\rvert_{A^{s_1}}\geq h\Big)\;,
\end{equation}
for which the last term can be controlled as before by choosing $\dfrac{1}{3}<s_1<\dfrac{1}{2}$ . The inequality $\eqref{last}$ holds since we assume $g\geq1$ and therefore $\lambda_1>1$.
\end{proof}

\subsection{Moment estimates}
The inequality presented in Theorem \ref{error estimate} enables an estimation of the moments of $\tau_{\mathcal{B}(h)}$. The proof of the next corollary relies on the step described in \cite[Proposition 3.1.12]{berglund2006noise}.
\begin{cor}
For any $k\in\mathbb{N}\setminus\{0\}$, $0<s<\dfrac{1}{2}$, $\sigma \sqrt{q_*}\ll h$ and equivalent assumptions to Theorem \ref{error estimate}, the following estimate holds:
\begin{equation}\label{moment1}
\mathbb{E}\Big(\tau_{\mathcal{B}(h)}^k\Big)\geq 
\dfrac{1}{(k+1)}\Bigg(\dfrac{\epsilon}{c} \Big(\dfrac{h^2}{q_*\sigma^2}\Big)^{-\frac{1}{2}} \textnormal{exp}\Bigg\{\tilde\kappa\dfrac{h^2}{q_* \sigma^2}\Bigg\} \Bigg)^k
\end{equation}
for a constant $c=c_{\kappa,s}>0$ and $\tilde{\kappa}=\kappa\bigg(1-O\Big(\dfrac{h}{\epsilon^\nu}\Big)\bigg)$, given $h,\kappa,\nu>0$ from Theorem \ref{error estimate}.
\end{cor}

\begin{proof}
The $k^\text{th}-$moment can be estimated by
\begin{equation*}
\begin{split}
&\mathbb{E}\Big(\tau_{\mathcal{B}(h)}^k\Big)=\int_0^\infty k t^{k-1} \mathbb{P}\Big(\tau_{\mathcal{B}(h)}\geq t\Big)\txtd t
\geq\int_0^T k t^{k-1} \Big(1-\mathbb{P}\Big(\tau_{\mathcal{B}(h)}<t\Big)\Big)\txtd t\\
&\geq\int_0^T k t^{k-1} \Bigg(1-C_{\frac{h^2}{q_* \sigma^2}} (\kappa,t,\epsilon,s)\text{exp}\Bigg\{-\kappa\dfrac{h^2}{\sigma^2}\bigg(1-O\Big(\dfrac{h}{\epsilon^\nu}\Big)\bigg)\Bigg\}\Bigg)\txtd t
\end{split}
\end{equation*}
for any $T>0$. From the hypothesis we know that $C_{\frac{h^2}{q_* \sigma^2}}(\kappa,t,\epsilon,s)\sim\dfrac{t}{\epsilon}  \Big(\dfrac{h^2}{q_* \sigma^2}\Big)^{\frac{1}{2}}$ by construction, therefore there exists a $c>0$ such that $C_{\frac{h^2}{q_* \sigma^2}}(\kappa,t,\epsilon,s)<c\dfrac{t}{\epsilon}  \Big(\dfrac{h^2}{q_* \sigma^2}\Big)^{\frac{1}{2}}$  and
\begin{equation*}
\mathbb{E}\Big(\tau_{\mathcal{B}(h)}^k\Big)
\geq\int_0^T k t^{k-1} \Bigg(1-c\dfrac{t}{\epsilon}  \Big(\frac{h^2}{q_* \sigma^2}\Big)^{\frac{1}{2}} \text{exp}\Bigg\{-\tilde\kappa\dfrac{h^2}{q_* \sigma^2}\Bigg\}\Bigg)\txtd t
=T^k-\dfrac{kc}{(k+1)\epsilon}T^{k+1} \Big(\frac{h^2}{q_* \sigma^2}\Big)^{\frac{1}{2}} \text{exp}\Bigg\{-\tilde\kappa\dfrac{h^2}{q_* \sigma^2}\Bigg\}\;.
\end{equation*}
The last term on the right member can be optimized in $T$ at $T=\dfrac{\epsilon}{c} \Big( \dfrac{h^2}{q_*\sigma^2}\Big)^{-\frac{1}{2}} \text{exp}\Bigg\{\tilde\kappa\dfrac{h^2}{q_*\sigma^2}\Bigg\}$, from which the corollary is proven.
\end{proof}
The right-hand side in \eqref{moment1} depends on $g$ and $\alpha$ due to the presence of $c$ and $\tilde\kappa$. The nature of the dependence is not trivial but it is visible that it affects the bound also in the exponential function.

\section{Numerical simulations}

In order to cross-validate and visualize the results we have obtained, we use numerical simulations. We simulate \eqref{mainsyst} using a finite difference discretization and the semi-implicit Euler-Maruyama method (\cite[Chapter 10]{lord2014introduction}). In detail: \begin{itemize}
\item An integer $N\gg1$ has been chosen in order to study the interval $\mathcal{O}=[0,L]$ in $N+2$ points, each distant $h=\dfrac{L}{N+1}$ from the closest neighbouring point.
\item The time was approximated by studying an interval of length $T\geq 5000$ at the values $\{j\;dt\}_{j=0}^{nt}$ for $nt:=\dfrac{T}{dt}\in\mathbb{N}$ and $nt\gg0$. 
\item The Laplacian operator is simulated as, $\mathtt{\Delta}$, the tridiagonal $N\times N$ matrix with values $-\dfrac{2}{h^2}$ on main diagonal and $\dfrac{1}{h^2}$ on the first superdiagonal and subdiagonal. The operator $A_\alpha$ was then approximated with 
\begin{equation*}\mathtt{A}_\alpha:=\mathtt{\Delta}-\mathtt{g}+\alpha \mathtt{I},\end{equation*}
for $\mathtt{g}$ the diagonal $N\times N$ matrix with elements $\mathtt{g}_{n,n}=g(n\;h)$ for any $n\in\{1,...,N\}$ and $\mathtt{I}$ as the $N\times N$ identity matrix. 
\item The values assumed by the solution $u$ in $\mathring{\mathcal{O}}$ are approximated by the $N\times(nt+1)$ matrix $\mathtt{u}$. Set $j\in\{0,...,nt\}$, the $j^\text{th}$ column of $\mathtt{u}$ is labeled $\mathtt{u}_j$.
\item The noise is been studied through the following: \begin{itemize} 
\item the integer $0<M\ll N$, chosen in order to indicate the number of directions of interest in $\mathcal{H}$ on which the noise will be considered to have effect;
\item the $M\times N$ matrix $\mathtt{e}'$ with elements $\mathtt{e}'(m,n):=e'_m(n \; h)$;
\item the $M\times M$ orthonormal matrix $\mathtt{O}$ defined by $\mathtt{O}(j_1,j_2):=\langle b_{j_1},e_{j_2}'\rangle$;
\item the vertical vector $\mathtt{q}$ composed of the first $M$ eigenvalues of $Q$ ordered by index and $\mathtt{q}^{\frac{1}{2}}$, the element-wise squared root of $\mathtt{q}$.
\item the random vertical vectors $\mathtt{W}_j$ with $M$ elements generated by independent standard Gaussian distributions each $j^\textnormal{th}$ iteration.
\end{itemize}
\end{itemize}
Such constructions enable the approximation
\begin{equation}\label{app}
\mathtt{u}_{j+1}=(\mathtt{I}-\mathtt{A}_\alpha dt)^{-1}( \mathtt{u}_j-\mathtt{u}^3_j dt+ \sigma \mathtt{q}^{\frac{1}{2}} \mathtt{W}_j \mathtt{O} \mathtt{e}' \sqrt{dt})
\end{equation}
for any $j\in\{0,...,nt-1\}$. Figure \ref{fig:m=10_SPDE_cos(3x)+1} and Figure \ref{fig:m=10_SPDE_lin} show the resulting plots for $g(x)=cos(3x)+1$ and for $g(x)=\dfrac{x}{L}$ distinguishing the cases in which $\alpha$ is less or higher than $\lambda_1$.\footnote{$L=2\pi$, $N=200$, $T=10000$, $nt=100000$, $\sigma=0.05$, $M=D=10$.\\ For different values of $\alpha$, distinct random generated $\mathtt{O}$ and $\mathtt{q}$ are used. The maximum value in $\mathtt{q}$ is set to be $1$.\\
The initial value has been taken close to the null function because of its relevance in the theory.} It is visible in figures $(a)$ that for $\alpha<\lambda_1$ the solution remains close to the null function and assumes no persistent shape. Pictures $(b)$ display the change caused by the crossing of the bifurcation threshold. In particular, the fact that the solution jumps away from the null functions and remains close to an equilibrium. The perturbation generated by noise can then create jumps to other stable deterministic stationary solutions whose shape is defined by choice of $g$.

\begin{figure}[ht!]
    \centering
    \subfloat[$\alpha=1.15<\lambda_1$]{\begin{overpic}[scale=0.3]{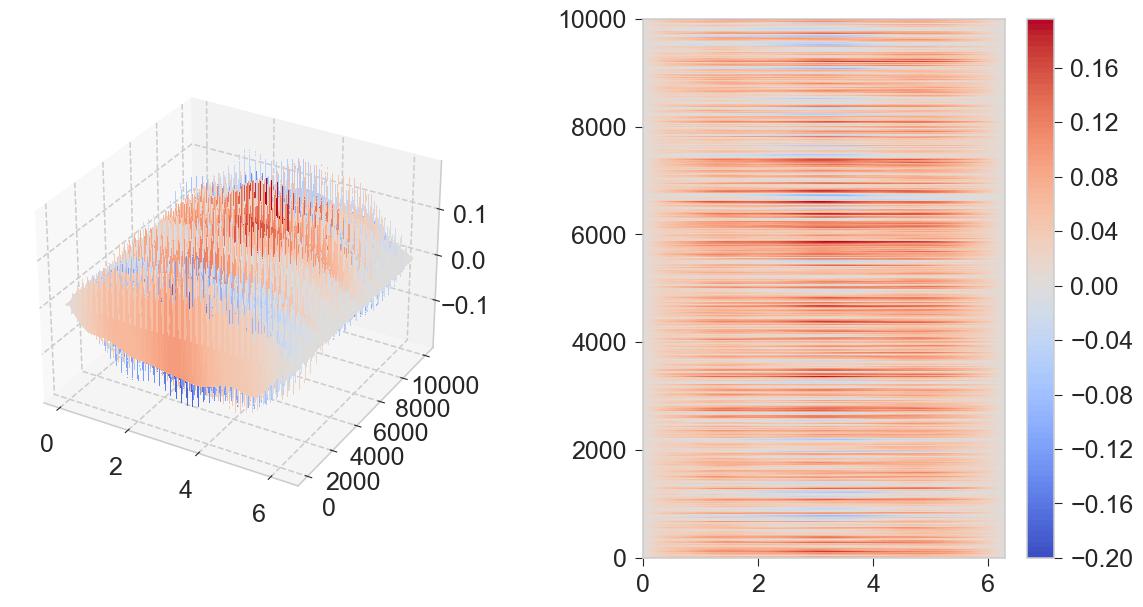}
    \put(100,60){\footnotesize{x}}
    \put(708,-25){\footnotesize{x}}
    \put(380,100){\footnotesize{t}}
    \put(515,270){\footnotesize{t}}
    \end{overpic}}
    \hspace{8mm}
    \subfloat[$\alpha=1.25>\lambda_1$]{\begin{overpic}[scale=0.3]{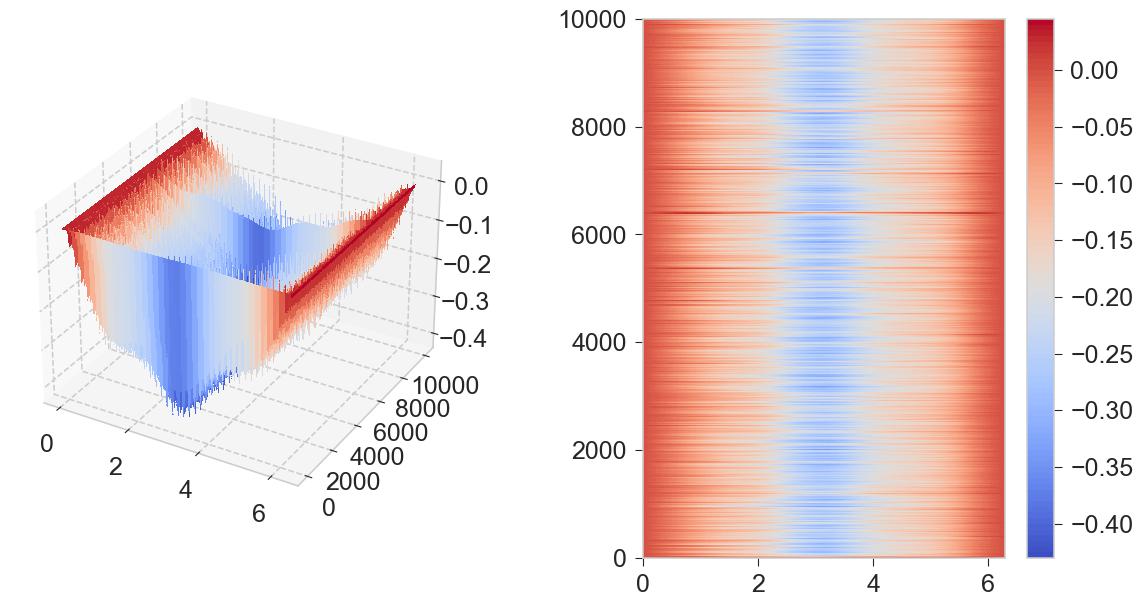}
    \put(100,60){\footnotesize{x}}
    \put(708,-25){\footnotesize{x}}
    \put(380,100){\footnotesize{t}}
    \put(515,270){\footnotesize{t}}
    \end{overpic}}\\
    \caption{Simulation of \eqref{mainsyst} with $g(x)=cos(3x)+1$ and $\lambda_1\approx 1.188$. Each subfigure presents a surf plot and a contour plot obtained with \eqref{app} with the same noise sample. On the left $\alpha$ is chosen before the bifurcation and on the right is taken beyond the bifurcation threshold. Metastable behaviour is visible on the second case.}
    \label{fig:m=10_SPDE_cos(3x)+1}
    \subfloat[$\alpha=0.65<\lambda_1$]{\begin{overpic}[scale=0.3]{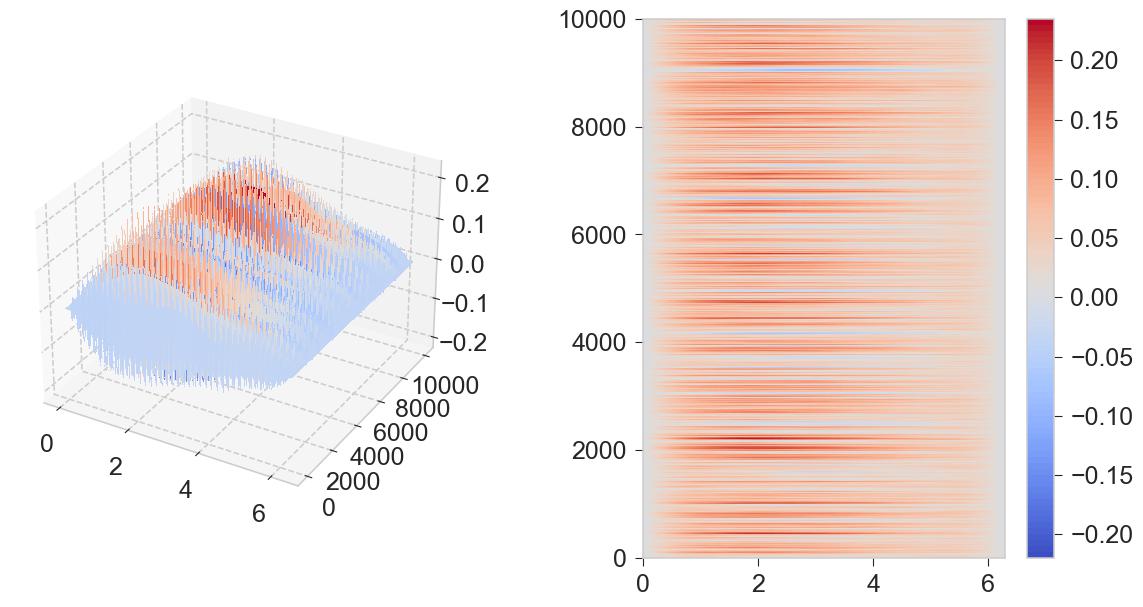}
    \put(100,60){\footnotesize{x}}
    \put(708,-25){\footnotesize{x}}
    \put(380,100){\footnotesize{t}}
    \put(515,270){\footnotesize{t}}
    \end{overpic}}
    \hspace{8mm}
    \subfloat[$\alpha=0.75>\lambda_1$]{\begin{overpic}[scale=0.3]{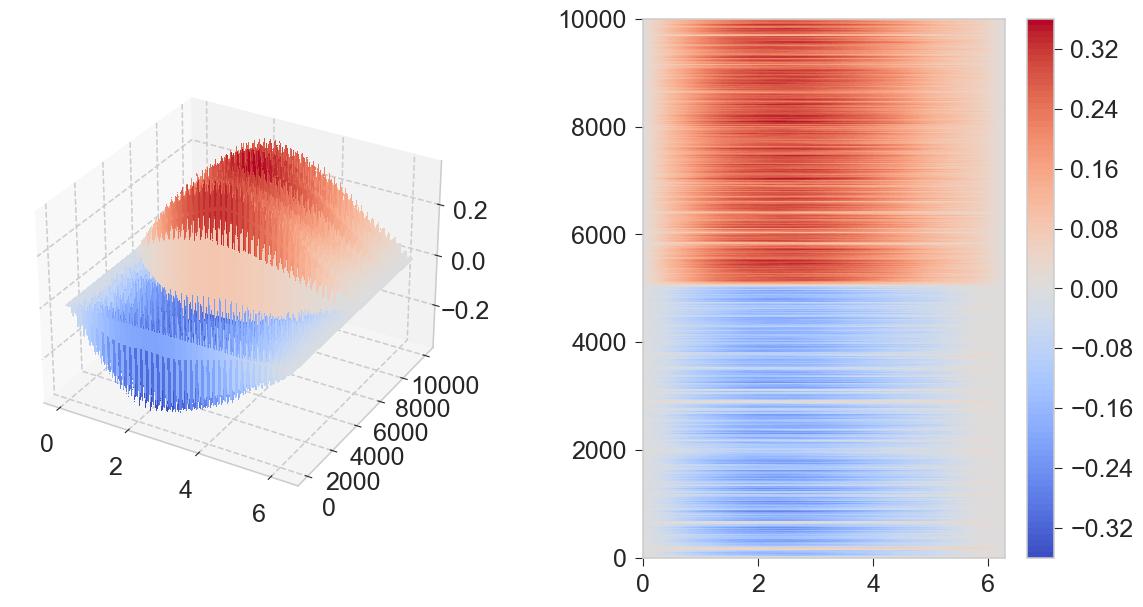}
    \put(100,60){\footnotesize{x}}
    \put(708,-25){\footnotesize{x}}
    \put(380,100){\footnotesize{t}}
    \put(515,270){\footnotesize{t}}
    \end{overpic}}
    \caption{Simulation of \eqref{mainsyst} with $g(x)=\dfrac{x}{L}$ and $\lambda_1\approx 0.708$. The choice of $\alpha$ and the corresponding behaviour is equivalent to the previous figure. The shape of $g$ influences the value $\lambda_1$ and the equilibria of the deterministic system \eqref{mainsyst}, i.e. for $\sigma=0$. Therefore it affects the behaviour of the solution and the bifurcation threshold.}
    \label{fig:m=10_SPDE_lin}
\end{figure}

\paragraph{Simulation of early warning signs.}\mbox{}\\
As previously stated, the linearization is more effective for $\alpha$ not close to $\lambda_1$. The rest of the section is devoted to compare numerically the early warning signs, meant as the left-hand side of \eqref{EWS1},\eqref{EWS2} and \eqref{EWS3}, applied to solutions of $\eqref{mainsyst}$ with the expected analytical result given by their application on \eqref{linsyst}. We show the effect on the early warning signs of the dissipative nonlinear term present in $\eqref{mainsyst}$ and how it hinders the divergences, in the limit $\alpha\longrightarrow\lambda_1$ from below, of the right-hand side of \eqref{EWS1},\eqref{EWS2} and \eqref{EWS3}.\\
The fact that the invariant measure of the linear system $\eqref{linsyst}$, $\mu$, is Gaussian with covariance operator $V_\infty$ and mean equal to the null function (\cite[Theorem 5.2]{DaPrato}) implies that
\begin{equation}\label{help}
\langle f_1,V_\infty f_2\rangle =\int_\mathcal{H} \langle f_1,w\rangle \langle f_2, w\rangle  \txtd \mu(w)\;,
\end{equation}
for all $f_1,f_2\in \mathcal{H}$. Therefore $\langle V_\infty e_{k_1},e_{k_2}\rangle$ from $\eqref{ews1}$ can be compared with
\begin{equation}\label{sim2}
\dfrac{1}{nt}\sum_{j=1}^{nt} \llangle\mathtt{u}_j, \mathtt{e}_{k_1}\rrangle \llangle\mathtt{u}_j, \mathtt{e}_{k_2}\rrangle-\bigg(\dfrac{1}{nt}\sum_{j_1=1}^{nt}\llangle\mathtt{u}_{j_1}, \mathtt{e}_{k_1}\rrangle\bigg) \bigg(\dfrac{1}{nt}\sum_{j_2=1}^{nt}\llangle\mathtt{u}_{j_2}, \mathtt{e}_{k_2}\rrangle\bigg)\;,
\end{equation}
that is the numerical covariance of the projection of the solution of \eqref{app} on the selected approximations of the eigenfunctions of $A_\alpha$. These are constructed as $\mathtt{e}_k(n):=e_k(n \; h)$ and obtained numerically through the "quantumstates" MATLAB function defined in \cite{driscoll2014chebfun}. The numerical scalar product $\llangle\cdot,\cdot\rrangle$ is defined by $\llangle\mathtt{v},\mathtt{w}\rrangle=h \sum_{n=1}^N \mathtt{v}(n) \mathtt{w}(n)$, for any $\mathtt{v},\mathtt{w}\in\mathbb{R}^N$.\footnote{The multiplication by $h$ is justified by the fact that the functions involved have value $0$ in $x=0$ and $x=L$.}\\
The plots in Figure \ref{fig:m=10_gen} illustrate for two examples of $g$ the results of \eqref{sim2} for the indexes $k=k_1=k_2\in\{1,2,3,4,5\}$ and $\alpha$ close to $\lambda_1$. They are then compared with 
\begin{equation}\label{blue}
\dfrac{\sigma^2}{2(\lambda_1-\alpha)}\sum_{j=1}^M \mathtt{q} (j) \llangle\mathtt{e}_1,\mathtt{b}_j\rrangle^2\;,
\end{equation}
 which is the numerical approximation of the right-hand side of $\eqref{ews1}$ on $k=1$, with $\{\mathtt{b}_j\}_{j=1}^M$ the row vectors of $\mathtt{O}\mathtt{e}'$ meant to replicate the eigenfunctions of $Q$.\footnote{$L=2\pi$, $N=100$, $T=5000$, $nt=100000$, $\sigma=0.01$, $M=D=10$. The two plots show results for different $\mathtt{O}$ and $\mathtt{q}$ which are randomly generated as previously described.}\\
It is clear from Figure \ref{fig:m=10_gen} that the dissipativity given by the nonlinear term in the system \eqref{mainsyst} hinders the variance of the system for $\alpha$ close to $\lambda_1$ and that the difference between early warning sign on \eqref{linsyst} and the average of \eqref{sim2} with $k_1=k_2=1$ on solutions given by \eqref{app} with different noise samples grows with $\alpha$. For $\alpha$ distant from $\lambda_1$ the values of such results is close and the behaviour of the plots is similar.

\begin{figure}[ht!]
    \centering
    \subfloat[$g(x)=cos(3x)+1$, $\lambda_1\approx1.188$]{
    \begin{overpic}[scale=0.45]{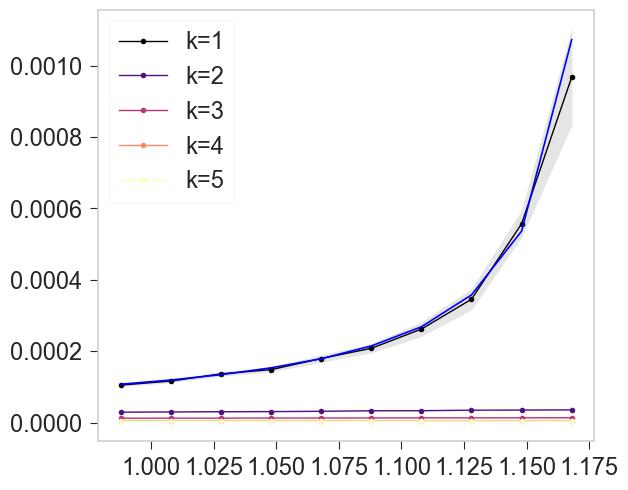}
    \put(550,-25){\footnotesize{$\alpha$}}
    \put(-70,550){\footnotesize{\rotatebox{270}{$\langle e_k,V_\infty e_k\rangle$}}}
    \end{overpic}}
    \hspace{16mm}
    \subfloat[$g(x)=\dfrac{x}{L}$, $\lambda_1\approx0.708$]{
    \begin{overpic}[scale=0.45]{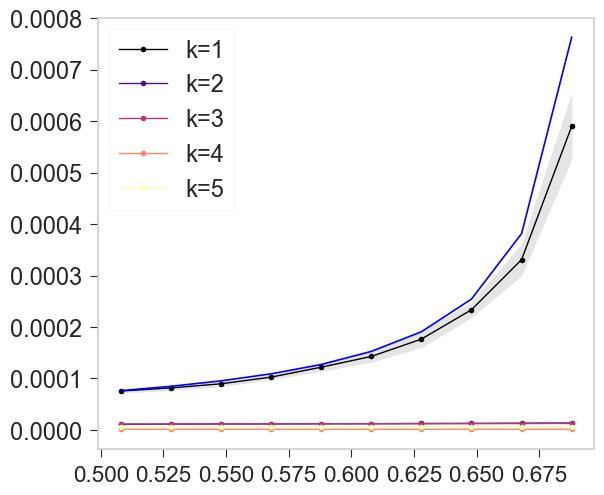}
    \put(550,-25){\footnotesize{$\alpha$}}
    \put(-70,550){\footnotesize{\rotatebox{270}{$\langle e_k,V_\infty e_k\rangle$}}}
    \end{overpic}}
    \caption{The results of \eqref{sim2} applied on the matrix $\mathtt{u}$, matrix obtained through the iteration of \eqref{app}, are displayed in the picture for $k=k_1=k_2\in\{1,2,3,4,5\}$ and different choices of $g$. The dots indicate the mean of such values obtained from $10$ simulations with same parameters and initial conditions, but generated with different noise samples. The grey area has width equal to the double of the recorded numerical standard deviation and it is centered on the mean results. The blue line displays the result $\eqref{blue}$ for the linear system. For $\alpha$ distant from $\lambda_1$ the black and blue line show similar values.}
    \label{fig:m=10_gen}
\end{figure}

A similar method can also be applied to replicate the early warning sign in $\eqref{ews3}$. We can choose a function $f_0\in L^\infty(\mathcal{H})$ such that $f_0(0)=f_0(L)=0$ and for which there exist the integer $0<p<N+2$ and $x_0=ph\in\mathcal{O}$ that satisfy 
\begin{equation}\label{ews2a}
\langle V_\infty f_0,f_0\rangle \approx\sigma^2\sum_{j_1=1}^\infty \sum_{j_2=1}^\infty \dfrac{ e_{j_1}(x_0) e_{j_2}(x_0)}{\lambda_{j_1}+\lambda_{j_2}-2\alpha}\bigg(\sum_{n=1}^\infty q_n \langle e_{j_1},b_n\rangle \langle e_{j_2}, b_n\rangle \bigg)\;,
\end{equation}
by Corollary \ref{EWS3}. From Proposition \ref{rootsprop}, one expects a hyperbolic-function divergence when $\alpha$ reaches $\lambda_1$ from below. Figures \ref{fig:m=10_p_cos(3x)+1} and \ref{fig:m=10_p_lin} compare $\langle V_\infty f_0,f_0\rangle $ for different values of $\alpha<\lambda_1$ 
in the form 
\begin{equation}\label{sim3}
\tilde{V}(p):=\dfrac{1}{nt}\sum_{j=1}^{nt} \mathtt{u}_j(p) \mathtt{u}_j(p)-\bigg(\dfrac{1}{nt}\sum_{j_1=1}^{nt}\mathtt{u}_{j_1}(p)\bigg) \bigg(\dfrac{1}{nt}\sum_{j_2=1}^{nt}\mathtt{u}_{j_2}(p)\bigg)\;,
\end{equation}
shown as black dots, and the approximation of the right-hand side in equation $\eqref{ews2a}$
\begin{equation}\label{blue2}
\sigma^2\sum_{j_1=1}^{M_1} \sum_{j_2=1}^{M_1} \dfrac{ \mathtt{e}_{j_1}(p) \mathtt{e}_{j_2}(p)}{\lambda_{j_1}+\lambda_{j_2}-2\alpha}\bigg(\sum_{n=1}^M \mathtt{q}(n) \llangle \mathtt{e}_{j_1},\mathtt{b}_n\rrangle \llangle \mathtt{e}_{j_2}, \mathtt{b}_n\rrangle \bigg)\;,
\end{equation}
in blue for a certain integer $0<M_1\ll N$.\footnote{Such truncation index is taken as $M_1=30$ in the figures.} The subplots in Figure \ref{fig:m=10_p_cos(3x)+1} and in Figure \ref{fig:m=10_p_lin} are given by the same simulations as in Figure \ref{fig:m=10_gen} $(a)$ and \ref{fig:m=10_gen} $(b)$ respectively. Each dot in the subplot is the average of \eqref{sim3} obtained from $10$ simulations which differ only by the sample for the noise taken. The numerical standard deviation is represented by the grey area.\footnote{The special case of cylindrical Wiener process, $Q=I$ for $I$ meant as the identity operator on $\mathcal{H}$, is particularly easy to simulate because of the freedom in the choice of the eigenfunctions of $Q$. Thus we could assume $\sum_{n=1}^\infty \langle e_{j_1},b_n\rangle \langle e_{j_2}, b_n\rangle =\delta_{j_1}^{j_2}$, the Kronecker delta.}\\
\begin{figure}[ht!]
    \centering
    \subfloat[$p=20$]{\begin{overpic}[scale=0.33]{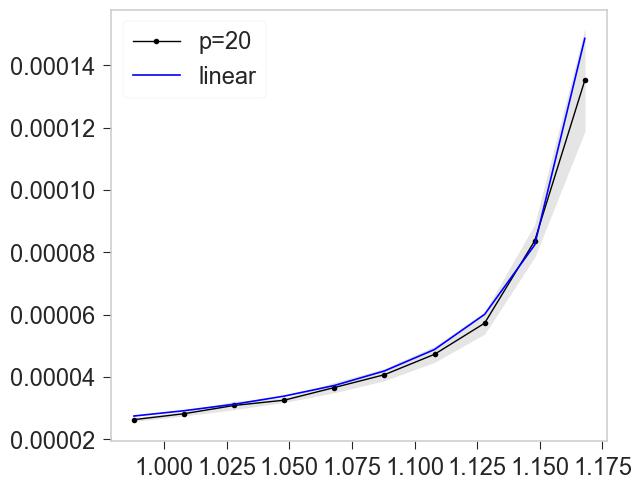}
    \put(550,-50){\footnotesize{$\alpha$}}
    \put(-100,500){\footnotesize{\rotatebox{270}{$\tilde{V}(p)$}}}
    \end{overpic}}
    \hspace{5mm}
    \subfloat[$p=50$]{\begin{overpic}[scale=0.33]{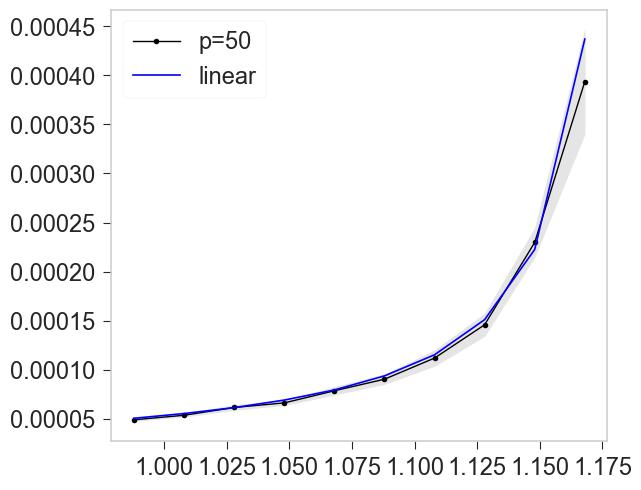}
    \put(550,-50){\footnotesize{$\alpha$}}
    \put(-100,500){\footnotesize{\rotatebox{270}{$\tilde{V}(p)$}}}
    \end{overpic}}
    \hspace{5mm}
    \subfloat[$p=70$]{\begin{overpic}[scale=0.33]{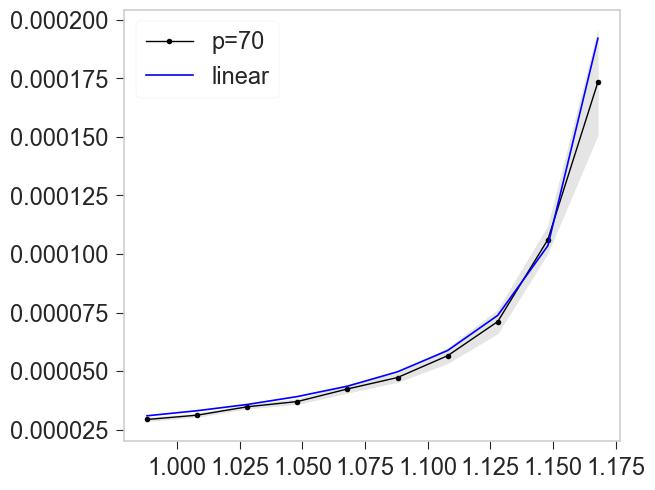}
    \put(550,-50){\footnotesize{$\alpha$}}
    \put(-100,500){\footnotesize{\rotatebox{270}{$\tilde{V}(p)$}}}
    \end{overpic}}\\
    \caption{Simulations of $\tilde{V}(p)$ obtained from $10$ sample solutions of \eqref{mainsyst} with $g(x)=cos(3x)+1$ and $\lambda_1\approx 1.188$ simulated with \eqref{app}. The black dots indicate the mean results and the width of the grey area corresponds to the double of the standard deviation for the relative $\alpha$. It is clear that the early warning sign is close to the expected result for the linear case until a neighbourhood of the bifurcation threshold on which the nonlinear dissipative term avoids the divergence. The space discretization is achieved taking $N=100$ internal points of $\mathcal{O}$ into account.}
    \label{fig:m=10_p_cos(3x)+1}
    \subfloat[$p=20$]{\begin{overpic}[scale=0.33]{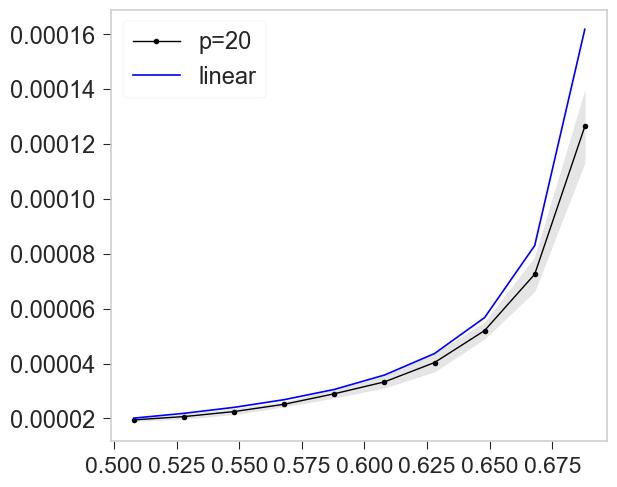}
    \put(550,-50){\footnotesize{$\alpha$}}
    \put(-100,500){\footnotesize{\rotatebox{270}{$\tilde{V}(p)$}}}
    \end{overpic}}
    \hspace{5mm}
    \subfloat[$p=50$]{\begin{overpic}[scale=0.33]{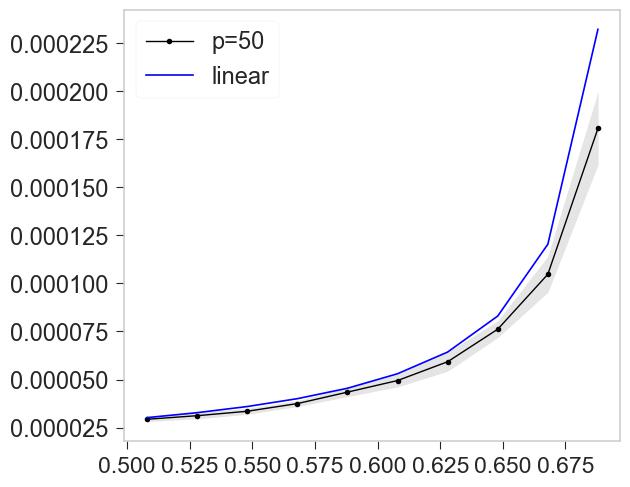}
    \put(550,-50){\footnotesize{$\alpha$}}
    \put(-100,500){\footnotesize{\rotatebox{270}{$\tilde{V}(p)$}}}
    \end{overpic}}
    \hspace{5mm}
    \subfloat[$p=70$]{\begin{overpic}[scale=0.33]{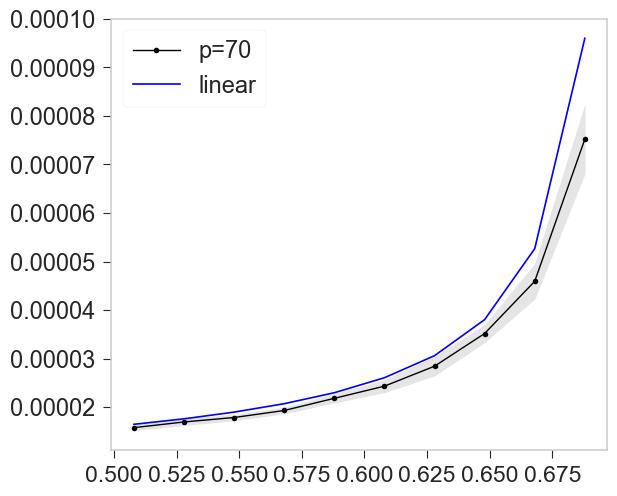}
    \put(550,-50){\footnotesize{$\alpha$}}
    \put(-100,500){\footnotesize{\rotatebox{270}{$\tilde{V}(p)$}}}
    \end{overpic}}\\
    \caption{Simulations of $\tilde{V}(p)$ obtained from $10$ sample solutions of \eqref{mainsyst} with $g(x)=\dfrac{x}{L}$ and $\lambda_1\approx 0.708$. The results are similar to the previous figure but the choice of $g$ appears to have influence over the difference of the expected early warning sign applied on the solution of the linear system \eqref{linsyst} and its simulation applied on \eqref{mainsyst}.}
    \label{fig:m=10_p_lin}
\end{figure}

As in Figure $\ref{fig:m=10_gen}$, the early warning sign \eqref{sim3} on the linear system assumes higher values than the average numerical variance obtained on projections on chosen spaces \eqref{blue2} of solutions of the nonlinear system. The difference in the two results is more evident when the dissipative nonlinear term is more relevant, which is close to the bifurcation. For such values of $\alpha$ it seems also evident that the standard deviation on the simulations is wider. The cause of it is clear from Theorem \ref{tm1} and Theorem \ref{tm2}.

\section*{Conclusion and outlook}

In this paper we have studied properties of a generalization of a Chafee-Infante type PDE with Dirichlet boundary conditions on an interval, by introducing a component heterogeneous in space. We have proven that it has a local supercritical pitchfork bifurcations from the trivial branch of zero solutions, equivalently to the original homogeneous PDE. We have then considered the effects of the crossing of the first bifurcation when the system is perturbed by additive noise. We have shown the existence of a global attractor regardless the crossing of the bifurcation threshold and found the rate at which the highest possible value of the FTLE approaches $0$. It is then a crucial observation that this rate of the FTLE approaching zero is precisely the inverse to the order of divergence found in the early warning sign given by the variance over infinite time of the linearized system along almost any direction on $\mathcal{H}$ when approaching the bifurcation threshold from below. An equivalent rate was then proven when studying the variance for any specific internal point of the interval. Note carefully that the consistent rates for FTLE scaling and the corresponding inverse covariance operator and pointwise-variance scaling are obtained by completely different proof techniques. Yet, one can clearly extract from the proofs that the linearized variational and linearized covariance equations provide the explicit rates, so a natural conjecture is that the same principle of common scaling laws will also apply to a wide variety of other bifurcations. In summary, here we have already given a complete picture of a very general class of SPDEs with a additive noise when the deterministic PDE exhibits a classical single-eigenvalue crossing bifurcation point from a trivial branch.  \\

In order to study the reliability of the early warning sign on the nonlinear system, the first exit-time from a small neighbourhood of the deterministic solution was studied by obtaining an upper bound of its distribution function. Hence, lower bounds of the moments have been derived. Lastly, numerical simulations have shown the consistency of the behaviour of the early warning sign applied on the solutions of the linear and nonlinear system for sufficient distance from the threshold.\\

We highlight we have obtained much freedom on the choice of the covariance operator $Q$, that defines the noise. In particular, almost any assumed set of eigenfunctions of $Q$ works. However, the eigenvalues are all assumed to be positive, in order to ensure a definition for the invariant measure. Our results hold also for weaker assumptions and a broader choice of operators. For instance, the requirements for the existence of the early warning signs of the linear problem hold for diagonalizable operators whose eigenfunctions synchronize to those of $Q$, as described in \eqref{sync}, and with eigenvalues diverging with a sufficiently fast polynomial rate. Whereas a lower bound for the moments of the first exit-time of the nonlinear solution has been found, an estimate for the moments of the distance between the solutions of the linear and nonlinear problem remains an open problem but our numerical results only show a slow divergence very close to the bifurcation point, which is a well-known phenomenon for finite-dimensional bifurcation problems with noise.

\section*{Acknowledgments}
This project has received funding from the European Union’s Horizon 2020 research and innovation programme under Grant Agreement 956170. The authors want to thank Maximilian Engel and Alexandra Neamtu for helpful discussions.

\appendix

\section{Appendix: Properties of the Schr\"odinger operator}

The Schr\"odinger operator $A:=\Delta-g$ inherits important properties from the Laplacian operator, assumed the Dirichlet conditions to be satisfied, and $g\in L^\infty(\mathcal{O})$ almost everywhere positive. Firstly, the system

\begin{equation}\label{eq:1}\begin{cases}
A u =- w\\
\left.u(\cdot,t)\right|_{\partial\mathcal{O}}=0\;\;\;\;,\;\;\;\;\forall t\geq0\;
\end{cases}\end{equation}
admits unique weak solutions in $\mathcal{V}$ for any $w\in \mathcal{H}$. In fact they satisfy
\begin{equation*} \langle (-A)u,v\rangle =\langle \nabla u, \nabla v\rangle +\langle u, g v\rangle =\langle w, v\rangle \;\;\;,\;\;\forall v \in \mathcal{V}\;\;.\end{equation*}
By Poincarè Inequality there exist $c,c'>0$ so that
\begin{equation*}\begin{split}
&\lvert\lvert v\rvert\rvert_\mathcal{V}^2=-\langle v,\Delta v\rangle \leq -\langle v,\Delta v\rangle +\langle v,gv\rangle =-\langle v,Av\rangle =: \lvert\lvert v\rvert\rvert_A\\
&\leq \lvert\lvert v\rvert\rvert_\mathcal{V}^2+\lvert\lvert g\rvert\rvert_\infty \lvert\lvert v\rvert\rvert_\mathcal{H}^2 \leq c \Big( \lvert\lvert v\rvert\rvert_\mathcal{V}^2+ \lvert\lvert v\rvert\rvert_\mathcal{H}^2 \Big) \leq c' \lvert\lvert v\rvert\rvert_\mathcal{V} \end{split}\end{equation*}
for any $v \in \mathcal{V}$. Therefore the spaces $\mathcal{V}=\mathcal{D}((-\Delta)^\frac{1}{2})$ and $\mathcal{D}((-A)^\frac{1}{2})$ are the same in the sense that the norms on which they are defined are equivalent.\\
The existence and uniqueness of the weak solutions of \eqref{eq:1} is implied by the fact that the map $ v\mapsto \langle w,v\rangle =\langle (-A)u,v\rangle$ is continuous in $\mathcal{V}$ with norm $\lvert\lvert\cdot\rvert\rvert_A$:

\begin{equation}\label{eq:2} \lvert\langle w,v\rangle \rvert \leq \lvert\lvert w \rvert\rvert_\mathcal{H} \; \lvert\lvert v \rvert\rvert_\mathcal{H} \leq c \lvert\lvert w \rvert\rvert_\mathcal{H} \; \lvert\lvert v \rvert\rvert_\mathcal{V}  \leq c \lvert\lvert w \rvert\rvert_\mathcal{H} \; \lvert\lvert v \rvert\rvert_A, \end{equation}
for a constant $c>0$. We can then use Riesz Theorem to assert that $\exists! u \in \mathcal{V}$ so that $\langle \nabla u, \nabla v\rangle +\langle u,g v\rangle =\langle w,v\rangle $. Defining the inverse of $-A$, the operator $(-A)^{-1}:\mathcal{H}\longrightarrow \mathcal{V}$, we have shown that the system \eqref{eq:1} admits an unique solution $(-A)^{-1} w=u\in\mathcal{V}$ for any $w\in \mathcal{H}$. From Rellich-Kondrachov Theorem and the fact that $(-A)^{-1}:\mathcal{H}\longrightarrow \mathcal{H}$ is a self-adjoint compact operator, its eigenfunctions form a basis in $\mathcal{H}$. Since $(-A)^{-1}$ is the inverse of $A$, they share the same eigenfunctions $\{e_k\}_{k\in\mathbb{N}\setminus\{0\}}$. We label as $\eta_k$ the eigenvalue of $(-A)^{-1}$ corresponding to the eigenfunction $e_k$ for any $k\in\mathbb{N}\setminus\{0\}$ and we set $\{\eta_k\}_{k\in\mathbb{N}\setminus\{0\}}$ to be a decreasing sequence.\\
From
\begin{equation*}\langle \nabla (-A)^{-1} w, \nabla v\rangle +\langle (-A)^{-1}w,gv\rangle =\langle w,v\rangle \;\;, \end{equation*}
by taking $v,w=e_k \in \mathcal{V}$ for any $k\in\mathbb{N}\setminus\{0\}$ we can show that

\begin{equation*}\eta_k \Big(\langle \nabla e_k,\nabla e_k\rangle +\langle e_k,g\; e_k\rangle \Big) =\langle e_k,e_k\rangle \;\; \end{equation*}
and therefore $\eta_k$ is strictly positive. By construction we have proven that the eigenvalues of $-A$, the family $\{\lambda_k\}_{k\in\mathbb{N}\setminus\{0\}}$, are also strictly positive.
\medskip

A deep description on the asymptotic behaviour of the eigenfunctions and eigenvalues of $-A$ can be found in \cite[Theorem 2.4]{poschel1987inverse}. In particular
\begin{equation}\label{sync}
\Big\lvert \lambda_k-\lambda_k'-\int_0^L g(x) \txtd x\Big\rvert^2+\Big\lvert\Big\lvert e_k-e_k' \Big\rvert\Big\rvert_\infty+\dfrac{1}{k}\Big\lvert\Big\lvert \dfrac{d}{dx}e_k-\dfrac{d}{dx}e_k' \Big\rvert\Big\rvert_\infty=O\Big(\dfrac{1}{k}\Big)
\end{equation}
indicates that for high index $k$ the influence of the heterogeneity on space induced by $g$ on the respective mode becomes less relevant and $A$ behaves similarly to $\Delta-\int_0^L g(x) \txtd x$ along such direction.\\
Another important property of $A$ is $\lambda_k<\lambda_{k+1}$ for $k\in\mathbb{N}\setminus\{0\}$, i.e. the eigenvalues are simple (\cite[Theorem 2.2]{poschel1987inverse}).

\section{Appendix: Operator Theory}
In the current appendix we make use of methods and results by \cite{DaPrato} and definitions from \cite{edmunds2018elliptic}. Its goal is to show the continuity in $\mathcal{V}$ of the solution of \eqref{mainsyst}.
\begin{deff}
Let $X$ be a Banach space. A \textbf{$\mathcal{C}_0$-semigroup} on $X$ is defined as a family of operators $S=\{S(t):t\geq0\}$, which are assumed continuous from $X$ on itself, such that
\begin{itemize}
\item $S(t_1)S(t_2)=S(t_1+t_2)$ for all $t_1,t_2\geq 0$,
\item $S(0)=I$,
\item for each $\phi\in X$, $S(\cdot)\phi:[0,+\infty)\rightarrow X$ is continuous.
\end{itemize}
Additionally, if $\lvert\lvert S(t)\rvert\rvert_{\mathcal{L}(X)}\leq 1$ for all $t\geq 0$ it is called a \textbf{$\mathcal{C}_0$-contraction semigroup} on $X$.\\
The \textbf{infinitesimal generator} of a $\mathcal{C}_0$-semigroup on $X$ is an operator $A:\mathcal{D}(A)\rightarrow X$ such that
\begin{equation*} A=\underset{t\rightarrow0}{\lim} \dfrac{ S(t) \phi -\phi}{t}\end{equation*}
for all $\phi\in \mathcal{D}(A)$.
\end{deff}
The Schrödinger operator $A=\Delta-g$ generates a $\mathcal{C}_0$-semigroup on $\mathcal{H}$ due to the following theorem.
\begin{thm}[Hille-Yosida]\label{hille-yosida}
A map $A:\mathcal{D}(A)\rightarrow X$, for which $\mathcal{D}(A)\subset X$, is the generator of a $\mathcal{C}_0$-contraction semigroup on the Banach space $X$ if and only if $A$ is closed, densely defined and
\begin{equation*} \lvert\lvert(\lambda I-A)^{-1}\rvert\rvert_{\mathcal{L}(X)}\leq \dfrac{1}{\lambda}\end{equation*}
for all $\lambda>0$.
\end{thm}
The following lemma is a generalization of Lemma $5.19$  by \cite{DaPrato} and requires weaker assumptions than the ones associated to the covariance operator $Q$ used in this paper, i.e. properties \ref{firstproperty}, \ref{secondproperty} and \ref{thirdproperty}.
\begin{lm}\label{contlm}
Set $D\in\mathbb{N}\setminus\{0\}$ and $D'\in\mathbb{N}$. Assume that there exists a dual indexed sequence $\{\rho_{j_1}^{j_2}\}_{j_1,j_2\in\mathbb{N}\setminus\{0\}}$ in $[0,1]$, so that the eigenfunctions, $\{b_j\}_{j\in\mathbb{N}\setminus\{0\}}$, and the eigenvalues, $\{q_j\}_{j\in\mathbb{N}\setminus\{0\}}$, of $Q$ satisfy the following properties:
\begin{enumerate}[label=\textnormal{(B\arabic*)}]
    \item \label{propertyB1} for all $0<j_2\leq D$, $b_{j_2}=\sum_{j_1\leq D} \rho_{j_1}^{j_2} e_{j_1}'$,
    \item \label{propertyB2} for all $D<j_2$, $b_{j_2}=\sum_{\{\lvert j_1-j_2\rvert\leq D', j_1>D\}} \rho_{j_1}^{j_2} e_{j_1}'$,
    \item \label{propertyB3} there exists $\gamma>0$ for which $\sum_{j=1}^\infty q_j \lambda_{j+D'}'^\gamma <+\infty$.
\end{enumerate}
Then, for
\begin{equation}\label{wdelta} w_\Delta(t,x):=\sum_{k=1}^\infty \sqrt{q_k} \sum_{n=1}^{\infty} \rho_n^k e_n'(x) \int_0^t \textnormal{e}^{-\lambda_n' (t-t_1)} \txtd \beta_k(t_1), \end{equation}
for $t\geq 0$, $x\in[0,L]$ and $\{\beta_k\}_{k\in\mathbb{N}\setminus\{0\}}$ a family of independent Wiener processes, the following estimations hold: there exists $C_1>0$ such that
\begin{equation}\label{ine1} \mathbb{E}\lvert \nabla( w_\Delta (t,x)-w_\Delta(t,x')) \rvert^2\leq C_1 \lvert x -x' \rvert^2 \end{equation}
\begin{equation}\label{ine2} \mathbb{E}\lvert \nabla( w_\Delta (t,x)-w_\Delta(t',x)) \rvert^2\leq C_1 \lvert t -t' \rvert^2 \end{equation}
for all $t,t'\geq 0$ and for all $x,x'\in [0,L]$.
\end{lm}
\begin{proof}
For simplicity we define $\{\lambda_{-j}'\}_{j\in\mathbb{N}}$ as $\lambda_{-j}:=\lambda_1$ for any $j\in\mathbb{N}$.\\
Before proving the well-posedness of $\nabla w_\Delta(t,x)$ in $L^2(\Omega,\mathcal{F},\mathbb{P})$ we introduce tools that are found in the proof. First is the existence of $C>0$ for which  $\lvert e_k'(x)\rvert\leq C$ and $\lvert \nabla e_k'(x)\rvert\leq C\sqrt{\lambda_k'}$ for all $x\in[0,L]$ and $k\in\mathbb{N}\setminus\{0\}$. Another useful tool is the fact that $\sum_{j=1}^\infty q_j \lambda_{j+D'}'^\gamma <+\infty$ implies, for monoticity and order of divergence of $\lambda_k'$ in respect to $k$, that $\sum_{j=1}^\infty q_j \lambda_{j+D'}'^\epsilon <+\infty$ for any $\epsilon\leq \gamma$.\\
We now prove that the outer series that defines $\nabla w_\Delta(t,x)$ converges in $L^2(\Omega,\mathcal{F},\mathbb{P})$. In order to achieve that we use $0\leq\rho_{j_1}^{j_2}\leq1$ for all $j_1,j_2$ and Ito's isometry:
\begin{equation*}\begin{split} &\mathbb{E} \Bigg\lvert \sum_{k=k_1}^{k_2} \sqrt{q_k} \sum_{n=1}^{\infty} \rho_n^k \nabla e_n'(x) \int_0^t \textnormal{e}^{-\lambda_n' (t-t_1)} \txtd \beta_k(t_1) \Bigg\rvert^2\\
&\leq C^2 \sum_{k=k_1}^{k_2} q_k \sum_{n,m} \lvert\rho_n^k \rho_m^k\rvert \sqrt{\lambda_n'} \sqrt{\lambda_m'} \int_0^t \textnormal{e}^{-(\lambda_n'+\lambda_m') (t-t_1)} \txtd t_1 =C^2 \sum_{k=k_1}^{k_2} q_k \sum_{n,m} \lvert\rho_n^k \rho_m^k\rvert \dfrac{\sqrt{\lambda_n'} \sqrt{\lambda_m'}}{\lambda_n'+\lambda_m'} \\
& \leq C^2 \sum_{k=k_1}^{D} q_k \sum_{n,m\leq D} \dfrac{\sqrt{\lambda_n'} \sqrt{\lambda_m'}}{\lambda_n'+\lambda_m'}+ C^2 \sum_{k=D+1}^{k_2} q_k \sum_{\lvert n-k\rvert\leq D',\lvert m-k\rvert\leq D'} \dfrac{\sqrt{\lambda_n'} \sqrt{\lambda_m'}}{\lambda_n'+\lambda_m'}\\
&\leq C^2 \sum_{k=k_1}^{D} q_k D^2 \dfrac{\lambda_{D}'}{2 \lambda_{1}'} + C^2 \sum_{k=D+1}^{k_2} q_k (2D'+1)^2 \dfrac{\lambda_{k+D'}'}{2 \lambda_{k-D'}'}
\leq \dfrac{C^2}{2 \lambda_1'} \max_{j\in \{D,1+2D'\}}\{\lambda_{j}'\;j^2\} \sum_{k=1}^{\infty} q_k <\infty,\end{split}\end{equation*}
because of the definition of $\{\lambda_j'\}_j$.\\
The proof for the estimate $\eqref{ine1}$ is the following:
\begin{equation*} \mathbb{E} \lvert \nabla( w_\Delta (t,x)-w_\Delta(t,x')) \rvert^2\leq \sum_{k=1}^\infty q_k \sum_{n,m} \lvert\rho_n^k \rho_m^k\rvert \sqrt{\lambda_n'} \sqrt{\lambda_m'} \lvert e_n''(x) - e_n''(x') \rvert \lvert e_m''(x) - e_m''(x') \rvert  \int_0^t \textnormal{e}^{-(\lambda_n'+\lambda_m') (t-t_1)} \txtd t_1 \end{equation*}
by defining $e_k''(x)=\sqrt{\dfrac{2}{L}}\text{cos}\Big(\dfrac{\pi k}{L} x\Big)$, for all $k\in\mathbb{N}\setminus\{0\}$ and $x\in\mathcal{O}$. Since $\lvert \nabla e_k''(x)\rvert\leq C\sqrt{\lambda_k'}$, from Lagrange's theorem we can obtain
\begin{equation}\label{qqq}\lvert e_k''(x)-e_k''(x')\rvert \leq 2^{1-2\epsilon} C \lambda_k'^\epsilon \lvert x-x'\rvert^{2\epsilon}\end{equation}
for all $x,x'\in[0,L]$, $\epsilon\leq \dfrac{1}{2}$ and $k\in\mathbb{N}\setminus\{0\}$. In particular \eqref{qqq} holds for $\epsilon=\dfrac{\gamma}{4}$.\\
Then
\begin{equation*}\begin{split} &\mathbb{E} \lvert \nabla( w_\Delta (t,x)-w_\Delta(t,x')) \rvert^2\leq 2^{2-\gamma} C^2 \sum_{k=1}^\infty q_k \sum_{n,m} \lvert\rho_n^k \rho_m^k\rvert \dfrac{\lambda_n'^{\frac{2+\gamma}{4}} \lambda_m'^{\frac{2+\gamma}{4}}}{\lambda_n'+\lambda_m'} \lvert x - x' \rvert^{\gamma}\\
& \leq 2^{2-\gamma} C^2 \bigg(D^2\sum_{k=1}^D q_k \dfrac{\lambda_{D}'^{1+\frac{\gamma}{2}}}{2 \lambda_1'} + (1+2D')^2\sum_{k=D+1}^\infty q_k \dfrac{\lambda_{k+D'}'^{1+\frac{\gamma}{2}}}{2 \lambda_{k-D'}'}\bigg) \lvert x - x' \rvert^{\gamma}\\
& \leq 2^{1-\gamma} \dfrac{C^2}{\lambda_1'} \max_{j\in \{D,1+2D'\}}\{\lambda_{j}'\;j^2\} \bigg(\sum_{k=1}^D q_k \lambda_{D}'^{\frac{\gamma}{2}} + \sum_{k=D+1}^\infty q_k \lambda_{k+D'}'^{\frac{\gamma}{2}}\bigg) \lvert x - x' \rvert^{\gamma}.
\end{split} \end{equation*}
For $t>t'>0$, the inequality $\eqref{ine2}$ is the result of 
\begin{equation*}\begin{split} &\mathbb{E} \lvert \nabla( w_\Delta (t,x)-w_\Delta(t',x)) \rvert^2\\
&\leq C^2 \sum_{k=1}^\infty q_k \sum_{n,m} \lvert\rho_n^k \rho_m^k\rvert \sqrt{\lambda_n'} \sqrt{\lambda_m'} \Bigg(\int_{t'}^t \textnormal{e}^{-(\lambda_n'+\lambda_m')(t-t_1)} \txtd t_1+ \int_0^{t'} \lvert \textnormal{e}^{-\lambda_n'(t-t_1)}- \textnormal{e}^{-\lambda_n'(t'-t_1)}\rvert \lvert \textnormal{e}^{-\lambda_m'(t-t_1)}- \textnormal{e}^{-\lambda_m'(t'-t_1)}\rvert \txtd t_1\Bigg).
\end{split}\end{equation*}
The first term is controlled as follows:
\begin{equation*} C^2 \sum_{k=1}^\infty q_k \sum_{n,m} \lvert\rho_n^k \rho_m^k\rvert \sqrt{\lambda_n'} \sqrt{\lambda_m'} \int_{t'}^t \textnormal{e}^{-(\lambda_n'+\lambda_m')(t-t_1)} \txtd t_1= 
C^2 \sum_{k=1}^\infty q_k \sum_{n,m} \lvert\rho_n^k \rho_m^k\rvert \dfrac{\sqrt{\lambda_n'} \sqrt{\lambda_m'}}{\lambda_n'+\lambda_m'} (1-\textnormal{e}^{-(\lambda_n'+\lambda_m')(t-t')})
\end{equation*}
and, by the fact that for all $\epsilon\in [0,1]$ there exists a $c_\epsilon>0$ that satisfies $\lvert \textnormal{e}^{-t}-\textnormal{e}^{-t'} \rvert\leq c_\epsilon \lvert t-t' \rvert^\epsilon$ for all $t,t'\geq 0$,
\begin{equation*}\begin{split}& C^2 \sum_{k=1}^\infty q_k \sum_{n,m} \lvert\rho_n^k \rho_m^k\rvert \sqrt{\lambda_n'} \sqrt{\lambda_m'} \int_{t'}^t \textnormal{e}^{-(\lambda_n'+\lambda_m')(t-t_1)} \txtd t_1
\leq c_\gamma C^2 \sum_{k=1}^\infty q_k \sum_{n,m} \lvert\rho_n^k \rho_m^k\rvert \dfrac{\sqrt{\lambda_n'} \sqrt{\lambda_m'}}{(\lambda_n'+\lambda_m')^{1-\gamma}} \lvert t-t' \rvert^\gamma\\
& \leq c_\gamma \dfrac{C^2}{2 \lambda_1'} \max_{j\in \{D,1+2D'\}}\{\lambda_{j}'\} \Bigg( \sum_{k=1}^D q_k \sum_{n,m\leq D} (\lambda_n'+\lambda_m')^\gamma+ \sum_{k=D+1}^\infty q_k \sum_{\lvert n-k\rvert\leq D',\lvert m-k\rvert\leq D'} (\lambda_n'+\lambda_m')^\gamma \Bigg) \lvert t-t' \rvert^\gamma\\
&\leq c_\gamma \dfrac{2^{\gamma-1} C^2}{\lambda_1'} \max_{j\in \{D,1+2D'\}}\{\lambda_{j}'\;j^2\} \Bigg( \sum_{k=1}^D q_k (\lambda_{D}')^\gamma+ \sum_{k=D+1}^\infty q_k (\lambda_{k+D'}')^\gamma \Bigg) \lvert t-t' \rvert^\gamma.
\end{split}\end{equation*}
The second term is studied similarly:
\begin{equation*}\begin{split} 
&C^2 \sum_{k=1}^\infty q_k \sum_{n,m} \lvert\rho_n^k \rho_m^k\rvert \sqrt{\lambda_n'} \sqrt{\lambda_m'}  \int_0^{t'} \big( \textnormal{e}^{-\lambda_n'(t-t_1)}- \textnormal{e}^{-\lambda_n'(t'-t_1)}\big) \big( \textnormal{e}^{-\lambda_m'(t-t_1)}- \textnormal{e}^{-\lambda_m'(t'-t_1)}\big) \txtd t_1\\
&\leq C^2 \sum_{k=1}^\infty q_k \sum_{n,m} \lvert\rho_n^k \rho_m^k\rvert \sqrt{\lambda_n'} \sqrt{\lambda_m'}  \\
&\times\int_0^{t'} \big( \textnormal{e}^{-\lambda_n'(t-t_1)-\lambda_m'(t-t_1)}- \textnormal{e}^{-\lambda_n'(t'-t_1)-\lambda_m'(t-t_1)}- \textnormal{e}^{-\lambda_n'(t-t_1)-\lambda_m'(t'-t_1)}+ \textnormal{e}^{-\lambda_n'(t'-t_1)-\lambda_m'(t'-t_1)}\big) \txtd t_1\\
&= C^2 \sum_{k=1}^\infty q_k \sum_{n,m} \lvert\rho_n^k \rho_m^k\rvert \dfrac{\sqrt{\lambda_n'} \sqrt{\lambda_m'}}{\lambda_n'+\lambda_m'} \\
&\times\Big( \textnormal{e}^{-\lambda_n' t+\lambda_n' t' -\lambda_m' t+ \lambda_m' t'} 
- \textnormal{e}^{-\lambda_n' t-\lambda_m' t}
- \textnormal{e}^{-\lambda_m' t+\lambda_m' t'}
+ \textnormal{e}^{-\lambda_n' t'-\lambda_m' t}
- \textnormal{e}^{-\lambda_n' t+\lambda_n' t'}
+ \textnormal{e}^{-\lambda_n' t-\lambda_m' t'}
+ 1
- \textnormal{e}^{-\lambda_n' t'-\lambda_m' t'} \Big)
\end{split}\end{equation*}
\begin{equation*}\begin{split}
&\leq C^2 \sum_{k=1}^\infty q_k \sum_{n,m} \lvert\rho_n^k \rho_m^k\rvert \dfrac{\sqrt{\lambda_n'} \sqrt{\lambda_m'}}{\lambda_n'+\lambda_m'}\\
&\times\Big( \lvert \textnormal{e}^{-\lambda_n' t+\lambda_n' t' -\lambda_m' t+ \lambda_m' t'} 
- \textnormal{e}^{-\lambda_m' t+\lambda_m' t'}\rvert
+\lvert 1 - \textnormal{e}^{-\lambda_n' t+\lambda_n' t'} \rvert
+\lvert \textnormal{e}^{-\lambda_n' t-\lambda_m' t'} - \textnormal{e}^{-\lambda_n' t-\lambda_m' t} \rvert
+ \lvert \textnormal{e}^{-\lambda_n' t'-\lambda_m' t} - \textnormal{e}^{-\lambda_n' t'-\lambda_m' t'} \rvert \Big)\\
&\leq 2 C^2 c_\gamma \sum_{k=1}^\infty q_k \sum_{n,m} \lvert\rho_n^k \rho_m^k\rvert \dfrac{\sqrt{\lambda_n'} \sqrt{\lambda_m'}}{\lambda_n'+\lambda_m'} 
(\lambda_n'^{\frac{\gamma}{2}}+\lambda_m'^{\frac{\gamma}{2}})
\lvert t-t' \rvert^\gamma\\
& \leq 2 C^2 c_\gamma \Bigg( \sum_{k=1}^D q_k \sum_{n,m\leq D} \dfrac{\sqrt{\lambda_n'} \sqrt{\lambda_m'}}{\lambda_n'+\lambda_m'} 
(\lambda_n'^{\frac{\gamma}{2}}+\lambda_m'^{\frac{\gamma}{2}}) + \sum_{k=D+1}^\infty q_k \sum_{\lvert n-k\rvert\leq D',\lvert m-k\rvert\leq D'} \dfrac{\sqrt{\lambda_n'} \sqrt{\lambda_m'}}{\lambda_n'+\lambda_m'} 
(\lambda_n'^{\frac{\gamma}{2}}+\lambda_m'^{\frac{\gamma}{2}}) \Bigg) \lvert t-t' \rvert^\gamma \\
& \leq c_\gamma \dfrac{C^2}{\lambda_1'} \max_{j\in \{D,1+2D'\}}\{\lambda_{j}'\;j^2\}  \Bigg( \sum_{k=1}^D q_k
\lambda_{D}'^{\frac{\gamma}{2}} + \sum_{k=D+1}^\infty q_k 
\lambda_{k+D'}'^{\frac{\gamma}{2}} \Bigg) \lvert t-t' \rvert^\gamma
\end{split}\end{equation*}
for which the symbol $\times$ is used as the product operation.
\end{proof}
Having proven the previous lemma, we can use \cite[Theorem $5.20$]{DaPrato} to state that $w_\Delta$, defined in \eqref{wdelta}, admits a version with continuous paths in $\mathcal{V}$. Through the fact that the Laplacian operator satisfies Theorem \ref{hille-yosida} we can imply with \cite[Theorem $5.27$]{DaPrato} that $w_A:[0,T]\times\mathcal{O}\longrightarrow L^2(\Omega,\mathcal{F},\mathbb{P})$, defined as 
\begin{equation} \label{wA} w_A(t,x):=\sum_{k=1}^\infty \sqrt{q_k} \sum_{n=1}^{\infty} \rho_n^k e_n(x) \int_0^t \textnormal{e}^{-\lambda_n' (t-t_1)} \txtd \beta_k(t_1),\quad \forall t>0\;\;\;\text{and}\;\;\;x\in\mathcal{O},\end{equation}
admits a version continuous in $\mathcal{V}$ as well. On such results, \cite[Theorem $7.13$]{DaPrato} states the existence and uniqueness of the mild solution of \eqref{mainsyst} in $\mathcal{C}([0,+\infty);\mathcal{V})$ if $u_0\in \mathcal{V}$. Finally, \cite[Theorem $7.16$]{DaPrato} leads to the $\mathbb{P}-a.s.$ existence of a unique mild solution of $\eqref{mainsyst}$
\begin{equation*} u\in L^2(\Omega\times(0,T);\mathcal{V})\cap L^2(\Omega;\mathcal{C}([0,T];\mathcal{H}))\end{equation*}
for all $T>0$.
\paragraph{Remark.}
In this paper we use Lemma \ref{contlm} for $D'=0$, sufficient to have great freedom on a finite number of eigenfunctions of $Q$.\\ \medskip
The lemma can be generalized in different aspects. The choice of the space on which the continuity of the solution of $\eqref{mainsyst}$ is wanted depends on $\gamma$. For instance, assuming $-1<\gamma<0$ the continuity of the paths of $w_\Delta$ in $\mathcal{H}$ can be proven in an equivalent manner. We note also that, tracking the same steps of the proof of Lemma \ref{contlm}, the continuity in $\mathcal{V}$ of $w_A:[0,T]\times\mathcal{O}\longrightarrow L^2(\Omega,\mathcal{F},\mathbb{P})$, defined in \eqref{wA},
can be proven in the case in which the properties \ref{propertyB1}, \ref{propertyB2} and \ref{propertyB3} were assumed in relation with the eigenfunctions and eigenvalues of $A$, instead of those of $\Delta$, i.e.
\begin{enumerate}[label=\textnormal{(B\arabic*')}]
    \item \label{propertyB1'} for all $0<j_2\leq D$, $b_{j_2}=\sum_{j_1\leq D} \rho_{j_1}^{j_2} e_{j_1}$,
    \item \label{propertyB2'} for all $D<j_2$, $b_{j_2}=\sum_{\{\lvert j_1-j_2\rvert\leq D', j_1>D\}} \rho_{j_1}^{j_2} e_{j_1}$,
    \item \label{propertyB3'} there exists $\gamma>0$ for which $\sum_{j=1}^\infty q_j \lambda_{j+D'}^\gamma <+\infty$.
\end{enumerate}
This is possible due to $\eqref{sync}$. \\
Lastly, assuming instead of the properties \ref{propertyB1} and \ref{propertyB2}
\begin{equation*}
    \lvert\lvert e_k-b_k\rvert\rvert_\mathcal{H}=O\Big(\dfrac{1}{k}\Big)
\end{equation*}
and $Q$ to be only trace classe, the continuity of $w_A$ in $\mathcal{V}^s$ can be shown for $0<s<1$.

\section{Appendix}
Set the Hilbert space $X$ with basis $\{\phi_j\}_{j\in\mathbb{N}\setminus\{0\}}$, scalar product $\langle \cdot,\cdot\rangle_X$ and $f_1,f_2\in X$. By definition of basis, we know that the sequences $\{f_n^k:=\sum_{j=1}^n \langle f_k, \phi_j\rangle_X \phi_j\}_{n\in\mathbb{N}\setminus\{0\}}$ converge strongly in the norm defined by the scalar product, $\lvert\lvert\cdot\rvert\rvert_X$, to $f_k=\sum_{j=1}^\infty \langle f_k, \phi_j\rangle_X \phi_j$, for $k\in\{1,2\}$. It is then well known and easy to prove the following consequence:
\begin{equation}\label{AC1}\begin{split}
&\lvert\langle f_n^1,f_n^2\rangle_X-\langle f_1,f_2\rangle_X\rvert\leq\lvert\langle f_n^1-f_1,f_n^2\rangle_X\rvert+\lvert\langle f_1,f_n^2-f_2\rangle_X\rvert\\
&\leq\lvert\lvert f_n^1-f_1 \rvert\rvert_X \; \lvert\lvert f_n^2 \rvert\rvert_X +\lvert\lvert f_n^2-f_2 \rvert\rvert_X \; \lvert\lvert f_1 \rvert\rvert_X\leq \max\Big\{\lvert\lvert f_1 \rvert\rvert_X,\lvert\lvert f_2 \rvert\rvert_X\Big\}\Big(\lvert\lvert f_n^1-f_1 \rvert\rvert_X+\lvert\lvert f_n^2-f_2 \rvert\rvert_X\Big)\underset{n\rightarrow\infty}{\longrightarrow}0\;.
\end{split}\end{equation}
By definition
\begin{equation}\label{AC2}\begin{split}
&\langle f_n^1,f_n^2\rangle_X=\Big\langle\sum_{j_1=1}^n \langle f_1,\phi_{j_1}\rangle_X \phi_{j_1}, \sum_{j_2=1}^n \langle f_2,\phi_{j_2}\rangle_X \phi_{j_2}\Big\rangle_X\\
&=\sum_{j_1=1}^n \sum_{j_2=1}^n \langle f_1,\phi_{j_1}\rangle_X \langle f_2,\phi_{j_2}\rangle_X \langle \phi_{j_1},\phi_{j_2}\rangle_X=\sum_{j=1}^n \langle f_1,\phi_j\rangle_X \langle f_2,\phi_j\rangle_X\;.
\end{split}\end{equation}
Combining $\eqref{AC1}$ and $\eqref{AC2}$ we obtain
\begin{equation*}
\langle f_1,f_2\rangle_X=\sum_{j=1}^\infty \langle f_1,\phi_j\rangle_X \langle f_2,\phi_j\rangle_X\;.
\end{equation*}

\newpage
\printbibliography

@article{CI,
  title={A bifurcation problem for a nonlinear partial differential equation of parabolic type},
  author={Chafee, Nathaniel and Infante, Ettore Ferrari},
  journal={Appl. Anal.},
  volume={4},
  number={1},
  pages={17--37},
  year={1974},
  publisher={Taylor \& Francis}
}

@ARTICLE{GnannKuehnPein,
   author = "M. Gnann and C. Kuehn and A. Pein",
   title = "Towards sample path estimates for fast-slow {SPDEs}",
   journal = "Euro. J. Appl. Math.",
	 volume = 30,
	 number = 5,
   pages = {1004--1024},
   year = 2019,
   }

@book{allee1949principles,
  title={Principles of animal ecology.},
  author={Allee, Warder Clyde and Park, Orlando and Emerson, Alfred E and Park, Thomas and Schmidt, Karl P and others},
  number={Edn 1},
  year={1949},
  publisher={WB Saundere Co. Ltd.}
}

@article{allen1972ground,
  title={Ground state structures in ordered binary alloys with second neighbor interactions},
  author={Allen, Samuel Miller and Cahn, John W},
  journal={Acta Mater.},
  volume={20},
  number={3},
  pages={423--433},
  year={1972},
  publisher={Elsevier}
}

@article{fitzhugh1955mathematical,
  title={Mathematical models of threshold phenomena in the nerve membrane},
  author={FitzHugh, Richard},
  journal={Bull. math. biophys.},
  volume={17},
  number={4},
  pages={257--278},
  year={1955},
  publisher={Springer}
}

@article{stommel1961thermohaline,
  title={Thermohaline convection with two stable regimes of flow},
  author={Stommel, Henry},
  journal={Tellus},
  volume={13},
  number={2},
  pages={224--230},
  year={1961},
  publisher={Wiley Online Library}
}

@book{Henry81,
  title={Geometric theory of semilinear parabolic equations},
  author={Henry, Daniel},
  volume={840},
  year={2006},
  publisher={Springer}
}

@article{CF,
  title={Additive noise destroys a pitchfork bifurcation},
  author={Crauel, Hans and Flandoli, Franco},
  journal={J. Dynam. Differential Equations},
  volume={10},
  number={2},
  pages={259--274},
  year={1998},
  publisher={Springer}
}

@article{Caraballo,
  title={The effect of noise on the Chafee-Infante equation: a nonlinear case study},
  author={Caraballo, Tom{\'a}s and Crauel, Hans and Langa, Jos{\'e} and Robinson, James},
  journal={Proc. Amer. Math. Soc.},
  volume={135},
  number={2},
  pages={373--382},
  year={2007}
}

@article{callaway,
  title={The dichotomy spectrum for random dynamical systems and pitchfork bifurcations with additive noise},
  author={Callaway, Mark and Doan, Thai Son and Lamb, Jeroen SW and Rasmussen, Martin},
  journal={Ann. Inst. Henri Poincar{\'e} Probab.},
  volume={53},
  number={4},
  pages={1548--1574},
  year={2017},
  organization={Institut Henri Poincar{\'e}}
}

@misc{https://doi.org/10.48550/arxiv.2108.11073,
  doi = {10.48550/ARXIV.2108.11073},
  
  url = {https://arxiv.org/abs/2108.11073},
  
  author = {Blumenthal, Alex and Engel, Maximilian and Neamtu, Alexandra},
  
  title = {On the pitchfork bifurcation for the Chafee-Infante equation with additive noise},
  
  publisher = {arXiv},
  
  year = {2021},
  
  copyright = {arXiv.org perpetual, non-exclusive license}
}

@book{DaPrato,
  title={Stochastic equations in infinite dimensions},
  author={Da Prato, Giuseppe and Zabczyk, Jerzy},
  year={2014},
  publisher={Cambridge university press}
}

@article{debussche,
  title={Hausdorff dimension of a random invariant set},
  author={Debussche, Arnaud},
  journal={J. Math. Pures Appl.},
  volume={77},
  number={10},
  pages={967--988},
  year={1998},
  publisher={Elsevier}
}

@article{chueshov,
  title={Non-random invariant sets for some systems of parabolic stochastic partial differential equations},
  author={Chueshov, ID and Vuillermot, P-A},
  journal={Stoch. Anal. Appl.},
  volume={22},
  number={6},
  pages={1421--1486},
  year={2004},
  publisher={Taylor \& Francis}
}

@article{twardowska,
  title={An approximation theorem of Wong-Zakai type for nonlinear stochastic partial differential equations},
  author={Twardowska, Krystyna},
  journal={Stoch. Anal. Appl.},
  volume={13},
  number={5},
  pages={601--626},
  year={1995},
  publisher={Taylor \& Francis}
}

@article{early,
  title={Early-warning signs for pattern-formation in stochastic partial differential equations},
  author={Gowda, Karna and Kuehn, Christian},
  journal={Commun. Nonlinear Sci. Numer. Simul.},
  volume={22},
  number={1-3},
  pages={55--69},
  year={2015},
  publisher={Elsevier}
}

@article{chueshov2001inertial,
  title={Inertial manifolds and forms for stochastically perturbed retarded semilinear parabolic equations},
  author={Chueshov, Igor D and Scheutzow, M},
  journal={J. Dynam. Differential Equations},
  volume={13},
  number={2},
  pages={355--380},
  year={2001},
  publisher={Springer}
}

@article{caraballo2000stability,
  title={Stability and random attractors for a reaction-diffusion equation with multiplicative noise},
  author={Caraballo, Tom{\'a}s and Langa, Jos{\'e} A and Robinson, James C},
  journal={Discrete Contin. Dyn. Syst.},
  volume={6},
  number={4},
  pages={875},
  year={2000},
  publisher={American Institute of Mathematical Sciences}
}

@book{cerrai,
  title={Second order PDE’s in finite and infinite dimension: a probabilistic approach},
  author={Cerrai, Sandra},
  year={2001},
  publisher={Springer}
}

@book{da1996ergodicity,
  title={Ergodicity for infinite dimensional systems},
  author={Da Prato, Giuseppe and Zabczyk, Jerzy and Zabczyk, J},
  volume={229},
  year={1996},
  publisher={Cambridge University Press}
}

@article{khas1960ergodic,
  title={Ergodic properties of recurrent diffusion processes and stabilization of the solution to the Cauchy problem for parabolic equations},
  author={{Khas’minskii}, Rafail Z},
  journal={Theory Probab. Appl.},
  volume={5},
  number={2},
  pages={179--196},
  year={1960},
  publisher={SIAM}
}

@article{crauel2000white,
  title={White noise eliminates instability},
  author={Crauel, Hans},
  journal={Arch. Math.},
  volume={75},
  number={6},
  pages={472--480},
  year={2000},
  publisher={Springer}
}

@article{agmon1962eigenfunctions,
  title={On the eigenfunctions and on the eigenvalues of general elliptic boundary value problems},
  author={Agmon, Shmuel},
  journal={Comm. Pure Appl. Math.},
  volume={15},
  number={2},
  pages={119--147},
  year={1962},
  publisher={Wiley Online Library}
}

@book{da2004functional,
  title={Functional analytic methods for evolution equations},
  author={Da Prato, Giuseppe and Kunstmann, Peer Christian and Lasiecka, Irena and Lunardi, Alessandra and Schnaubelt, Roland and Weis, Lutz},
  year={2004},
  publisher={Springer Science \& Business Media}
}

@book{edmunds2018elliptic,
  title={Elliptic differential operators and spectral analysis},
  author={Edmunds, David Eric and Evans, W Desmond},
  year={2018},
  publisher={Springer}
}

@book{debussche2013dynamics,
  title={The dynamics of nonlinear reaction-diffusion equations with small L{\'e}vy noise},
  author={Debussche, Arnaud and H{\"o}gele, Michael and Imkeller, Peter},
  volume={2085},
  year={2013},
  publisher={Springer}
}

@book{da2004kolmogorov,
  title={Kolmogorov equations for stochastic PDEs},
  author={Da Prato, Giuseppe},
  year={2004},
  publisher={Springer Science \& Business Media}
}

@book{lord2014introduction,
  title={An introduction to computational stochastic PDEs},
  author={Lord, Gabriel J and Powell, Catherine E and Shardlow, Tony},
  volume={50},
  year={2014},
  publisher={Cambridge University Press}
}

@article{crandall1971bifurcation,
  title={Bifurcation from simple eigenvalues},
  author={Crandall, Michael G and Rabinowitz, Paul H},
  journal={J. Funct. Anal.},
  volume={8},
  number={2},
  pages={321--340},
  year={1971},
  publisher={Elsevier}
}

@book{kuehn2019pde,
  title={PDE Dynamics: An Introduction},
  author={Kuehn, Christian},
  volume={23},
  year={2019},
  publisher={SIAM}
}

@book{poschel1987inverse,
  title={Inverse spectral theory},
  author={Poschel, Jurgen},
  year={1987},
  publisher={Academic Press}
}

@book{berglund2006noise,
  title={Noise-induced phenomena in slow-fast dynamical systems: a sample-paths approach},
  author={Berglund, Nils and Gentz, Barbara},
  year={2006},
  publisher={Springer Science \& Business Media}
}

@article{berglund2022stochastic,
  title={Stochastic resonance in stochastic PDEs},
  author={Berglund, Nils and Nader, Rita},
  journal={Stoch. Partial Differ. Equ.},
  pages={1--40},
  year={2022},
  publisher={Springer}
}

@article{berglund2002pathwise,
  title={Pathwise description of dynamic pitchfork bifurcations with additive noise},
  author={Berglund, Nils and Gentz, Barbara},
  journal={Probab. Theory Related Fields},
  volume={122},
  number={3},
  pages={341--388},
  year={2002},
  publisher={Springer}
}

@article{berglund2013sharp,
  title={Sharp estimates for metastable lifetimes in parabolic SPDEs: Kramers' law and beyond},
  author={Berglund, Nils and Gentz, Barbara},
  journal={Electron. J. Probab.},
  volume={18},
  pages={1--58},
  year={2013},
  publisher={Institute of Mathematical Statistics and Bernoulli Society}
}

@book{lunardi2018interpolation,
  title={Interpolation theory},
  author={Lunardi, Alessandra},
  volume={16},
  year={2018},
  publisher={Springer}
}

@article{arnold1995random,
  title={Random dynamical systems},
  author={Arnold, Ludwig},
  journal={Dyn. Syst.},
  pages={1--43},
  year={1995},
  publisher={Springer}
}

@article{kuehn2011mathematical,
  title={A mathematical framework for critical transitions: Bifurcations, fast--slow systems and stochastic dynamics},
  author={Kuehn, Christian},
  journal={Phys. D: Nonlinear Phenom.},
  volume={240},
  number={12},
  pages={1020--1035},
  year={2011},
  publisher={Elsevier}
}

@article{kuehn2013mathematical,
  title={A mathematical framework for critical transitions: normal forms, variance and applications},
  author={Kuehn, Christian},
  journal={J. Nonlinear Sci.},
  volume={23},
  number={3},
  pages={457--510},
  year={2013},
  publisher={Springer}
}

@article{kuehn2019scaling,
  title={Scaling laws and warning signs for bifurcations of SPDEs},
  author={Kuehn, Christian and Romano, Francesco},
  journal={European J. Appl. Math.},
  volume={30},
  number={5},
  pages={853--868},
  year={2019},
  publisher={Cambridge University Press}
}

@article{berglund2012hunting,
  title={Hunting French ducks in a noisy environment},
  author={Berglund, Nils and Gentz, Barbara and Kuehn, Christian},
  journal={J. Differ. Equ.},
  volume={252},
  number={9},
  pages={4786--4841},
  year={2012},
  publisher={Elsevier}
}

@article{berglund2015random,
  title={From random Poincar{\'e} maps to stochastic mixed-mode-oscillation patterns},
  author={Berglund, Nils and Gentz, Barbara and Kuehn, Christian},
  journal={J. Dynam. Differential Equations},
  volume={27},
  number={1},
  pages={83--136},
  year={2015},
  publisher={Springer}
}

@book{driscoll2014chebfun,
  title={Chebfun guide},
  author={Driscoll, Tobin A and Hale, Nicholas and Trefethen, Lloyd N},
  year={2014},
  publisher={Pafnuty Publications, Oxford}
}

\end{document}